\documentclass[aihp]{imsart}
\RequirePackage[colorlinks,citecolor=blue,urlcolor=blue]{hyperref}
\usepackage{amsmath,amsmath,amssymb,amsfonts,qtree}
\usepackage{amsthm}
\usepackage[american]{babel}
\usepackage{booktabs}
\usepackage{graphicx}
\usepackage{flafter}
\usepackage[section]{placeins}
\usepackage{float}
\usepackage{makecell}
\usepackage{comment}
\usepackage[mathscr]{eucal}
\usepackage{bbm}
\usepackage{todonotes}
\usepackage{enumerate}
\presetkeys{todonotes}{color=red!30}{linecolor=black!50}
\startlocaldefs

\def\eps{{\varepsilon}}

\def\Bbb E{\mathbb{E}}
\def\Bbb R{\mathbb{R}}
\newtheorem{corollary}{Corollary}[section]
\usepackage{mathtools}
\makeatletter \@addtoreset{equation}{section}

\makeatother

\newtheorem{lemma}{Lemma}[section]
\newtheorem{theorem}{Theorem}[section]
\newtheorem{proposition}{Proposition}[section]

\newtheorem{remark}{Remark}[section]

 \makeatletter
\def\@fnsymbol#1{\ensuremath{\ifcase#1\or * \or \mathsection \or ** \else\@ctrerr\fi}}
\makeatother

\font\tencmmib=cmmib10 \skewchar\tencmmib '60
\newfam\cmmibfam
\textfont\cmmibfam=\tencmmib

\font\tenmsb=msbm10 
\def\Bbb#1{\hbox{\tenmsb#1}}

\def\lessim{\ \lower4pt\hbox{$
\buildrel{\displaystyle <}\over\sim$}\ }
\def\gessim{\ \lower4pt\hbox{$\buildrel{\displaystyle >}
\over\sim$}\ }

\def\eps{\varepsilon}

\def\go0{\to 0}

\def\leftitem#1{\item{\hbox to\parindent{\enspace#1\hfill}}}

\def\sg{\sigma}

\def\sg2{\sigma^2}

\def\__{_{\infty}}
\numberwithin{equation}{section} 

\newcommand{\1}{{\rm 1}\kern-0.24em{\rm I}}

\newtheorem{assumption}{Assumption}

\endlocaldefs

\begin{document}

\begin{frontmatter}
\title{Estimation of smooth functionals of covariance operators: jackknife bias reduction and bounds in terms of effective rank}
\runtitle{Estimation of smooth functionals}

\begin{aug}
\author{\fnms{Vladimir} \snm{Koltchinskii}\ead[label=e1]{vlad@math.gatech.edu}} 
%\thankstext{t1}{Supported in part by NSF Grant DMS-2113121}
\runauthor{V. Koltchinskii}

%\affiliation{Georgia Institute of Technology\thanksmark{m1}}

\address{School of Mathematics\\
Georgia Institute of Technology\\
Atlanta, GA 30332-0160\\
\printead{e1}\\
% \phantom{E-mail:\ }
%\printead*{e2}
}
\end{aug}
\vspace{0.2cm}
{\small \today}
\vspace{0.2cm}

\begin{abstract}
Let $E$ be a separable Banach space and let $X, X_1,\dots, X_n, \dots$ be i.i.d. Gaussian random variables taking values in  $E$ with mean zero and  
unknown covariance operator $\Sigma: E^{\ast}\mapsto E.$ The complexity of estimation of $\Sigma$ based on the observations $X_1,\dots, X_n$ is naturally 
characterized by the so called effective rank of $\Sigma:$ ${\bf r}(\Sigma):= \frac{{\mathbb E}_{\Sigma}\|X\|^2}{\|\Sigma\|},$ where $\|\Sigma\|$ is the operator 
norm of $\Sigma.$ 
Given a smooth real valued functional $f$ defined on the space $L(E^{\ast},E)$ of symmetric linear operators from $E^{\ast}$ into $E$ (equipped with the operator norm),
our goal is to study the problem of estimation of $f(\Sigma)$ based on $X_1,\dots, X_n.$ A jackknife type bias reduction method will be considered for this problem and the dependence of the Orlicz norm error rates of the resulting estimators of $f(\Sigma)$ 
on the effective rank ${\bf r}(\Sigma),$ the sample size $n$ and the degree of H\"older smoothness $s$ of functional $f$ will be studied. In particular, it will be shown that, if ${\bf r}(\Sigma)\lesssim n^{\alpha}$ for some $\alpha\in (0,1)$ and $s\geq \frac{1}{1-\alpha},$ then 
the classical $\sqrt{n}$-rate is attainable and, if $s> \frac{1}{1-\alpha},$ then asymptotic normality and asymptotic efficiency of the resulting estimators hold. 
Previously, the results of this type (for different estimators) were obtained only in the case of finite dimensional Euclidean space $E={\mathbb R}^d$ and for covariance operators 
$\Sigma$
whose spectrum is bounded away from zero (in which case, ${\bf r}(\Sigma)\asymp d$).
\end{abstract}



\end{frontmatter}

\newpage

\section{Introduction and preliminaries}
\label{sec:intro}

In this paper, we study a problem of estimation of the value $f(\Sigma)$ of a smooth real valued functional $f$ based on i.i.d. mean zero Gaussian observations 
$X_1,\dots, X_n$ with values in a separable Banach space and with unknown covariance operator $\Sigma.$ The complexity of such estimation problems is 
naturally characterized by the effective rank ${\bf r}(\Sigma)$ of covariance operator $\Sigma.$ Our goal is to develop estimators of $f(\Sigma)$
with an optimal dependence of the risk on the sample size $n,$ effective rank ${\bf r}(\Sigma)$ and the degree of smoothness of functional $f.$ 
Up to now, this problem has been studied only in the case of a finite dimensional Euclidean space ${\mathbb R}^d$ with the spectrum of the covariance 
$\Sigma$ being bounded from above and bounded away from zero. In such a ``nearly isotropic" case, the effective rank ${\bf r}(\Sigma)$ is of the same order as 
dimension $d.$ We study the problem in a dimension free framework with its complexity being characterized by ${\bf r}(\Sigma).$

In the rest of Section \ref{sec:intro}, we review the results on bounds for sample covariance operators in terms of effective rank and provide the definitions of H\"older smoothness classes in Banach spaces as well as Orlicz norms used throughout the paper. We also provide a brief review of known results on functional estimation.  The main results of the paper are stated in Section \ref{sec:Main_th}. In Section \ref{jackknife}, we discuss a functional estimation method based on linear aggregation of plug-in estimators with different sample sizes and its jackknife version.  
Concentration bounds for functionals of sample covariance needed to prove the main results are developed in Section \ref{sec:conc}. The proofs of the main results, stated in Section \ref{sec:Main_th}, are provided in Section \ref{sec:proof_main}, in which the first two subsections deal with the proofs of upper bounds and Subsection \ref{min_max_loc_proof} deals with the proofs of minimax lower bounds.

\subsection{Estimation of covariance operators and effective rank}

Let $E$ be a separable Banach space with the dual space $E^{\ast}.$ Given a bounded linear functional $u\in E^{\ast}$ and a vector $x\in E,$ we will use the inner product 
notation $\langle x,u\rangle$ for the value of $u$ at $x$ (in the case when $E={\mathbb H}$ is a Hilbert space, this becomes the actual inner product).
Denote by $L(E^{\ast}, E)$ the space of all bounded linear operators 
$A: E^{\ast}\mapsto E$ equipped with the operator norm
\begin{align*}
\|A\|:= \sup_{u\in E^{\ast},\|u\|\leq 1} \|Au\|.
\end{align*}
Let $X$ be a Gaussian random variable in $E$ with mean zero and covariance operator $\Sigma: E^{\ast}\mapsto E,$
\begin{align*}
\Sigma u := {\mathbb E} \langle X, u\rangle X, u\in E^{\ast}.
\end{align*}
Given i.i.d. copies $X_1,\dots, X_n$ of $X,$ let $\hat \Sigma_n : E^{\ast}\mapsto E$ be the sample covariance operator based on $X_1,\dots, X_n:$
\begin{align*}
\hat \Sigma_n u := n^{-1}\sum_{j=1}^n \langle X_j, u\rangle X_j, u\in E^{\ast}.
\end{align*} 
Recall that, for $x,y\in E,$ the tensor product $x\otimes y$ could be viewed as a bounded operator from $E^{\ast}$ to $E,$ defined 
as follows: $(x\otimes y)u := x \langle y,u\rangle, u\in E^{\ast}.$ Then 
$\Sigma= {\mathbb E}(X\otimes X)$ and $\hat \Sigma_n= n^{-1}\sum_{j=1}^n X_j \otimes X_j.$

The estimation error $\|\hat \Sigma_n-\Sigma\|$ of sample covariance operator $\hat \Sigma_n$ can be characterized in terms of operator norm $\|\Sigma\|$
and so called {\it effective rank} ${\bf r}(\Sigma)$ of covariance operator $\Sigma,$ defined as follows:
\begin{align*}
{\bf r}(\Sigma):= \frac{{\mathbb E}\|X\|^2}{\|\Sigma\|}.
\end{align*}
Note that 
\begin{align*}
{\bf r}(\Sigma)= \frac{{\mathbb E}\sup_{\|u\|\leq 1} \langle X,u\rangle^2}{\sup_{\|u\|\leq 1} {\mathbb E}\langle X,u\rangle^2}\geq 1
\end{align*}
and that ${\bf r}(\lambda \Sigma)={\bf r}(\Sigma), \lambda>0.$
Moreover, it is easy to check that 
\begin{align*}
{\bf r}(\Sigma)\leq {\rm rank}(\Sigma)\leq {\rm dim }(E),
\end{align*}
and, 
if $E={\mathbb H}$ is a Hilbert space, then 
\begin{align*}
{\bf r}(\Sigma)= \frac{{\rm tr}(\Sigma)}{\|\Sigma\|}.
\end{align*}
%Note that the bound ${\bf r}(\Sigma)\leq {\rm dim}(E)$ holds also for general Banach spaces.
It is also worth mentioning that the effective rank ${\bf r}(\Sigma)$ was used earlier (under a different name) in the literature on local theory of Banach spaces, 
in particular, in connection with so called Gaussian reformulation of Dvoretzky's theorem (see \cite{Pisier}, Theorem 4.4).

The next results were proved in \cite{Koltchinskii_Lounici}.

\begin{theorem}
\label{KL_1}
The following bound holds:
\begin{align*}
{\mathbb E}\|\hat \Sigma_n-\Sigma\|\asymp \|\Sigma\|\biggl(\sqrt{\frac{{\bf r}(\Sigma)}{n}} \bigvee \frac{{\bf r}(\Sigma)}{n}\biggr).
\end{align*} 
\end{theorem}

\begin{theorem}
\label{KL_2}
For all $t\geq 1$ with probability at least $1-e^{-t},$
\begin{align*}
\Bigl|\|\hat \Sigma_n-\Sigma\|- {\mathbb E}\|\hat \Sigma_n-\Sigma\|\Bigr|
\lesssim \|\Sigma\| \biggl(\Bigl(\sqrt{\frac{{\bf r}(\Sigma)}{n}}\vee 1\Bigr)\sqrt{\frac{t}{n}}\bigvee \frac{t}{n}\biggr).
\end{align*}
\end{theorem}

In the above theorems and in what follows, we write $A\lesssim B$ (for variables $A,B>0$) if there exists a numerical constant $C>0$ such that $A\leq CB.$
The notation $A\gtrsim B$ is equivalent to $B\lesssim A$ and the notation $A\asymp B$ means that $A\lesssim B$ and $A\gtrsim B.$ In some cases, 
it is convenient to allow the constants in the above relationships to depend on some parameters. In such cases, the signs $\lesssim, \gtrsim, \asymp$ will be provided with subscripts:
say, $A\lesssim_s B$ means that there exists $C_s>0$
such that $A\leq C_s B.$

The bounds of theorems \ref{KL_1} and \ref{KL_2} show that the size of $\|\hat \Sigma_n-\Sigma\|$ is characterized in terms of $\|\Sigma\|$ and effective rank ${\bf r}(\Sigma).$
In the case of finite dimensional Euclidean space $E={\mathbb R}^d,$ ${\bf r}(\Sigma)\leq d,$ so, the bounds of the theorems imply well known bounds on 
$\|\hat \Sigma_n-\Sigma\|$ in terms of dimension $d.$ Such dimension dependent bounds are sharp in the cases when covariance $\Sigma$ is ``nearly isotropic"
in the sense that the spectrum $\sigma(\Sigma)$ of covariance $\Sigma$ is bounded from above and bounded away from zero by positive numerical constants.  
In fact, in this case ${\bf r}(\Sigma)\asymp d.$
In the ``non isotropic" case, when the eigenvalues of covariance could be close to zero (in particular, in the infinite dimensional case), the dimension free 
description of covariance $\Sigma$ in terms of its effective rank ${\bf r}(\Sigma)$ becomes necessary. 

Throughout the paper, we use the generic norm notation $\|\cdot\|$ for a variety of different norms in different spaces. For instance, 
it is the norm of Banach space $E,$ of its dual space $E^{\ast},$ the operator norm in the space $L(E^{\ast}, E),$ the operator norms in the spaces of multilinear forms in a Banach space, etc. In such cases, we do not provide the notation $\|\cdot\|$ with subscripts and its meaning should be clear from the context. Only if there is some ambiguity due to the fact that multiple norms are used in the same space, the subscripts
will be used. For instance, for operators acting in a Hilbert space ${\mathbb H},$ we will use the notation $\|\cdot\|_p$ for the Schatten $p$-norms
for $p\geq 1$ with $\|\cdot\|_2$ being the Hilbert--Schmidt norm and $\|\cdot\|_1$ being the nuclear norm. Norms in function spaces and spaces of r.v.
(such as the H\"older norms and Orlicz norms) will be also provided with subscripts.

\subsection{Orlicz spaces, H\"older smoothness, etc}
\label{sec:prelim}

Before stating the main results, we will introduce some notations used throughout the paper. Let $\psi:{\mathbb R}_+\mapsto {\mathbb R}_+$ be 
a convex nondecreasing function with $\psi(0)=0.$ 
The Orlicz $\psi$-norm of a r.v. $\eta$ is defined as 
\begin{align*}
\|\eta\|_{\psi}: = \|\eta\|_{L_{\psi}({\mathbb P})}:=\inf\Bigl\{c\geq 0: {\mathbb E} \psi\Bigl(\frac{|\eta|}{c}\Bigr)\leq 1\Bigr\}.
\end{align*}
Let $L_{\psi}({\mathbb P})$ be the space of all r.v. on a probability space $(\Omega, \Sigma, {\mathbb P})$ with finite $\psi$-norm.

For $\psi(u)=u^p, u\geq 0, p\geq 1,$ the $\psi$-norm coincides with the usual $L_p$-norm. The choice of $\psi =\psi_{\alpha},$ where
$\psi_{\alpha}(u):= e^{u^{\alpha}}-1, u\geq 0, \alpha\geq 1,$
yields the so called $\psi_{\alpha}$-norms and various spaces of random variables with exponentially decaying tails, in particular, the space 
$L_{\psi_2}({\mathbb P})$ of sub-gaussian r.v. and the space $L_{\psi_1}({\mathbb P})$ of sub-exponential r.v. 
It is well known that, for $\alpha\geq 1,$ the $\psi_{\alpha}$-norm is equivalent to the following norm defined in terms of the $L_p$-norms:
\begin{align}
\label{equiv_psi_alpha}
\|\eta\|_{\psi_{\alpha}}\asymp \sup_{p\geq 1} p^{-1/\alpha} {\mathbb E}^{1/p} |\eta|^p.
\end{align}
Note that the right hand side also defines a norm of $\eta$ for $\alpha \in (0,1)$ and, despite the fact that in this case $\|\cdot\|_{\psi_{\alpha}}$
is not a norm (since $\psi_{\alpha}$ is not convex), the equivalence \eqref{equiv_psi_alpha} still holds for $\alpha\in (0,1)$ (with constants depending on $\alpha$).
With a little abuse of notation, we will use the right hand side of \eqref{equiv_psi_alpha} as the definition of $\psi_{\alpha}$-norm for $\alpha\in (0,1)$
(in fact, it will be convenient to use this definition for all $\alpha>0$). Alternatively, it is possible to modify the definition of function $\psi_{\alpha}, \alpha<1$ in a neighborhood 
of $0$ so that it is convex and defines an equivalent Orlicz norm.

For two nondecreasing convex loss functions $\varphi, \psi:{\mathbb R}_+\mapsto {\mathbb R}_+$ with $\varphi (0)=\psi(0)=0,$ 
we write $\varphi \preceq \psi$ and say that $\varphi$ is dominated by $\psi$ iff there exist constants $c_1, c_2>0$ such that $\varphi (u)\leq c_1 \psi(c_2 u), u\geq 0.$
This implies that $\|\eta\|_{\varphi}\lesssim_{c_1,c_2} \|\eta\|_{\psi}.$

In what follows, we will measure 
the risk of estimators of $f(\Sigma)$ by the $\psi$-norms with a properly chosen loss function $\psi.$

For an arbitrary set (space) $S,$ a Banach space $F$ and $g:S\mapsto F,$ denote 
\begin{align*}
\|g\|_{L_{\infty}}:= \sup_{x\in S}\|g(x)\|.
\end{align*}
Also, if $F_1, F_2$ are Banach spaces and $g: F_1\mapsto F_2,$ let
\begin{align*} 
\|g\|_{\rm Lip}:= \sup_{x_1, x_2\in F_1, x_1\neq x_2}\frac{\|g(x_1)-g(x_2)\|}{\|x_1-x_2\|} 
\end{align*}
and, for $\rho \in (0,1],$ let
\begin{align*} 
 \|g\|_{{\rm Lip}_{\rho}}:= \sup_{x_1, x_2\in F_1, x_1\neq x_2}\frac{\|g(x_1)-g(x_2)\|}{\|x_1-x_2\|^{\rho}}. 
\end{align*}
We also denote 
\begin{align*}
{\rm Lip}(F_1, F_2):= \{g:F_1\mapsto F_2: \|g\|_{{\rm Lip}}<\infty\}\ {\rm and}\  
{\rm Lip}_{\rho}(F_1, F_2):= \{g:F_1\mapsto F_2: \|g\|_{{\rm Lip}_{\rho}}<\infty\}.
\end{align*}
Similarly, one can define spaces ${\rm Lip}(U,F_2)$ and ${\rm Lip}_{\rho}(U,F_2)$ and the corresponding norms $\|\cdot \|_{\rm Lip}=\|\cdot\|_{{\rm Lip}(U)}, \|\cdot \|_{{\rm Lip}_{\rho}}=\|\cdot\|_{{\rm Lip}_{\rho}(U)}$
for an arbitrary subset $U\subset F_1.$ 
If $F_1=F, F_2={\mathbb R},$ we write ${\rm Lip}(F)= {\rm Lip}(F, {\mathbb R})$ and 
${\rm Lip}_{\rho}(F)= {\rm Lip}_{\rho}(F, {\mathbb R}).$

For a $k$ times Fr\'echet differentiable functional $f: F \mapsto {\mathbb R},$ $f^{(k)}$ denotes its $k$-th  Fr\'echet derivative
(with $f^{(0)}=f$ and $f^{(1)}=f'$). 
Note that, for $x\in F,$ $f^{(k)}(x)$ is a bounded symmetric $k$-linear form. For such a form $M[h_1,\dots, h_k],$
its operator norm is defined as 
\begin{align*}
\|M\|:= \sup_{\|h_1\|, \dots, \|h_k\|\leq 1}|M[h_1,\dots, h_k]|.
\end{align*}
In particular, for $k=1,$ $f'(x)$ can be viewed as a bounded linear functional on $F,$ so, $f'(x)\in F^{\ast}$ and we can write $f'(x)[h]=\langle h,f'(x)\rangle, h, x\in F.$

In what follows, the spaces of bounded symmetric $k$-linear forms are always equipped with the operator norm.
For $s=k+\rho,$ $k\geq 0, \rho\in (0,1],$ the H\"older $C^s$-norm of $k$ times Fr\'echet continuously differentiable functional $f$ is defined by 
\begin{align*}
\|f\|_{C^s} := \max_{0\leq j\leq k} \sup_{x\in F} \|f^{(j)}(x)\| \vee \sup_{x,y\in F, x\neq y}\frac{\|f^{(k)}(x)- f^{(k)}(y)\|}{\|x-y\|^{\rho}}.
\end{align*}
We will also need the following weighted version of $C^s$-norm: for $a>0,$ let 
\begin{align*}
\|f\|_{C^{s,a}} := \max_{0\leq j\leq k} a^j \sup_{x\in F} \|f^{(j)}(x)\| \vee a^{s}\sup_{x,y\in F, x\neq y}\frac{\|f^{(k)}(x)- f^{(k)}(y)\|}{\|x-y\|^{\rho}}.
\end{align*}
Clearly, 
\begin{align*}
(a\wedge 1)^s \|f\|_{C^s} \leq \|f\|_{C^{s,a}} \leq (a\vee 1)^s \|f\|_{C^s}.
\end{align*}
Similarly, for $s=k+\rho, k\geq 0, \rho \in (0,1],$ one can define $C^s(U)$-norms and $C^{s,a}(U)$-norms for $k$ times Fr\'echet continuously differentiable functions $f:U\mapsto {\mathbb R}$
in an open subset $U\subset F.$

\begin{remark}
\label{Holder_Besov}
\normalfont 
In this paper, we are interested in smoothness of functionals on the space $L(E^{\ast}, E)$ of symmetric operators from $E^{\ast}$ into $E$ (equipped with the operator norm). 
In particular, in the case when $E={\mathbb H}$ is a separable Hilbert space and $L(E^{\ast}, E)=L({\mathbb H})$ is the space of bounded self-adjoint operators in ${\mathbb H},$ an interesting class of functionals 
is $f(S):= \langle g(S), B\rangle, S\in L({\mathbb H}),$ where $g:{\mathbb R}\mapsto {\mathbb R}$ is a smooth function of real variable and $B:{\mathbb H}\mapsto {\mathbb H}$ is a linear operator.
It is known from the operator theory that the H\"older smoothness of operator function $L({\mathbb H})\ni S\mapsto g(S)\in L({\mathbb H})$ is related to the Besov smoothness of function $g:{\mathbb R}\mapsto {\mathbb R}.$
In particular, for all $s>0$ and for all functions $g$ from the Besov space $B^s_{\infty,1}({\mathbb R}),$
\begin{align*}
\|g(\cdot)\|_{C^{s}(L({\mathbb H}))}\lesssim_s \|g\|_{B^{s}_{\infty,1}({\mathbb R})}
\end{align*}
(see \cite{Koltchinskii_2017}, Corollary 2 and references therein). Therefore, for functional $f(S), S\in L({\mathbb H}),$ we have 
\begin{align*}
\|f\|_{C^{s}(L({\mathbb H}))}\lesssim_s \|g\|_{B^{s}_{\infty,1}({\mathbb R})} \|B\|_1.
\end{align*}
This fact can be used to reduce the problem of estimation of many important functionals of covariance to the problem of estimation of H\"older smooth functionals. 
For instance, this approach could be applied to the functional $f(\Sigma)= \langle P(\Sigma) u,v\rangle,$ which is a bilinear form of a spectral projection $P(\Sigma)$ of covariance operator $\Sigma$ corresponding to its eigenvalue $\lambda(\Sigma)$
that is ``well separated" from the rest of its spectrum.
\end{remark}

\subsection{A brief review of recent results on estimation of functionals of covariance operators}

The problem of estimation of functionals of parameters of high-dimensional and infinite-dimensional statistical models has a long history going back 
to the 1970s. A very incomplete list of important references includes \cite{Levit_1, Levit_2, Ibragimov, Ibragimov_Khasm_Nemirov, Bickel_Ritov, Nemirovski_1990, Birge, Laurent, Lepski, Nemirovski, Robins, C_C_Tsybakov}. One of the main difficulties in this problem is related to the fact that, in the case of a high-dimensional parameter $\theta$ and its 
reasonable estimator $\hat \theta$ (for instance, the maximum likelihood estimator), a naive plug-in estimator $f(\hat \theta)$ of a non-linear functional $f(\theta)$ would typically have a large bias and, due to this, would fail to achieve optimal convergence rates. Thus, the development of bias reduction methods becomes a crucial ingredient of functional estimation. Motivated by this problem, several general approaches to bias reduction (jackknife, bootstrap and Taylor expansions methods) have been studied in recent paper \cite{Jiao} in the case of estimation of a smooth function of the parameter of binomial model and it was shown that,
even in this simple case, the analysis of these approaches lead to non-trivial problems of approximation theory. 

Estimation of functionals of covariance operators (for instance, in the case of normal models) is largely motivated by high-dimensional problems in multivariate statistical analysis. 
In particular, in principal component analysis, it is of importance to develop methods of estimation of linear forms of eigenvectors of unknown covariance, or of bilinear forms of its 
spectral projections, as well as some other functionals. Such problems have been studied in \cite{Koltchinskii_Lounici_bilinear, Koltchinskii_Lounici_AOS} in a dimension free framework with the effective rank ${\bf r}(\Sigma)$ playing the role of complexity parameter (in the case of i.i.d. data in a Hilbert space ${\mathbb H}$). 
The bias reduction method developed in \cite{Koltchinskii_Lounici_bilinear}
for estimation of linear forms of eigenvectors of $\Sigma$ was rather specialized and based on certain representations of the bias of bilinear forms of spectral projections of sample 
covariance. This method was developed to the full extent in \cite{Koltchinskii_Nickl} yielding asymptotically efficient estimators of linear forms of eigenvectors of $\Sigma$
provided that ${\bf r}(\Sigma)=o(n).$ The extension of this method to the case of estimation of bilinear forms of spectral projections of $\Sigma$ (corresponding to multiple eigenvalues)
has not been developed yet.

These difficulties motivated the study of more general bias reduction methods in the problem of estimation of functionals of the form $f(\Sigma):= \langle g(\Sigma), B\rangle,$
where $g:{\mathbb R}\mapsto {\mathbb R}$ is a smooth function of real variable \cite{Koltchinskii_2017, Koltchinskii_2018}. It is known that, if $B$ is a nuclear operator and function $g$ belongs to the Besov space $B^s_{\infty,1}({\mathbb R}),$ then functional $f$ defined on the space of bounded self-adjoint operators (equipped with an operator norm)
is of H\"older smoothness $s$ (see Remark \ref{Holder_Besov}). Note that the bilinear form $\langle P_\lambda u,v\rangle$ of the spectral projection $P_{\lambda}$ of operator $\Sigma$ corresponding 
to its eigenvalue $\lambda$ could be represented as $\langle g(\Sigma), B\rangle$ for a smooth function $g$ supported in a small enough neighborhood of $\lambda$
(that does not contain other eigenvalues of $\Sigma$). Moreover, by standard perturbation arguments, this representation also holds in an operator norm neighborhood 
of $\Sigma.$ Thus, estimation of bilinear forms of spectral projections of $\Sigma$ could be reduced to estimation of functionals of the form $\langle g(\Sigma), B\rangle.$

The bias reduction method studied in \cite{Koltchinskii_2017, Koltchinskii_2018, Koltchinskii_Zhilova, Koltchinskii_Zhilova_19} (called the {\it bootstrap chain bias reduction}) 
goes back to the idea of iterated bootstrap bias reduction \cite{Hall_1}. In the case of binomial model,
it was also studied in \cite{Jiao} based on the ideas and results from classical approximation theory. To describe this approach in the case of functionals of covariance (the case of general parameter $\theta$ is similar), 
consider the following {\it Wishart operator} 
\begin{align*}
({\mathcal T} h)(\Sigma) := {\mathbb E}_{\Sigma} h(\hat \Sigma_n) = \int_{{\mathcal C}_+({\mathbb H})} g(S)P(\Sigma, dS), h\in L_{\infty}({\mathcal C}_{+}({\mathbb H}))
\end{align*}
acting in the space of bounded functions on the cone ${\mathcal C}_+({\mathbb H})$ of covariance operators in ${\mathbb H},$ where 
\begin{align*}
P(\Sigma, A):= {\mathbb P}_{\Sigma}\{\hat \Sigma_n\in A\}, A\subset {\mathcal C}_+({\mathbb H})
\end{align*}
is the distribution of the sample covariance operator $\hat \Sigma_n$ (clearly, $P(\Sigma, A)$ is a Markov kernel in the cone ${\mathcal C}_+({\mathbb H})).$
Let $({\mathcal B} h)(\Sigma):= ({\mathcal T}h)(\Sigma)-h(\Sigma), \Sigma \in {\mathcal C}_+({\mathbb H}).$ Since the bias of the plug-in estimator $f(\hat \Sigma_n)$
of $f(\Sigma)$ is equal to $({\mathcal B}f)(\Sigma),$ the first order bias correction yields an estimator $f_1(\hat \Sigma_n):= f(\hat \Sigma_n)-({\mathcal B}f)(\hat \Sigma_n).$
Similarly, the bias of estimator $({\mathcal B}f)(\hat \Sigma_n)$ of $({\mathcal B}f)(\Sigma)$ is equal to $({\mathcal B}^2 f)(\Sigma),$ and the second order bias correction 
yields an estimator $f_2(\hat \Sigma_n):= f(\hat \Sigma_n)-({\mathcal B}f)(\hat \Sigma_n)+({\mathcal B}^2 f)(\hat \Sigma_n).$ After $k$ iterations, we get an estimator 
\begin{align*}
f_k(\hat \Sigma_n) := \sum_{j=0}^k (-1)^j ({\mathcal B}^j f)(\hat \Sigma_n).
\end{align*}
Another way to look at it is to observe that to find an estimator of $f(\Sigma)$ with a small bias, one has to solve approximately the operator equation $({\mathcal T}h)(\Sigma)= f(\Sigma), \Sigma\in {\mathcal C}_+({\mathbb H})$ and to use $h(\hat \Sigma_n)$ as an estimator. Since ${\mathcal T}={\mathcal I}+ {\mathcal B},$ the solution of the equation $Th =f$ can be formally written as a Neumann series 
\begin{align*} 
h(\Sigma)= \sum_{j=0}^{\infty} (-1)^j ({\mathcal B}^j f)(\Sigma)
\end{align*}
and the partial sums of this series
\begin{align*}
f_k(\Sigma) := \sum_{j=0}^{k} (-1)^j ({\mathcal B}^j f)(\Sigma)
\end{align*}
provide approximate solutions of the equation. 

It is easy to see that the bias of estimator $f_k(\hat \Sigma_n)$ is given by the following formula: 
\begin{align*}
{\mathbb E}_{\Sigma} f_k(\hat \Sigma_n)- f(\Sigma) = (-1)^k ({\mathcal B}^{k+1} f)(\Sigma),
\end{align*}
so, to show that the bias of estimator $f_k(\hat \Sigma_n)$ is sufficiently small when the functional $f$ is sufficiently smooth, it is enough to control 
the size of  $({\mathcal B}^{k+1} f)(\Sigma)$ for smooth $f.$  This is based on representation of functions ${\mathcal B}^k f, k\geq 0$ in terms 
of so called {\it bootstrap chain}, which is a Markov chain $\{\hat \Sigma_n^{(k)}: k\geq 0\}$ with $\hat \Sigma_n^{(0)}=\Sigma$ and with transition probability 
kernel $P(\Sigma;A).$ This chain can be interpreted as an application of iterative parametric bootstrap to estimator $\hat \Sigma_n:$ for $k=0,$ the chain starts at $\Sigma;$ 
for $k=1,$ its value $\hat \Sigma_n^{(1)}=\hat \Sigma_n;$ for $k=2,$ its value $\hat \Sigma_n^{(2)}$ is a bootstrap estimator based on i.i.d. data $\hat X_1,\dots, \hat X_n$
resampled from $N(0;\hat \Sigma_n)$ (conditionally on $\hat \Sigma_n$), etc. It is not hard to see that 
\begin{align*}
({\mathcal B}^k f)(\Sigma) = {\mathbb E}_{\Sigma}\sum_{j=0}^k (-1)^{k-j}{k\choose j} f(\hat \Sigma_n^{(j)}),
\end{align*}
which is the expectation of the $k$-th order difference of the values of functional $f$ along the trajectory of bootstrap chain. 
For a function $f:{\mathbb R}\mapsto {\mathbb R},$ its $k$-th order difference at point $x\in {\mathbb R}$ with step $h$ is 
\begin{align*}
(\Delta^k_h f)(x) = \sum_{j=0}^k (-1)^{k-j} {k\choose j} f(x+j h)
\end{align*}
and, if $f$ is $k$ times continuously differentiable, $(\Delta^k_h f)(x) = O(h^k).$
In view of theorems \ref{KL_1}, \ref{KL_2},
under the assumption ${\bf r}(\Sigma)\lesssim n,$ $\|\hat \Sigma_n-\Sigma\|$ is roughly of the order $\|\Sigma\|\sqrt{\frac{{\bf r}(\Sigma)}{n}},$ and this is also the size of the jumps 
of bootstrap chain $\{\hat \Sigma_n^{(k)}: k\geq 0\}.$ By an analogy with the behavior of the $k$-th order differences of smooth functions in the real line, one could expect that, for a $k$ times continuously differentiable functional $f,$ one would have that 
\begin{align*}
({\mathcal B}^k f)(\Sigma)\lesssim \biggl(\|\Sigma\|\sqrt{\frac{{\bf r}(\Sigma)}{n}}\biggr)^k.
\end{align*}
This would imply that estimator $f_k(\hat \Sigma_n)$ does have a reduced bias when ${\bf r}(\Sigma)$ is small comparing with $n$ and 
$k$ is sufficiently large. The rigorous proof of this fact is, however, rather complicated and it required the development of a number of probabilistic 
and analytic tools. It was done in \cite{Koltchinskii_2017, Koltchinskii_2018} for functionals of the form $f(\Sigma)=\langle g(\Sigma), B\rangle$ 
and in \cite{Koltchinskii_Zhilova_19} for general H\"older smooth functionals, yielding sharp bounds on the risk of estimator $f_k(\hat \Sigma_n)$ and the proof of its asymptotic efficiency (in fact, in the last paper the problem of estimation of functionals of unknown mean and covariance of normal model was studied). 
In both cases, however, it was done {\it only} in the case of covariances 
in the Euclidean space ${\mathbb R}^d$ and under the assumption that the spectrum of $\Sigma$ belongs to the interval $[1/a,a]$ for some $a\geq 1.$
Note that, in this case, ${\bf r}(\Sigma)\asymp d,$ so, there is no need to use the effective rank of $\Sigma$ as a complexity parameter. 
This assumption essentially means that the covariance is ``almost" isotropic and the results similar to the ones stated in Section \ref{sec:Main_th} below 
were proved for estimator $f_k(\hat \Sigma_n)$ with dimension $d$ instead of ${\bf r}(\Sigma).$ In particular, it was shown that, for such functional estimators, 
the upper bound on the estimation error is of the order $\frac{1}{\sqrt{n}}+ \Bigl(\sqrt{\frac{d}{n}}\Bigr)^s,$ where $s$ is the degree of H\"older smoothness 
of the functional. 

{\it We conjecture that similar results hold for estimator $f_k(\hat \Sigma_n)$ in the general dimension free case with complexity of the problem characterized by the effective rank ${\bf r}(\Sigma)$}, but we do not know how to prove this and the methods developed in  \cite{Koltchinskii_2017, Koltchinskii_2018, Koltchinskii_Zhilova_19} do not seem to suffice for this. However, as it is stated in Section \ref{sec:Main_th} and will be proved later in the paper, such results do hold 
for jackknife type estimators and their proofs are much simpler than for bootstrap chain estimators. 

Minimax bounds similar to the bound of Theorem \ref{loc_min_max} below were proved in \cite{Koltchinskii_Zhilova} for functionals of parameters of Gaussian shift models. In \cite{Koltchinskii_Zhilova_19},  a similar result was stated for functionals of mean 
and covariance of high-dimensional normal model. However, in this case, the bound is attained for functionals depending only on the mean and the result easily followed from the minimax bounds for Gaussian shift models obtained in \cite{Koltchinskii_Zhilova}. 
For functionals of covariance, the problem remained open and it is solved in the current paper (Theorem \ref{loc_min_max}). The proof is based on some of the ideas already present in the case of Gaussian shift model, but the argument is much more sophisticated in the case of functionals of covariance (see Section \ref{min_max_loc_proof}).

Some other results on bootstrap chain bias reduction and efficient estimation of functionals of unknown parameters for Gaussian shift models, more general random shift models, log-concave location models can be found in \cite{Koltchinskii_Zhilova, Koltchinskii_Zhilova_21, Koltchinskii_Wahl}.
In \cite{Koltchinskii_2021}, this method was studied in the case of general high-dimensional parametric models under the assumption 
that there exists an estimator of unknown parameter whose distribution could be approximated by Gaussian with sufficient accuracy. 
In \cite{Fan}, a version of Taylor expansions method of bias reduction was developed for estimation of functionals of the unknown mean 
with not necessarily i.i.d. and not necessarily Gaussian additive random noise. This method was also applied in \cite{Fan} to additive 
functionals with not necessarily smooth components (via an approximation by smooth functionals). 

Although the focus of the current paper is on estimation of H\"older smooth functionals of covariance, the bias reduction method we develop could be useful 
in estimation of non-smooth functionals as well, usually, in combination with other tools (as it is the case with other bias reduction methods briefly reviewed in this section).
However, in our view, the development of higher order bias reduction methods for functionals in various smoothness classes is the basic problem in functional estimation and its understanding 
is crucial for the development of similar methods for broader classes functionals lacking smoothness.

\section{Main results}
\label{sec:Main_th}

In this section, we provide a number of results on estimation error rates for smooth functionals of covariance in the $L_p$-norms as well as some other Orlicz norms, depending on the degree of smoothness $s$ of the functionals, the effective rank ${\bf r}(\Sigma)$ of the covariance operator $\Sigma$ and the sample size $n.$ 
We also provide results on the first order linear approximation of functional estimators needed to establish their normal approximation and efficiency properties.
For smaller degree of smoothness $s\leq 2,$ the minimax optimal estimation error rates are attained for the plug-in estimator. For larger smoothness $s>2,$ more sophisticated methods of bias reduction are needed to achieve optimal rates and the functional estimators are based on linear aggregation of plug-in estimators with different sample sizes and jackknife type methods.   
The table below summarizes various types of results stated in this section:

\vskip 2mm

\begin{tabular}{p{55pt}|p{82pt}|p{82pt}|p{82pt}|p{50pt}}
\toprule
& Smoothness~$\leqslant 2$,\ glob.~upp.~bound, plug-in estim.
& Smoothness~$>2$, glob.~upp.~bound, jackknife estim.
& Smoothness~$>2$, loc.~upp.~bound, jackknife estim. &Lower bounds\\
\midrule
estimation error & Theorem 2.1 & Theorem 2.3 & Theorem 2.5 &
Theorems 2.7 \& 2.8\\
\midrule
first order approx. & Theorem 2.2 & Theorem 2.4 & Theorem 2.6 &\\
\bottomrule
\end{tabular}

\vskip 2mm

Recall that we are interested in estimation of the value $f(\Sigma)$ of a functional $f:L(E^{\ast}, E)\mapsto {\mathbb R}$ based on i.i.d. Gaussian observations 
$X_1,\dots, X_n \sim N(0,\Sigma)$ with values in $E$ and unknown covariance $\Sigma.$
First we state a couple of results on estimation of $f(\Sigma)$ for a functional $f$ of smoothness $s\in (0,2].$
It turns out that, in this case, a simple plug-in estimator $f(\hat \Sigma_n)$ suffices to achieve optimal error rates.

\begin{theorem}
\label{main_th_0}
Let $X_1,\dots, X_n$ be i.i.d. $N(0,\Sigma).$ 

(i) Suppose $f: L(E^{\ast},E)\mapsto {\mathbb R}$ satisfies the H\"older condition 
with exponent $\rho\in (0,1].$ Then, for all $p\geq 1,$
\begin{align*}
&
\Bigl\|f(\hat \Sigma_n)-f(\Sigma)\Bigr\|_{L_p({\mathbb P}_{\Sigma})}
\lesssim \|f\|_{{\rm Lip}_{\rho}}\|\Sigma\|^{\rho}
\biggl[\biggl(\sqrt{\frac{{\bf r}(\Sigma)}{n}}\vee \frac{{\bf r}(\Sigma)}{n}\biggr)^{\rho}
+ \biggl(\sqrt{\frac{{\bf r}(\Sigma)}{n}}\vee 1\biggr)^{\rho} \biggl(\sqrt{\frac{p}{n}}\Biggr)^{\rho} +\biggl(\frac{p}{n}\biggr)^{\rho}\biggr].
\end{align*}

(ii) Suppose $f: L(E^{\ast},E)\mapsto {\mathbb R}$ is Lipschitz and, moreover, it is Fr\'echet continuously differentiable with the derivative $f'$ satisfying the H\"older condition 
with exponent $\rho\in (0,1].$
Then, for all $p\geq 1,$ 
\begin{align*}
&
\Bigl\|f(\hat \Sigma_n)-f(\Sigma)\Bigr\|_{L_p({\mathbb P}_{\Sigma})} 
\lesssim 
\|f\|_{{\rm Lip}}\|\Sigma\|
\biggl(\sqrt{\frac{p}{n}}\biggl(\sqrt{\frac{{\bf r}(\Sigma)}{n}}\vee 1\biggr)+ \frac{p}{n}\biggr)
+ \|f'\|_{{\rm Lip}_{\rho}} \|\Sigma\|^{1+\rho}\biggl(\sqrt{\frac{{\bf r}(\Sigma)}{n}} \bigvee \frac{{\bf r}(\Sigma)}{n}\biggr)^{1+\rho}.
\end{align*}
\end{theorem}

The following corollary is immediate.

\begin{corollary}
\label{cor_main_th_0}
(i) Under the assumptions of Theorem \ref{main_th_0} (i),
\begin{align*}
\Bigl\|f(\hat \Sigma_n)-f(\Sigma)\Bigr\|_{L_{\psi_{1/\rho}}({\mathbb P}_{\Sigma})}
\lesssim \|f\|_{{\rm Lip}_{\rho}}\|\Sigma\|^{\rho}
\biggl(\sqrt{\frac{{\bf r}(\Sigma)}{n}}\vee \frac{{\bf r}(\Sigma)}{n}\biggr)^{\rho}.
\end{align*}
(ii) Under the assumptions of Theorem \ref{main_th_0} (ii),
\begin{align*}
&
\Bigl\|f(\hat \Sigma_n)-f(\Sigma)\Bigr\|_{L_{\psi_{1}}({\mathbb P}_{\Sigma})}
\lesssim 
\|f\|_{{\rm Lip}}\biggl(\sqrt{\frac{{\bf r}(\Sigma)}{n}}\vee 1\biggr)\frac{\|\Sigma\|}{\sqrt{n}} +\|f'\|_{{\rm Lip}_{\rho}} \|\Sigma\|^{1+\rho}\biggl(\sqrt{\frac{{\bf r}(\Sigma)}{n}} \bigvee \frac{{\bf r}(\Sigma)}{n}\biggr)^{1+\rho}.
\end{align*}
\end{corollary}

Note that, for $s=\rho \in (0,1),$ the error rate of plug-in estimator $f(\hat \Sigma_n)$ is slower than $n^{-1/2}.$ For $s=1,$ the $n^{-1/2}$-rate is possible, but only when the effective rank 
${\bf r}(\Sigma)\lesssim 1.$  For $s=1+\rho, \rho\in (0,1],$ the $n^{-1/2}$-rate is possible for the plug-in estimator even for large ${\bf r}(\Sigma),$ provided that the degree of smoothness 
$s$ is above certain threshold.  
Namely, this is the case when ${\bf r}(\Sigma)\lesssim n^{\alpha}$ for some $\alpha\in (0,1/2]$ and $s\geq \frac{1}{1-\alpha}.$ The optimality of the error rates of Theorem \ref{main_th_0} and Corollary \ref{cor_main_th_0} (or, more precisely, of the local versions of these results in the case when $E={\mathbb H}$ is a Hilbert space) follows from Theorem \ref{loc_min_max} stated later in this section. 

In the case when the degree of smoothness $s$ of the functional $f$ is larger than $1,$ it is also possible to prove the following theorem that, in particular, implies normal approximation properties of plug-in estimator $f(\hat \Sigma_n).$

\begin{theorem}
\label{main_th_2_0}
Suppose $X_1,\dots, X_n$ are i.i.d. $N(0,\Sigma)$ with ${\bf r}(\Sigma)\lesssim n.$
Let $f: L(E^{\ast},E)\mapsto {\mathbb R}$ be a Fr\'echet continuously differentiable functional with 
the first derivative $f'$ satisfying the H\"older condition 
with exponent $\rho\in (0,1].$ Then, for all $p\geq 1,$
\begin{align*}
&
\Bigl\|f(\hat \Sigma_n)-f(\Sigma)- \langle \hat \Sigma_n-\Sigma, f '(\Sigma)\rangle\Bigr\|_{L_p({\mathbb P}_{\Sigma})} 
\\
&
\lesssim_{\rho}
\|f'\|_{{\rm Lip}_{\rho}}\|\Sigma\|^{1+\rho}
\biggl[\sqrt{\frac{p}{n}}
\Bigl(\sqrt{\frac{{\bf r}(\Sigma)}{n}}\Bigr)^{\rho} + 
\Bigl(\frac{p}{n}\Bigr)^{(1+\rho)/2} +\Bigl(\frac{p}{n}\Bigr)^{1+\rho} 
+
\biggl(\sqrt{\frac{{\bf r}(\Sigma)}{n}}\biggr)^{1+\rho}\biggr].
\end{align*}
\end{theorem}

An immediate consequence is the following corollary.

\begin{corollary}
\label{cor_main_th_2_0}
Under the assumptions of Theorem \ref{main_th_2_0},
\begin{align*}
&
\Bigl\|f(\hat \Sigma_n)-f(\Sigma)- \langle \hat \Sigma_n-\Sigma, f '(\Sigma)\rangle\Bigr\|_{L_{\psi_{1/(1+\rho)}}({\mathbb P}_{\Sigma})} 
\lesssim_{\rho}
\|f'\|_{{\rm Lip}_{\rho}}\|\Sigma\|^{1+\rho}
\biggl(\sqrt{\frac{{\bf r}(\Sigma)}{n}}\biggr)^{1+\rho}.
\end{align*}
\end{corollary}

Note that the error of the first order linear approximation of plug-in estimator $f(\hat \Sigma_n)$ becomes $o(n^{-1/2})$ when ${\bf r}(\Sigma)\lesssim n^{\alpha}$ for some $\alpha\in (0,1/2)$
and $s=1+\rho>\frac{1}{1-\alpha},$ which provides a way to establish normal approximation of $f(\hat\Sigma_n)$ in this case.

It is well known that, for functionals $f:L(E^{\ast}, E)\mapsto {\mathbb R}$ of smoothness $s>2,$ the plug-in estimator $f(\hat \Sigma_n)$ of $f(\Sigma)$ could become suboptimal due to its large bias and a bias reduction is needed to construct estimators with optimal error rates (see, e.g., \cite{Koltchinskii_2017, Koltchinskii_2018}). In particular, for $s>2,$ there are functionals $f$ for which  
plug-in estimator $f(\hat \Sigma_n)$ fails to achieve $n^{-1/2}$-rate when ${\bf r}(\Sigma)$ is larger that $n^{1/2}$ regardless of how large the degree of smoothness $s$ of the functional is. 
We will study an approach to this bias reduction problem based 
on {\it linear aggregation of several plug-in estimators with different sample sizes}  and with the coefficients of the linear combination chosen in such a way that the biases on the plug-in estimators almost cancel each other out.  
This idea is well known in bias reduction literature and, in particular, it leads to a class of jackknife bias reduction methods (see \cite{Jiao}).

To define our estimators of $f(\Sigma),$ let $k\geq 1$ and let $n_1,\dots, n_k$ denote the sample sizes of plug-in estimators.   
Assume that $n/c\leq n_1<n_2<\dots<n_k\leq n$ for some $c>1$ and denote $\vec{n}:= (n_1,\dots, n_k).$   
Define 
\begin{align*}
T_{f,\vec{n}}^{(1)}(X_1,\dots, X_n)
:=\sum_{j=1}^k C_j f(\hat \Sigma_{n_j}),
\end{align*}
where 
\begin{align}
\label{choice_weight}
C_j := \prod_{i\neq j} \frac{n_j}{n_j-n_i}, j=1,\dots, k.
\end{align}
It could be shown that $\sum_{j=1}^k C_j=1$ (see \cite{Jiao}). 

In the rest of the paper, the following assumption holds. 

\begin{assumption}
\label{assume_on_C_j}
\normalfont
Suppose that $\sum_{j=1}^k |C_j| \lesssim_k 1.$
\end{assumption}

Clearly, for Assumption \ref{assume_on_C_j} to hold, it is necessary that  
$n_{j+1}-n_j \asymp n, j=1,\dots, k-1.$
For instance, one could take $n_j:= q^{j-k} n, j=1,\dots, k$ for some $q>1.$

Let now ${\mathcal F}_{\rm sym}$ denote the $\sigma$-algebra generated by random variables of the form $\psi(X_1\otimes X_1,\dots, X_n\otimes X_n),$
where $\psi: L(E^{\ast}, E)\times \dots \times L(E^{\ast},E)\mapsto {\mathbb R}$ is a symmetric Borel function of $n$ variables.  
Let 
\begin{align*}
T_{f,\vec{n}}^{(2)}(X_1,\dots, X_n) 
:= {\mathbb E}\Bigl(T_{f,\vec{n}}^{(1)}(X_1,\dots, X_n)\Bigl|{\mathcal F}_{\rm sym}\Bigr),
\end{align*}
which could be easily written as a linear combination of $U$-statistics (see Section \ref{jackknife}).
It is well known that the plug-in estimator $f(\hat \Sigma_n)$ of $f(\Sigma)$ has a large bias (despite the fact that $\hat \Sigma_n$ is an unbiased estimator of $\Sigma$).
It turns out, however, that, for a smooth functional $f$ with the choice \eqref{choice_weight} of weights $C_j,$ the biases of estimators $f(\hat \Sigma_{n_j})$ of $f(\Sigma)$ almost cancel out resulting in a small bias of estimators 
$T^{(1)}_{f,\vec{n}} (X_1,\dots, X_n)$ and $T^{(2)}_{f,\vec{n}} (X_1,\dots, X_n).$ Of course, these estimators depend on the parameter $k$ and on the choice of sample sizes $n_1,\dots, n_k$ for which Assumption \ref{assume_on_C_j} holds. 
Moreover, {\it the choice of $k$ in the main results stated below depends on the degree of smoothness of the functional $f.$} 
However, to simplify the notations, we will write in what follows $T^{(1)}_f (X_1,\dots, X_n)=T^{(1)}_{f,\vec{n}} (X_1,\dots, X_n)$ and $T^{(2)}_f (X_1,\dots, X_n)=T^{(2)}_{f,\vec{n}} (X_1,\dots, X_n),$ suppressing their dependence on 
$k$ and $n_1,\dots, n_k.$

The following results will be proved. 

\begin{theorem}
\label{main_th_1}
Suppose $X_1,\dots, X_n$ are i.i.d. $N(0,\Sigma).$ 
Let $f: L(E^{\ast},E)\mapsto {\mathbb R}$ be Lipschitz and, moreover, for some $k\geq 2,$ let it be $k$ times Fr\'echet continuously differentiable with the $k$-th derivative $f^{(k)}$ satisfying the H\"older condition 
with exponent $\rho\in (0,1].$
Then, for $i=1,2$ and for all $p\geq 1,$
 \begin{align*}
 &
\Bigl\|T^{(i)}_f (X_1,\dots, X_n)-f(\Sigma)\Bigr\|_{L_p({\mathbb P}_{\Sigma})} 
\\
&
\lesssim_{k,\rho}
\|f\|_{{\rm Lip}}\|\Sigma\|
\biggl(\sqrt{\frac{p}{n}}\biggl(\sqrt{\frac{{\bf r}(\Sigma)}{n}}\vee 1\biggr)+ \frac{p}{n}\biggr)
+ \|f^{(k)}\|_{{\rm Lip}_{\rho}} \|\Sigma\|^{k+\rho}\biggl(\sqrt{\frac{{\bf r}(\Sigma)}{n}} \bigvee \frac{{\bf r}(\Sigma)}{n}\biggr)^{k+\rho}.
 \end{align*}
\end{theorem}

\vskip 2mm

\begin{corollary}
\label{cor_main_th_1}
Under the assumptions of Theorem \ref{main_th_1}, for $i=1,2,$
\begin{align*}
&
\Bigl\|T^{(i)}_f (X_1,\dots, X_n)-f(\Sigma)\Bigr\|_{L_{\psi_1}({\mathbb P}_{\Sigma})} 
\lesssim_s 
\|f\|_{{\rm Lip}} \biggl(\sqrt{\frac{{\bf r}(\Sigma)}{n}}\vee 1\biggr)\frac{\|\Sigma\|}{\sqrt{n}}
+ \|f^{(k)}\|_{{\rm Lip}_{\rho}} \|\Sigma\|^{k+\rho}\biggl(\sqrt{\frac{{\bf r}(\Sigma)}{n}} \bigvee \frac{{\bf r}(\Sigma)}{n}\biggr)^{k+\rho}.
\end{align*}
\end{corollary}

\begin{remark}
\normalfont
\begin{enumerate}
\item Clearly, the results of Corollary \ref{cor_main_th_1} and Corollary \ref{cor_main_th_0} (ii)
imply that the same bounds also hold for the $\psi$-norm error for an arbitrary loss function 
$\psi\preceq \psi_1.$

\item If ${\bf r}(\Sigma)\lesssim n,$ the bound of Corollary \ref{cor_main_th_1} simplifies as follows:
\begin{align*}
&
\Bigl\|T^{(i)}_f (X_1,\dots, X_n)-f(\Sigma)\Bigr\|_{L_{\psi_1}({\mathbb P}_{\Sigma})} 
\lesssim_s 
\|f\|_{{\rm Lip}}\frac{\|\Sigma\|}{\sqrt{n}}
+ \|f^{(k)}\|_{{\rm Lip}_{\rho}} \|\Sigma\|^{k+\rho}\biggl(\sqrt{\frac{{\bf r}(\Sigma)}{n}}\biggr)^{k+\rho}.
\end{align*}
Moreover, if ${\bf r}(\Sigma)\lesssim n^{\alpha}$ for some $\alpha\in (0,1)$ and $s:= k+\rho\geq \frac{1}{1-\alpha},$
then the first term of the bound is dominant and 
\begin{align*}
\Bigl\|T^{(i)}_f (X_1,\dots, X_n)-f(\Sigma)\Bigr\|_{L_{\psi_1}({\mathbb P}_{\Sigma})}= O(n^{-1/2})
\end{align*}
yielding the classical parametric rates of convergence of estimators $T^{(i)}_f (X_1,\dots, X_n)$ for sufficiently smooth functionals $f.$
\end{enumerate}
\end{remark}

The next theorem shows that if ${\bf r}(\Sigma)\lesssim n^{\alpha}$ for some $\alpha\in (0,1)$ and $s=k+\rho >\frac{1}{1-\alpha},$ then $T^{(2)}_f (X_1,\dots, X_n)-f(\Sigma)$
can be approximated by a sum of i.i.d. random variables with the remainder of the order $o(n^{-1/2}),$ which suffices to establish the normal approximation of 
$\sqrt{n}(T^{(2)}_f (X_1,\dots, X_n)-f(\Sigma))$ (yielding both asymptotic normality and asymptotic efficiency of estimator $T^{(2)}_f (X_1,\dots, X_n)$).

\begin{theorem}
\label{main_th_2}
Suppose $X_1,\dots, X_n$ are i.i.d. $N(0,\Sigma)$ with ${\bf r}(\Sigma)\lesssim n.$
Let $f: L(E^{\ast},E)\mapsto {\mathbb R}$ be a $k$ times Fr\'echet continuously differentiable functional for some $k\geq 2$ with 
the first derivative $f'$ satisfying the H\"older condition 
with exponent $\gamma\in (0,1]$ and the $k$-th derivative $f^{(k)}$ satisfying the H\"older condition 
with exponent $\rho \in (0,1].$ Then, for all $p\geq 1,$
\begin{align*}
&
\Bigl\|T^{(2)}_f (X_1,\dots, X_n)-f(\Sigma)- \langle \hat \Sigma_n-\Sigma, f '(\Sigma)\rangle\Bigr\|_{L_p({\mathbb P}_{\Sigma})} 
\\
&
\lesssim_{k,\rho, \gamma}
\|f'\|_{{\rm Lip}_{\gamma}}\|\Sigma\|^{1+\gamma}\sqrt{\frac{p}{n}}
\Bigl(\sqrt{\frac{{\bf r}(\Sigma)}{n}}\Bigr)^{\gamma} + \|f'\|_{{\rm Lip}_{\gamma}}\|\Sigma\|^{1+\gamma}
\biggl(\Bigl(\frac{p}{n}\Bigr)^{(1+\gamma)/2} + \Bigl(\frac{p}{n}\Bigr)^{1+\gamma} \biggr)
\\
&
+ \|f^{(k)}\|_{{\rm Lip}_{\rho}} \|\Sigma\|^{k+\rho}
\biggl(\sqrt{\frac{{\bf r}(\Sigma)}{n}}\biggr)^{k+\rho}.
\end{align*}
\end{theorem}

\begin{corollary}
\label{cor_main_th_2}
Suppose $X_1,\dots, X_n$ are i.i.d. $N(0,\Sigma)$ with ${\bf r}(\Sigma)\lesssim n.$
 If $f: L(E^{\ast},E)\mapsto {\mathbb R}$ is $k$ times Fr\'echet continuously differentiable for some $k\geq 2$ with $\|f'\|_{C^1}<\infty$ and 
with the $k$-th derivative $f^{(k)}$ satisfying the H\"older condition 
with exponent $\rho \in (0,1],$ then, for all $\beta\in [1/2,1),$ 
\begin{align*}
&
\Bigl\|T^{(2)}_f (X_1,\dots, X_n)-f(\Sigma)- \langle \hat \Sigma_n-\Sigma, f '(\Sigma)\rangle\Bigr\|_{L_{\psi_{\beta}}({\mathbb P}_{\Sigma})} 
\\
&
\lesssim_{k,\rho} 
\|f'\|_{C^1} 
\frac{\|\Sigma\|^{1/\beta}}{\sqrt{n}}
\biggl(\sqrt{\frac{{\bf r}(\Sigma)}{n}}\biggr)^{1/\beta-1}
+ \|f^{(k)}\|_{{\rm Lip}_{\rho}} \|\Sigma\|^{k+\rho}
\biggl(\sqrt{\frac{{\bf r}(\Sigma)}{n}}\biggr)^{k+\rho}.
\end{align*}
\end{corollary}

Since 
\begin{align*}
\langle \hat \Sigma_n-\Sigma, f'(\Sigma)\rangle
=n^{-1} \sum_{j=1}^n \langle X_j\otimes X_j, f '(\Sigma)\rangle - {\mathbb E}\langle X\otimes X, f '(\Sigma)\rangle,
\end{align*}
the bounds of 
Theorem \ref{main_th_2_0}, Corollary \ref{cor_main_th_2_0}, 
Theorem \ref{main_th_2} and Corollary \ref{cor_main_th_2} can be used to approximate the ``Orcliz risk" 
of estimators $f(\hat \Sigma_n)$ and $T^{(2)}_f (X_1,\dots, X_n)$ for losses $\psi$ dominated by the sub-exponential loss $\psi_1$ (in particular, for the $L_p$-losses) as well as to develop normal approximation 
bounds for these estimators. 

Let 
\begin{align*}
\sigma_f^2(\Sigma):= {\mathbb E}\Bigl(\langle X\otimes X, f '(\Sigma)\rangle - {\mathbb E}\langle X\otimes X, f '(\Sigma)\rangle\Bigr)^2.
\end{align*}
It could be shown (see Lemma \ref{bd_on_psi_1} below) that $\sigma_f(\Sigma) \lesssim \|\Sigma\| \|f'(\Sigma)\|.$ 

\begin{corollary}
\label{cor_main_th_2'''}
Under the assumptions of Theorem \ref{main_th_2_0}, 
\begin{align*}
&
\Bigl|\sqrt{n}\Bigl\|f(\hat \Sigma_n)-f(\Sigma)\Bigr\|_{L_2({\mathbb P}_{\Sigma})} - \sigma_f(\Sigma)\Bigr|
\lesssim_{\rho}
\|f'\|_{{\rm Lip}_{\rho}}\|\Sigma\|^{1+\rho}
\biggl(\sqrt{\frac{{\bf r}(\Sigma)}{n}}\biggr)^{1+\rho}.
\end{align*}
Under the assumptions of Theorem \ref{main_th_2}, the following bound holds:
\begin{align*}
&
\Bigl|\sqrt{n}\Bigl\|T^{(2)}_f (X_1,\dots, X_n)-f(\Sigma)\Bigr\|_{L_2({\mathbb P}_{\Sigma})} - \sigma_f(\Sigma)\Bigr|
\\
&
\lesssim_{k,\rho, \gamma}
\|f'\|_{{\rm Lip}_{\gamma}}\|\Sigma\|^{1+\gamma}
\Bigl(\sqrt{\frac{{\bf r}(\Sigma)}{n}}\Bigr)^{\gamma} 
+ \|f^{(k)}\|_{{\rm Lip}_{\rho}} \|\Sigma\|^{k+\rho}
\sqrt{n}\biggl(\sqrt{\frac{{\bf r}(\Sigma)}{n}}\biggr)^{k+\rho}.
\end{align*}
\end{corollary}

It follows from the bound of the last corollary that, if ${\bf r}(\Sigma)\lesssim n^{\alpha}$ for some $\alpha\in (0,1)$ and $s=k+\rho >\frac{1}{1-\alpha},$
$k\geq 2, \rho\in (0,1],$
then the normalized $L_2$-risk 
$$\sqrt{n}\Bigl\|T^{(2)}_f (X_1,\dots, X_n)-f(\Sigma)\Bigr\|_{L_2({\mathbb P}_{\Sigma})}$$ 
of estimator $T^{(2)}_f (X_1,\dots, X_n)$ converges 
to $\sigma_f(\Sigma)$ as $n\to \infty.$ 
Moreover, it is possible to show that, under the same conditions, $\sqrt{n}(T^{(2)}_f (X_1,\dots, X_n)-f(\Sigma))$
converges in distribution to $N(0,\sigma_f^2(\Sigma)).$ 
For $k=1,$ the same claims hold true for the plug-in estimator $f(\hat \Sigma_n).$

To state the results on normal approximation, it will be convenient to use so called Wasserstein $\psi$-distances $W_{\psi}(\eta_1,\eta_2)$ between  
r.v. $\eta_1, \eta_2,$ or, more precisely, between their distributions. For a convex non-decreasing function $\psi: {\mathbb R}_+\mapsto {\mathbb R}_+$ with $\psi(0)=0,$ define 
\begin{align*}
W_{\psi}(\eta_1,\eta_2) := W_{\psi,{\mathbb P}}:= W_{L_{\psi}({\mathbb P})}:=\inf\Bigl\{\|\eta_1'-\eta_2'\|_{\psi}: \eta_1' \overset{d}{=}\eta_1, \eta_2' \overset{d}{=}\eta_2\Bigr\},
\end{align*}
where the infimum is taken over all r.v. $\eta_1', \eta_2'$ defined on the same probability space $(\Omega, \Sigma, {\mathbb P})$
such that $\eta_1' \overset{d}{=}\eta_1, \eta_2' \overset{d}{=}\eta_2.$
For $\psi(u)=u^p, u\geq 0, p\geq 1,$ the Wasserstein $\psi$-distance becomes Wasserstein $L_p$-distance and it will be denoted 
$W_p(\eta_1,\eta_2)=W_{p,{\mathbb P}}(\eta_1,\eta_2).$ We can also use this definition for $\psi=\psi_{\alpha}, \alpha\geq 1.$ Finally, it could be applied to the case 
of $\alpha\in (0,1)$
subject to a modification of the definition of the $\psi_{\alpha}$-norms discussed above (see \eqref{equiv_psi_alpha}).

The following corollary will be proved. 

\begin{corollary}
\label{cor_norm_approx}
(i) Under the assumptions of Theorem \ref{main_th_2_0}, 
\begin{align*}
&
W_{2, {\mathbb P}_{\Sigma}} \biggl(\frac{\sqrt{n}(f(\hat \Sigma_n)-f(\Sigma))}{\sigma_f(\Sigma)}, Z\biggr)
\lesssim_{\rho}
\frac{\|\Sigma\|^2 \|f'(\Sigma)\|^2}{\sigma_f^2(\Sigma)}\frac{1}{\sqrt{n}}
+ \frac{\|f'\|_{{\rm Lip}_{\rho}} \|\Sigma\|^{1+\rho}}{\sigma_f(\Sigma)}
\sqrt{n}\biggl(\sqrt{\frac{{\bf r}(\Sigma)}{n}}\biggr)^{1+\rho}.
\end{align*}
Moreover, under the assumptions of Corollary \ref{cor_main_th_2_0}, 
\begin{align*}
&
W_{\psi_{1/(1+\rho)}, {\mathbb P}_{\Sigma}} \biggl(\frac{\sqrt{n}(f(\hat \Sigma_n)-f(\Sigma))}{\sigma_f(\Sigma)}, Z\biggr)
\lesssim_{\rho} 
C\biggl(\frac{\|\Sigma\| \|f'(\Sigma)\|}{\sigma_f(\Sigma)}\biggr)\frac{1}{\sqrt{n}}+
\frac{\|f'\|_{{\rm Lip}_{\rho}} \|\Sigma\|^{k+\rho}}{\sigma_f(\Sigma)}\sqrt{n}\biggl(\sqrt{\frac{{\bf r}(\Sigma)}{n}}\biggr)^{k+\rho}.
\end{align*}

(ii) Under the assumptions of Theorem \ref{main_th_2}, 
\begin{align*}
&
W_{2,{\mathbb P}_{\Sigma}} \biggl(\frac{\sqrt{n}(T_f^{(2)}(X_1,\dots, X_n)-f(\Sigma))}{\sigma_f(\Sigma)}, Z\biggr)
\lesssim_{k,\rho, \gamma}
\frac{\|\Sigma\|^2 \|f'(\Sigma)\|^2}{\sigma_f^2(\Sigma)}\frac{1}{\sqrt{n}}
\\
&
+
\frac{\|f'\|_{{\rm Lip}_{\gamma}}\|\Sigma\|^{1+\gamma}}{\sigma_f(\Sigma)}
\Bigl(\sqrt{\frac{{\bf r}(\Sigma)}{n}}\Bigr)^{\gamma} 
+ \frac{\|f^{(k)}\|_{{\rm Lip}_{\rho}} \|\Sigma\|^{k+\rho}}{\sigma_f(\Sigma)}
\sqrt{n}\biggl(\sqrt{\frac{{\bf r}(\Sigma)}{n}}\biggr)^{k+\rho}.
\end{align*}
Moreover, under the assumptions of Corollary \ref{cor_main_th_2} for all $\beta\in [1/2,1),$ 
\begin{align*}
&
W_{\psi_{\beta}, {\mathbb P}_{\Sigma}} \biggl(\frac{\sqrt{n}(T_f^{(2)}(X_1,\dots, X_n)-f(\Sigma))}{\sigma_f(\Sigma)}, Z\biggr)
\lesssim_{k,\rho} 
C\biggl(\frac{\|\Sigma\| \|f'(\Sigma)\|}{\sigma_f(\Sigma)}\biggr)\frac{1}{\sqrt{n}}
\\
&
+
\frac{\|f'\|_{C^1} \|\Sigma\|^{1/\beta}}{\sigma_f(\Sigma)}
\biggl(\sqrt{\frac{{\bf r}(\Sigma)}{n}}\biggr)^{1/\beta-1}
+ \frac{\|f^{(k)}\|_{{\rm Lip}_{\rho}} \|\Sigma\|^{k+\rho}}{\sigma_f(\Sigma)}\sqrt{n}\biggl(\sqrt{\frac{{\bf r}(\Sigma)}{n}}\biggr)^{k+\rho}.
\end{align*}

In both claims $C\biggl(\frac{\|\Sigma\| \|f'(\Sigma)\|}{\sigma_f(\Sigma)}\biggr)>0$ is a constant depending only on $\frac{\|\Sigma\| \|f'(\Sigma)\|}{\sigma_f(\Sigma)}.$
\end{corollary}

\begin{remark}
\normalfont
Let 
\begin{align*}
{\mathcal S}_{f}(A,B,r,\sigma_0):=\Bigl\{\|\Sigma\|\leq A, \|f'(\Sigma)\|\leq B, {\bf r}(\Sigma)\leq r, \sigma_f(\Sigma)\geq \sigma_0\Bigr\}
\end{align*}
for $A>0, B>0, r>0, \sigma_0>0.$ It immediately follows from Corollary \ref{cor_norm_approx} that, if $\alpha\in (0,1)$ and $s=k+\rho>\frac{1}{1-\alpha},$ $k\geq 2, \rho\in (0,1],$
then, for all $\beta<1,$ 
\begin{align*}
\sup_{\Sigma\in {\mathcal S}_{f}(A,B,n^{\alpha},\sigma_0)}W_{\psi_{\beta}, {\mathbb P}_{\Sigma}} \biggl(\frac{\sqrt{n}(T_f^{(2)}(X_1,\dots, X_n)-f(\Sigma))}{\sigma_f(\Sigma)}, Z\biggr)\to 0\ {\rm as}\ n\to\infty.
\end{align*}
This easily implies the asymptotic normality of estimator $T_f^{(2)}(X_1,\dots, X_n)$ of $f(\Sigma)$ with $\sqrt{n}$-rate and limit variance $\sigma_f(\Sigma):$
\begin{align*}
\sup_{\Sigma\in {\mathcal S}_{f}(A,B,n^{\alpha},\sigma_0)} \sup_{x\in {\mathbb R}}\biggl|
{\mathbb P}_{\Sigma}\biggl\{\frac{\sqrt{n}(T_f^{(2)}(X_1,\dots, X_n)-f(\Sigma))}{\sigma_f(\Sigma)}\leq x\biggr\}
-{\mathbb P}\{Z\leq x\}
\biggr|\to 0\ {\rm as}\ n\to\infty,
\end{align*}
and, moreover, for any loss function $\psi$ dominated by $\psi_{\beta}$ for some $\beta<1,$ we have 
\begin{align*}
\sup_{\Sigma\in {\mathcal S}_{f}(A,B,n^{\alpha},\sigma_0)} \biggl|\biggl\|\frac{\sqrt{n}(T_f^{(2)}(X_1,\dots, X_n)-f(\Sigma))}{\sigma_f(\Sigma)}\biggr\|_{L_{\psi}({\mathbb P}_{\Sigma})}-\|Z\|_{\psi}\biggr|\to 0\ {\rm as}\ n\to\infty.
\end{align*}
For $k=1,$ similar results hold for the plug-in estimator $f(\hat \Sigma_n).$ 
\end{remark}

In principle, to estimate the value $f(\Sigma)$ of a functional $f$ at unknown covariance $\Sigma$ it should be enough for the functional to be smooth {\it locally} in a neighborhood of $\Sigma$
rather than on the whole space $L(E^{\ast}, E).$ We will now state such local versions of the results discussed above. 
Moreover, we will show that these local bounds are minimax optimal. 
Note that if $f$ is Lipschitz in a neighborhood $U$ of $\Sigma,$ then, by McShane-Whitney extension theorem, it could be extended to the whole space 
with preservation of its Lipschitz constant (if $f$ is, in addition, bounded in $U$ by a constant, its Lipschitz extension could be chosen to be bounded by the same constant).
Thus, without loss of generality, we will assume in what follows that $f$ is Lipschitz in $L(E^{\ast}, E)$ (and, if needed, also bounded), 
but its higher order smoothness holds only locally.

We will start with a local version of Theorem \ref{main_th_1}.

\begin{theorem}
\label{main_th_1_local}
Suppose $X_1,\dots, X_n$ are i.i.d. $N(0,\Sigma)$ with ${\bf r}(\Sigma)\lesssim n.$
Let $f: L(E^{\ast},E)\mapsto {\mathbb R}$ be Lipschitz and, for some $k\geq 2,$ let it be $k$ times Fr\'echet continuously differentiable in an open ball $U=B(\Sigma,\delta)$
of radius $\delta>0$ with the $k$-th derivative satisfying the H\"older condition with exponent $\rho\in (0,1]$ in this ball. 
Suppose also that $\delta\leq \|\Sigma\|$ and, for a sufficiently large constant $C\geq 1,$ 
$
\delta \geq C \|\Sigma\| \sqrt{\frac{{\bf r}(\Sigma)}{n}}.
$
Then, for $i=1,2$ and for all $p\geq 1,$
 \begin{align}
 \label{local_bd_th_1}
 &
 \nonumber
\Bigl\|T^{(i)}_f (X_1,\dots, X_n)-f(\Sigma)\Bigr\|_{L_p({\mathbb P}_{\Sigma})} 
\\
&
\lesssim_{k,\rho}
\|f\|_{{\rm Lip}}\|\Sigma\|
\biggl(\sqrt{\frac{p}{n}}+ \frac{p}{n}\biggr)
+ \|f^{(k)}\|_{{\rm Lip}_{\rho}(U)} \|\Sigma\|^{k+\rho}\biggl(\sqrt{\frac{{\bf r}(\Sigma)}{n}}\biggr)^{k+\rho}
+ \max_{2\leq j\leq k} \|f^{(j)}(\Sigma)\| \Bigl(\frac{\|\Sigma\|}{\sqrt{n}}\Bigr)^j \exp\Bigl\{- \frac{c n \delta^2}{\|\Sigma\|^2}\Bigr\}
\end{align}
with some constant $c>0.$
If $k=1,$ a similar bound holds for the plug-in estimator $f(\hat \Sigma_n):$
 \begin{align*}
 &
\Bigl\|f(\hat \Sigma_n)-f(\Sigma)\Bigr\|_{L_p({\mathbb P}_{\Sigma})} 
\lesssim_{\rho}
\|f\|_{{\rm Lip}}\|\Sigma\|
\biggl(\sqrt{\frac{p}{n}}+ \frac{p}{n}\biggr)
+ \|f^{(1)}\|_{{\rm Lip}_{\rho}(U)} \|\Sigma\|^{1+\rho}\biggl(\sqrt{\frac{{\bf r}(\Sigma)}{n}}\biggr)^{1+\rho}.
\end{align*}
\end{theorem}

\begin{remark}
\normalfont
Note that, under the assumptions of Theorem \ref{main_th_1_local}, the last term in the righthand side of bound \eqref{local_bd_th_1} could be further bounded in several useful 
ways. In particular, 
\begin{align*}
\max_{2\leq j\leq k} \|f^{(j)}(\Sigma)\| \Bigl(\frac{\|\Sigma\|}{\sqrt{n}}\Bigr)^j \exp\Bigl\{- \frac{c n \delta^2}{\|\Sigma\|^2}\Bigr\} \leq 
\max_{2\leq j\leq k} \|\Sigma\|^j \|f^{(j)}(\Sigma)\| n^{-1/2}\exp\Bigl\{- \frac{c n \delta^2}{\|\Sigma\|^2}\Bigr\}
\end{align*}
and
\begin{align*}
\max_{2\leq j\leq k} \|f^{(j)}(\Sigma)\| \Bigl(\frac{\|\Sigma\|}{\sqrt{n}}\Bigr)^j \exp\Bigl\{- \frac{c n \delta^2}{\|\Sigma\|^2}\Bigr\} \lesssim_k 
\max_{2\leq j\leq k} \|f^{(j)}(\Sigma)\| \frac{\|\Sigma\|^2} {n}.
\end{align*}
\end{remark}

We will also prove the following local version of Theorem \ref{main_th_2}. 

\begin{theorem}
\label{main_th_2_local}
Suppose $X_1,\dots, X_n$ are i.i.d. $N(0,\Sigma)$ with ${\bf r}(\Sigma)\lesssim n.$
Let $f$ be a $k$ times Fr\'echet continuously differentiable functional in the ball $U=B(\Sigma,\delta)$
for some $\delta>0$ and $k\geq 2$ with 
$\|f'\|_{C^{k-1+\rho}(U)}<\infty$ for some $\rho\in (0,1].$ Suppose also that $f$ is extended to a Lipschitz 
functional on $L(E^{\ast},E)$ with preservation of its Lipschitz constant $\|f\|_{\rm Lip}=\|f'\|_{L_{\infty}(U)}.$
Finally, suppose that $\delta\leq \|\Sigma\|\wedge 1$ and, for a sufficiently large constant $C\geq 1,$ 
$
\delta \geq C \|\Sigma\| \sqrt{\frac{{\bf r}(\Sigma)}{n}}.
$
Then, for all $p\geq 1$ and for all $\gamma \in (0,1],$
\begin{align*}
&
\Bigl\|T^{(2)}_f (X_1,\dots, X_n)-f(\Sigma)- \langle \hat \Sigma_n-\Sigma, f '(\Sigma)\rangle\Bigr\|_{L_p({\mathbb P}_{\Sigma})} 
\\
&
\lesssim_{k,\rho, \gamma}
\|f'\|_{C^{k-1+\rho}(U)}\biggl( 
\frac{\|\Sigma\|^{1+\gamma}}{\delta^{\gamma}}
\biggl(\sqrt{\frac{p}{n}}
\Bigl(\sqrt{\frac{{\bf r}(\Sigma)}{n}}\Bigr)^{\gamma} + 
\Bigl(\frac{p}{n}\Bigr)^{(1+\gamma)/2} + \Bigl(\frac{p}{n}\Bigr)^{1+\gamma}\biggr)
+ 
\|\Sigma\|^{k+\rho}
\biggl(\sqrt{\frac{{\bf r}(\Sigma)}{n}}\biggr)^{k+\rho}\biggr).
\end{align*}
In particular, this implies that, for all $\beta\in [1/2,1),$
\begin{align*}
&
\Bigl\|T^{(2)}_f (X_1,\dots, X_n)-f(\Sigma)- \langle \hat \Sigma_n-\Sigma, f '(\Sigma)\rangle\Bigr\|_{L_{\psi_{\beta}}({\mathbb P}_{\Sigma})} 
\\
&
\lesssim_{k,\rho, \beta}
\|f'\|_{C^{k-1+\rho}(U)}
\biggl( 
\frac{\|\Sigma\|^{1/\beta}}{\delta^{1/\beta-1}}
\frac{1}{\sqrt{n}}
\Bigl(\sqrt{\frac{{\bf r}(\Sigma)}{n}}\Bigr)^{1/\beta-1} + 
\|\Sigma\|^{k+\rho}
\biggl(\sqrt{\frac{{\bf r}(\Sigma)}{n}}\biggr)^{k+\rho}\biggr).
\end{align*}
\end{theorem}

We will now state a simple corollary of Theorem \ref{main_th_1_local}, providing a uniform version of the bound of this theorem in a set of covariance operators with bounded effective rank. 
Namely, for $a>0$ and $r\geq 1,$ let ${\mathcal S}(a, r)$ be the set of covariance operators $\Sigma: E^{\ast}\mapsto E$
with $\|\Sigma\|\leq a$ and ${\bf r}(\Sigma)\leq r.$ Recall the definition of $C^{s,a}$-norms from Section \ref{sec:prelim}.

\begin{corollary}
\label{cor_unif_local}
Let $a>0, r\geq 1$ and $\Sigma_0\in {\mathcal S}(a, r).$ 
Suppose also that, for a sufficiently large constant $C\geq 1,$ 
\begin{align}
\label{cond_r}
C a\sqrt{\frac{r}{n}} < \delta \leq a\wedge 1.  
\end{align}
Let $U:=B(\Sigma_0, 2\delta)$ and let $s=k+\rho$ for some $k\geq 2$ and $\rho\in (0,1].$
Then, for $i=1,2$ and for all $p\geq 1,$
 \begin{align}
 \label{uni_up}
 &
\sup_{\|f\|_{C^{s,a}(U)}\leq 1}\sup_{\Sigma\in {\mathcal S}(a, r), \|\Sigma-\Sigma_0\|<\delta}\Bigl\|T^{(i)}_f (X_1,\dots, X_n)-f(\Sigma)\Bigr\|_{L_p({\mathbb P}_{\Sigma})} 
\lesssim_{s}
\biggl(\biggl(\sqrt{\frac{p}{n}}\vee \frac{p}{n}\biggr)
+ \biggl(\sqrt{\frac{r}{n}}\biggr)^{s}\biggr)\wedge 1.
\end{align}
For $k=1,$ a similar bound holds for the plug-in estimator $f(\hat \Sigma_n).$
\end{corollary}

For $p=2,$ the bound simplifies as follows:
\begin{align}
\label{uni_up-2} 
&
\sup_{\|f\|_{C^{s,a}(U)}\leq 1}\sup_{\Sigma\in {\mathcal S}(a, r), \|\Sigma-\Sigma_0\|<\delta}\Bigl\|T^{(i)}_f (X_1,\dots, X_n)-f(\Sigma)\Bigr\|_{L_2({\mathbb P}_{\Sigma})} 
\lesssim_{s}
\biggl(\frac{1}{\sqrt{n}}
+ \biggl(\sqrt{\frac{r}{n}}\biggr)^{s}\biggr)
\end{align}
(since, for $p=2$ and under condition \eqref{cond_r}, both terms in the sum in the right hand side of \eqref{uni_up} are bounded from above by a constant, there is no need to take the minimum with $1$).

Bounds \eqref{uni_up} and \eqref{uni_up-2} provide the size of the {\it maximal risk} of estimators $T^{(i)}_f (X_1,\dots, X_n)$ in the class of covariance operators 
${\mathcal S}(a,r)\cap \{\Sigma: \|\Sigma-\Sigma_0\|<\delta\}$ and 
in the class of $s$-smooth functionals $\{f: \|f\|_{C^{s,a}(U)}\leq 1\}.$
In the case of separable Hilbert space $E={\mathbb H},$ it is possible to show that bound \eqref{uni_up-2} is minimax optimal 
locally around {\it spiked covariance operators} $\Sigma_0,$ implying the optimal dependence of the bound on the sample size $n,$ the effective rank $r$ and the smoothness parameter $s.$

A covariance operator $\Sigma_0$ will be called a spiked covariance 
of rank $d$ if $\sigma(\Sigma):= \{\lambda, \mu, 0\}$ with $\lambda>\mu>0$ and with eigenvalue $\lambda$ being of multiplicity $1$ and eigenvalue $\mu$
being of multiplicity $d-1.$

\begin{theorem}
\label{loc_min_max}
Let $\Sigma_0\in {\mathcal S}(a,r)$ be a spiked covariance operator of rank $[r]$ with nonzero eigenvalues $\lambda = \gamma_1 a$ and $\mu=\gamma_2 a,$
where $0<\gamma_2< \gamma_1<1$ are numerical constants. 
Let $\kappa:= \gamma_2\wedge (\gamma_1-\gamma_2)\wedge (1-\gamma_1)$ and 
suppose that $\delta$ satisfies the condition  
\begin{align*}
c_1 \gamma_1 a\sqrt{\frac{r}{n}} < \delta \leq c_2 \kappa a\wedge 1
\end{align*}
with sufficiently large $c_1\geq 1$ and sufficiently small $c_2>0.$
Let $U:=B(\Sigma_0, 2\delta)$ and let $s>0.$
Then 
 \begin{align*}
 &
\sup_{\|f\|_{C^{s,a}(U)}\leq 1}\inf_{T}\sup_{\Sigma\in {\mathcal S}(a, r), \|\Sigma-\Sigma_0\|<\delta}\Bigl\|T(X_1,\dots, X_n)-f(\Sigma)\Bigr\|_{L_2({\mathbb P}_{\Sigma})} 
\gtrsim \biggl(\frac{\gamma_1}{\sqrt{n}}\vee \gamma_1^{s/2} \gamma_2^{s/2} \biggl(\sqrt{\frac{r}{n}}\biggr)^{s}\biggr),
\end{align*}
where the infimum is taken over all the estimators $T(X_1,\dots, X_n).$
\end{theorem}

Finally, we will state another local minimax lower bound that, along with the results of corollaries \ref{cor_main_th_2'''}  and \ref{cor_norm_approx}, imply the asymptotic efficiency 
of estimator $T^{(2)}_f(X_1,\dots, X_n)$ of functional $f(\Sigma).$  A similar result has been already stated in \cite{Koltchinskii_Zhilova_19} and its proof relies on van Trees inequality (see \cite{Koltchinskii_Nickl, Koltchinskii_2017} for similar arguments). 
As in the case of Theorem \ref{loc_min_max}, the result will be stated in the case when $E={\mathbb H}$ is a separable Hilbert space.
In this case, it is easy to check that 
\begin{align*}
\sigma_f^2(\Sigma) = 2\|\Sigma^{1/2} f'(\Sigma) \Sigma^{1/2}\|_2^2.
\end{align*}

\begin{theorem}
Let $a>1, r\geq 1$ and let $\Sigma_0$ be a spiked covariance operator of rank $[r]$ with its non-zero eigenvalues 
belonging to an interval $[a^{-1}+\bar \delta, a-\bar \delta]$ for some $\bar \delta>0.$  
Let $U:= B(\Sigma_0,2\bar \delta)$ and suppose $f\in C^{1}(U).$ 
Let 
\begin{align*}
\omega_{f'}(\Sigma_0, \delta):= \sup_{\|\Sigma-\Sigma_0\|<\delta} \|f'(\Sigma)-f'(\Sigma_0)\|, \delta \leq \bar \delta
\end{align*}
be a local continuity modulus of $f'$ at $\Sigma_0.$ For all $\beta>2,$ there exists a constant $D_{\beta}>0$ such that, for all $\delta<\bar \delta,$
\begin{align*}
&
\inf_{T_n} \sup_{\Sigma\in {\mathcal S}(a,r),\|\Sigma-\Sigma_0\|<\delta}\frac{\sqrt{n}\|T_n(X_1,\dots, X_n)-f(\Sigma)\|_{L_2({\mathbb P}_{\Sigma})}}{\sigma_f(\Sigma)}
\geq 
1-D_{\beta}\biggl[\frac{a \omega_{f'}(\Sigma_0,\delta)}{\sigma_f(\Sigma_0)}+ a^{\beta} \delta + \frac{a^2}{\delta^2 n}\biggr],
\end{align*}
where the infimum is taken over all the estimators $T_n(X_1,\dots, X_n).$
\end{theorem}

\begin{remark}
\normalfont
This bound could be viewed as a non-asymptotic version of H\`ajek-LeCam local asymptotic minimax theorem. If $\sigma_f(\Sigma_0)$ is bounded away from zero and $\omega_{f'}(\Sigma_0,\delta)\to 0$ as $\delta\to 0,$
it implies that 
\begin{align*}
\lim_{c\to \infty}\liminf_{n\to\infty} \inf_{T_n} \sup_{\Sigma\in {\mathcal S}(a,r),\|\Sigma-\Sigma_0\|<\frac{c}{\sqrt{n}}}
\frac{\sqrt{n}\|T_n(X_1,\dots, X_n)-f(\Sigma)\|_{L_2({\mathbb P}_{\Sigma})}}{\sigma_f(\Sigma)}\geq 1.
\end{align*}
It is easy to deduce from Corollary \ref{cor_main_th_2'''} that this asymptotic minimax lower bound is attained for estimator 
$T^{(2)}_f(X_1,\dots, X_n)$ (or, for $s\leq 2,$
for the plug-in estimator $f(\hat \Sigma_n)$) in the class of covariances 
${\mathcal S}(a,r)$ for a $C^s$-smooth functional $f$ under the assumptions that $r\leq n^{\alpha}$ for some $\alpha\in (0,1)$ and $s>\frac{1}{1-\alpha}.$
\end{remark}

\begin{remark}
\normalfont
Let ${\mathbb H}$ be a separable Hilbert space and let $\Sigma_0$ be a covariance operator with eigenvalues $\|\Sigma_0\|=\lambda_1=\dots=\lambda_l>\lambda_{l+1}\geq \lambda_{l+2}\geq \dots.$
Let $g_l:=\lambda_l-\lambda_{l+1}$ be the spectral gap of the top eigenvalue $\|\Sigma_0\|$ of $\Sigma_0.$ Let $U:=B(\Sigma_0, \delta)$ for  $\delta<g_l/8.$ For $\Sigma\in U,$ the orthogonal projection $P(\Sigma)$ onto the subspace 
generated by the eigenvectors corresponding to the first $l$ eigenvalues of $\Sigma$ is well defined and, moreover, $P(\Sigma)$ is a $C^{\infty}$-function in $U$ with $\|P^{(k)}\|_{L_{\infty}(U)}\lesssim_k g_l^{-k}, k=0,1,2, \dots$  (see Lemma \ref{PSigma_smooth} and Remark \ref{PSigma_smooth_R}). For a given nuclear operator $B$ with nuclear norm $\|B\|_1\leq 1,$
define $f(\Sigma):= \langle P(\Sigma), B\rangle.$ The goal is to estimate $f(\Sigma)$ based on i.i.d. observations $X_1,\dots, X_n\sim N(0,\Sigma).$ The problems of this nature 
are of importance in high-dimensional principal component analysis, where it is of interest to estimate bilinear forms $\langle P(\Sigma)u,v\rangle$ (in particular, the matrix entries in some basis) of spectral projections of covariance operator. When $l=1$ (that is, when the top eigenvalue of $\Sigma$ is simple), the problem can be phrased as estimation of linear functionals of principal components and rather specialized bias reduction method for such functionals was developed and studied in \cite{Koltchinskii_Lounici_bilinear, Koltchinskii_Nickl}. This method yields
asymptotically efficient estimators in classes of covariances with ${\bf r}(\Sigma)=o(n),$ but it is not known how to extend this approach to the case of spectral projections corresponding to multiple 
eigenvalues. Since $f(\Sigma)$ is a $C^{\infty}$-functional locally in the neighborhood $U$ of $\Sigma_0,$ one can try to use estimators $T_{f,k}^{(i)}(X_1,\dots, X_n), i=1,2$ for an arbitrary 
number $k\geq 2$ of plug-in estimators in the linear combination. Suppose $\Sigma\in U$ and ${\bf r}(\Sigma)\leq r$ for some $r\geq 1.$ 
Denote $\gamma:= \frac{\|\Sigma_0\|}{g_l}.$ Suppose also that, for a large enough constant $C>0,$
$C\gamma \sqrt{\frac{r}{n}}\leq 1.$ 
Using the bounds of Lemma \ref{PSigma_smooth} and Remark \ref{PSigma_smooth_R}, it is not hard to deduce from Theorem \ref{main_th_1_local} with $s=k+1$ 
the following bound that holds for all $k\geq 2,$ for $i=1,2$ and for $p\geq 1:$
\begin{align*}
 &
\Bigl\|T^{(i)}_f (X_1,\dots, X_n)-f(\Sigma)\Bigr\|_{L_p({\mathbb P}_{\Sigma})} 
\lesssim_k \gamma \Bigl(\sqrt{\frac{p}{n}}\vee \frac{p}{n}\Bigr) + \Bigl(\gamma \sqrt{\frac{r}{n}}\Bigr)^{k+1}.
\end{align*}
Further development of this approach (that might include, for instance, an adaptive choice of $k$ to achieve asymptotic efficiency) is beyond the scope of the current paper. 
\end{remark}

\section{Bias reduction via linear aggregation of plug-in estimators and jackknife method}
\label{jackknife}

In this section, we summarize several facts (many of them known) concerning an approach to bias reduction based on linear aggregation of plug-in estimators with different sample sizes and, in particular, a jackknife method 
of bias reduction in the problem of estimation of a smooth functional $f({\mathbb E}Y)$ of unknown mean ${\mathbb E}Y$ of a random variable $Y$ in a Banach space $F$ based on i.i.d. observations $Y_1,\dots, Y_n$ of $Y.$

As in Section \ref{sec:intro}, let $k\geq 1$ and $n/c\leq n_1<n_2<\dots<n_k\leq n$ for some $c>1.$ 
Let $\bar Y_n := \frac{Y_1+\dots+Y_n}{n}$ and define 
\begin{align*}
\hat T_f(Y_1,\dots, Y_n) := \sum_{j=1}^k C_j f(\bar Y_{n_j}),
\end{align*}
where $\{C_j: 1\leq j\leq k\}$ are defined by \eqref{choice_weight} and Assumption \ref{assume_on_C_j} holds. 
With these definitions, it is known (see \cite{Jiao}) that the following proposition holds:

\begin{proposition}
\label{prop_C_j}
The coefficients $C_j, j=1,\dots, k$ satisfy the following properties:
\begin{enumerate}[(i)]
\item $\sum_{j=1}^k C_j=1;$
\item $\sum_{j=1}^k \frac{C_j}{n_j^l}=0$ for $l=1,\dots, k-1.$
\end{enumerate}
\end{proposition}

The next proposition provides a bound on the bias of estimator $\hat T_f(Y_1,\dots, Y_n).$ 

\begin{proposition}
\label{bias_jack}
Let $f:F\mapsto {\mathbb R}$ be $k$ times continuously differentiable for some $k\geq 2$ with $f^{(k)}\in {\rm Lip}_{\rho}(F)$ for some $\rho \in (0,1].$
Let $s:=k+\rho.$
Then 
\begin{align*}
|{\mathbb E}\hat T_f(Y_1,\dots, Y_n)- f({\mathbb E} Y)| \lesssim_s 
\|f^{(k)}\|_{{\rm Lip}_{\rho}(F)}\max_{1\leq j\leq k}{\mathbb E} \|\bar Y_{n_j}-{\mathbb E}Y\|^s.
\end{align*}
\end{proposition}

A couple of simple bounds on the remainder of Taylor expansion will be used in the proof of Proposition \ref{bias_jack} and throughout the rest of the paper
(similar facts were used in \cite{Koltchinskii_2017, Koltchinskii_Zhilova} and they can be easily derived from standard formulas for the remainder of Taylor expansion, see, e.g., \cite{Cartan}, Chapter 1, Section 5.6). 
For a Fr\'echet differentiable functional $f: F \mapsto {\mathbb R},$ denote by 
\begin{align*}
S_f(x;h):= f(x+h)-f(x)- \langle h, f'(x)\rangle, x, h\in F
\end{align*} 
the remainder of the first order Taylor expansion of $f$ at point $x\in F.$

\begin{proposition}
\label{Taylor_rem}
For a Fr\'echet continuously differentiable functional $f: F \mapsto {\mathbb R}$ and for all $x, h, h'\in F,$
\begin{align*}
|S_f(x,h)|\lesssim \|f'\|_{{\rm Lip}_{\rho}(F)} \|h\|^{1+\rho}
\end{align*}
and 
\begin{align*}
|S_f(x,h)-S_f(x,h')|\lesssim \|f'\|_{{\rm Lip}_{\rho}(F)} (\|h\|^{\rho} \vee\|h'\|^{\rho})\|h-h'\|.
\end{align*}
The same bounds hold if $f$ is Fr\'echet continuously differentiable in a ball $B(x,\delta):=\{y: \|y-x\|<\delta\}$
and $\|h\|, \|h'\|<\delta.$
\end{proposition}

Similarly, for a $k$ times Fr\'echet differentiable functional $f: F \mapsto {\mathbb R},$
denote by 
\begin{align*}
S_f^{(k)}(x;h):= f(x+h)- \sum_{j=0}^k \frac{f^{(j)}(x)[h, \overset{j}{\dots}, h]}{j!}, x, h\in F
\end{align*} 
the remainder of the $k$-th order Taylor expansion of $f.$

\begin{proposition}
\label{Taylor_rem_high}
For a $k$ times Fr\'echet continuously differentiable functional $f: F \mapsto {\mathbb R}$ and for all $x, h \in F,$
\begin{align*}
|S_f^{(k)}(x,h)|\leq \frac{\|f^{(k)}\|_{{\rm Lip}_{\rho}(F)}}{k!} \|h\|^{k+\rho}.
\end{align*}
The same bound holds if $f$ is $k$ times Fr\'echet continuously differentiable in a ball $B(x,\delta)$
and $\|h\|<\delta.$
\end{proposition}

The proof of Proposition \ref{bias_jack} is based on the following simple lemma.

\begin{lemma}
\label{bias_control}
Let $F$ be a Banach space and 
let $f:F\mapsto {\mathbb R}$ be $k$ times Fr\'echet continuously differentiable for some $k\geq 2$ with $f^{(k)}\in {\rm Lip}_{\rho}(F)$ for some $\rho \in (0,1].$
Let $s:=k+\rho.$
Let $Y, Y_1,\dots, Y_n$ be i.i.d. r.v. in $F$ with 
distribution $P$ and let $\bar Y = \bar Y_n := \frac{Y_1+\dots+Y_n}{n}.$ Suppose ${\mathbb E}\|Y\|^s<\infty.$ Then 
\begin{align*}
{\mathbb E} f(\bar Y) - f({\mathbb E}Y)= \sum_{l=1}^{k-1} \frac{\beta_{l,k}(P)}{n^l}+ R,
\end{align*}
where the coefficients $\beta_{l,k}(P)$ depend on $P$ and $f$ (but not on $n$) and 
\begin{align*}
|R| \lesssim_s \|f^{(k)}\|_{{\rm Lip}_{\rho}(F)}{\mathbb E} \|\bar Y-{\mathbb E}Y\|^s.
\end{align*}
\end{lemma}

\begin{proof}
Let $M$ be a bounded symmetric $k$-linear form on a Banach space $F$ and let $\xi, \xi_1, \xi_2, \dots$ be i.i.d. r.v. in $F$ with ${\mathbb E}\xi=0,$ ${\mathbb E}\|\xi\|^k<\infty.$
Let $\bar \xi := \frac{\xi_1+\dots+\xi_n}{n}.$ Then 
\begin{align*}
{\mathbb E} M[\bar \xi, \overset{k}{\dots}, \bar \xi] = n^{-k} \sum_{1\leq i_1,\dots, i_k\leq n} {\mathbb E} M[\xi_{i_1},\dots, \xi_{i_k}].
\end{align*}
Suppose there are $m$ distinct indices $1\leq l_1<\dots <l_m\leq n$ among $i_1,\dots, i_k$ repeated with 
multiplicities $k_1, \dots, k_m,$ where $1\leq m \leq n,$ $k_1,\dots, k_m\geq 1,$ $k_1+\dots +k_m=k.$
Note that, if $k_j=1$ for some $j=1,\dots, m,$ then ${\mathbb E} M[\xi_{i_1},\dots, \xi_{i_k}]=0.$ 
Otherwise, $k_1\geq 2, \dots, k_m\geq 2$ and we have 
\begin{align*}
{\mathbb E} M[\xi_{i_1},\dots, \xi_{i_k}]= {\mathbb E}M[\xi_1, \overset{k_1}{\dots} \xi_1, \dots, \xi_m, \overset{k_m}{\dots} \xi_m]=: b_{k_1,\dots, k_m}.
\end{align*}
Denote 
\begin{align*}
B_m := \sum_{k_1+\dots+k_m=k, k_j\geq 2, j=1,\dots, m} b_{k_1,\dots, k_m}.
\end{align*}
Clearly, if $k_j\geq 2, j=1,\dots, m,$ then $m\leq k/2.$ Since there are ${n\choose m}$ choices of $l_1,\dots, l_m$, we have 
\begin{align*}
{\mathbb E} M[\bar \xi, \overset{k}{\dots}, \bar \xi] = \sum_{1\leq m\leq k/2} \frac{{n\choose m}}{n^k} B_m.
\end{align*}
Writing $m! {n\choose m}= \sum_{i=1}^m c_{m,i} n^i,$ we get 
\begin{align*}
{\mathbb E} M[\bar \xi, \overset{k}{\dots}, \bar \xi] = \sum_{1\leq m\leq k/2} \frac{B_m}{m!} \sum_{i=1}^m c_{m,i} n^{-(k-i)}
= \sum_{1\leq i\leq k/2} d_{i,k} n^{-(k-i)},
\end{align*}
where 
$$
d_{i,k}= d_{i,k}(M,\xi):=\sum_{i\leq m\leq k/2} c_{m,i}\frac{B_m}{m!}
$$
depends only on the multilinear form $M$ and the distribution of r.v. $\xi$ (but not on $n$). 

By the Taylor expansion, we have 
\begin{align*}
f(x+h) = \sum_{j=0}^k \frac{f^{(j)}(x)[h,\dots, h]}{j!} + S_f^{(k)}(x;h)
\end{align*}
with the following bound on the remainder (see Proposition \ref{Taylor_rem_high}):
\begin{align*}
|S_f^{(k)}(x;h)| \lesssim_s \|f^{(k)}\|_{{\rm Lip}_{\rho}(F)} \|h\|^s. 
\end{align*}
Since ${\mathbb E}f'({\mathbb E}Y)[\bar Y-{\mathbb E}Y]=0,$ this implies that, for all $k\geq 2,$
\begin{align*}
{\mathbb E} f(\bar Y) -f({\mathbb E}Y) = \sum_{j=2}^k \frac{{\mathbb E}f^{(j)}({\mathbb E}Y)[\bar Y-{\mathbb E}Y,\dots, \bar Y-{\mathbb E}Y]}{j!}
+{\mathbb E} S_f^{(k)}({\mathbb E}Y; \bar Y-{\mathbb E} Y)
\end{align*}
with 
\begin{align*}
|{\mathbb E} S_f^{(k)}({\mathbb E}Y; \bar Y-{\mathbb E} Y)| \lesssim_s \|f^{(k)}\|_{{\rm Lip}_{\rho}(F)}{\mathbb E} \|\bar Y-{\mathbb E}Y\|^s.
\end{align*}
Let $\xi=Y-{\mathbb E}Y,\ \xi_j:= Y_j-{\mathbb E}Y, j=1,\dots, n.$ Then $\bar Y-{\mathbb E}Y= \bar \xi$ and, for all $k\geq 2,$ 
\begin{align*}
{\mathbb E} f(\bar Y) -f({\mathbb E}Y) &= \sum_{j=2}^k \frac{{\mathbb E}f^{(j)}({\mathbb E}Y)[\bar \xi,\dots, \bar \xi]}{j!}
+{\mathbb E} S_f^{(k)}({\mathbb E}Y; \bar Y-{\mathbb E}Y) 
\\
&
= \sum_{j=2}^k \sum_{1\leq i\leq j/2} \tilde d_{i,j} n^{-(j-i)}+{\mathbb E} S_f^{(k)}({\mathbb E}Y; \bar Y-{\mathbb E}Y) ,
\end{align*}
where $\tilde d_{i,j} := d_{i,j}(f^{(j)}({\mathbb E Y}), \xi).$ Therefore,
\begin{align*}
{\mathbb E} f(\bar Y) -f({\mathbb E}Y) = \sum_{l=1}^{k-1} \frac{\beta_{l,k}(P)}{n^{l}}+
{\mathbb E} S_f^{(k)}({\mathbb E}Y; \bar Y-{\mathbb E}Y),
\end{align*}
where the coefficients 
\begin{align*}
\beta_{l,k}(P):= \sum_{(i,j)\in {\mathcal J}_{l,k}}\tilde d_{i,j},\ {\mathcal J}_{l,k} :=  \{(i,j): 2\leq j\leq k, 1\leq i\leq j/2, j-i=l\}
\end{align*} 
depend on the distribution $P$ of $Y$ (but not on $n$) and 
\begin{align*}
|{\mathbb E} S_f^{(k)}({\mathbb E}Y; \bar Y-{\mathbb E} Y)| \lesssim_s \|f^{(k)}\|_{{\rm Lip}_{\rho}(F)}{\mathbb E} \|\bar Y-{\mathbb E}Y\|^s.
\end{align*}

\end{proof}

We will now prove Proposition \ref{bias_jack}.

\begin{proof}
Using Proposition \ref{prop_C_j} (i) and (ii), we get
\begin{align*}
&
{\mathbb E}\hat T_f(Y_1,\dots, Y_n)- f({\mathbb E} Y) = \sum_{j=1}^k C_j ({\mathbb E} f(\bar Y_{n_j})-f({\mathbb E}Y))
\\
&
=\sum_{j=1}^k C_j \sum_{l=1}^{k-1} \frac{\beta_{l,k}(P)}{n_j^{l}} +\sum_{j=1}^k C_j {\mathbb E} S_f^{(k)}({\mathbb E}Y; \bar Y_{n_j}-{\mathbb E}Y)
\\
&
=:
\sum_{l=1}^{k-1} \beta_{l,k}(P)\sum_{j=1}^k\frac{C_j}{n_j^{l}}+ R = R,
\end{align*}
and, by Assumption \ref{assume_on_C_j}, 
\begin{align*}
|R| \leq \sum_{j=1}^k |C_j| \max_{1\leq j\leq k}|{\mathbb E} S_f^{(k)}({\mathbb E}Y; \bar Y_{n_j}-{\mathbb E}Y)|\lesssim_s \|f^{(k)}\|_{{\rm Lip}_{\rho}(F)}\max_{1\leq j\leq k}
{\mathbb E} \|\bar Y_{n_j}-{\mathbb E}Y\|^s.
\end{align*}

\end{proof}

Let ${\mathcal F}_{{\rm sym}}$ denote the $\sigma$-algebra generated by all symmetric functions $\psi(Y_1,\dots, Y_n)$
of r.v. $Y_1,\dots, Y_n$ and define 
\begin{align*}
\check T_f(Y_1,\dots, Y_n) := {\mathbb E}\Bigl(\hat T_f(Y_1,\dots, Y_n)\Bigl|{\mathcal F}_{{\rm sym}}\Bigr). 
\end{align*}
Clearly, ${\mathbb E}\check T_f(Y_1,\dots, Y_n)={\mathbb E}\hat T_f(Y_1,\dots, Y_n),$ implying that the biases of these two estimators 
are the same, and the bound of Proposition \ref{bias_jack} holds for estimator $\check T_f(Y_1,\dots, Y_n),$ too:

\begin{proposition}
\label{bias_jack_one}
Let $f:F\mapsto {\mathbb R}$ be $k$ times continuously differentiable for some $k\geq 2$ with $f^{(k)}\in {\rm Lip}_{\rho}(F)$ for some $\rho \in (0,1]$ and let $s=k+\rho.$
Then 
\begin{align*}
|{\mathbb E}\check T_f(Y_1,\dots, Y_n)- f({\mathbb E} Y)| \lesssim_s 
\|f^{(k)}\|_{{\rm Lip}_{\rho}(F)}\max_{1\leq j\leq k}{\mathbb E} \|\bar Y_{n_j}-{\mathbb E}Y\|^s.
\end{align*}
\end{proposition}

It is easy to see that estimator $\check T_f(Y_1,\dots, Y_n)$ can be represented as a linear combination of $U$-statistics.
Indeed, recall that, if $h: F\times \dots \times F\mapsto F_1$ is a symmetric function of $m\leq n$ variables in a Banach space $F$ with values in a Banach space $F_1,$ then  
\begin{align*}
U_n h = (U_n h)(Y_1,\dots, Y_n):= \frac{1}{{n\choose m}} \sum_{1\leq j_1<\dots<j_m\leq n} h(Y_{j_1}, \dots, Y_{j_m}) 
\end{align*}
is the $U$-statistic of order $m$ with kernel $h.$ 
Then, for all $1\leq j_1<\dots<j_m\leq n,$ 
\begin{align*}
(U_n h)(Y_1,\dots, Y_n)={\mathbb E}(h(Y_{j_1},\dots, Y_{j_m})|{\mathcal F}_{{\rm sym}}).
\end{align*}
Let now $h_j (Y_1,\dots, Y_{n_j}):= f(\bar Y_{n_j}).$ Then
\begin{align}
\label{U-Stat}
&
\check T_f(Y_1,\dots, Y_n) = \sum_{j=1}^k C_j {\mathbb E}(h_j (Y_1,\dots, Y_{n_j})|{\mathcal F}_{\rm sym})
= \sum_{j=1}^k C_j (U_n h_j) (Y_1,\dots, Y_n) = \sum_{j=1}^k C_j U_n f(\bar Y_{n_j}).
\end{align}

\begin{remark}
\normalfont
Note that to compute estimator $\check T_f(Y_1,\dots, Y_n),$ one needs to compute the value of $f$ for $\sum_{1\leq j\leq k}{n\choose n_j}$ sample means. 
Since we choose $n_j\asymp n, j=1,\dots, k,$ this number would grow exponentially with $n.$ To overcome this difficulty, one can use a Monte Carlo approximation of 
the $U$-statistics $U_n f(\bar Y_{n_j})$ involved in the formula for $\check T_f(Y_1,\dots, Y_n).$ To preserve the convergence rates (and asymptotic efficiency properties)
of the estimator, it is enough to approximate it with the accuracy $o(n^{-1/2}).$ To this end, for each $j=1,\dots, k,$ one can sample independently at random $N$ subsets 
of cardinality $n_j$ of the sample $Y_1,\dots, Y_n$ and approximate $U_n f(\bar Y_{n_j})$ by the average value of $f$ on the sample means for these $N$ subsets. 
If $n=o(N),$ the approximation error would be of the order $O(N^{-1/2})=o(n^{-1/2})$ and this method would require computing the value of $f$ just for $k N$ sample means. 
\end{remark}

We will need a couple of simple (and mostly known) facts. We give their proofs for completeness.  

\begin{proposition}
\label{bar_Y_m_Y_n}
Let $h(Y_1,\dots, Y_m):= \bar Y_m.$ Then, for all $m\leq n,$ 
\begin{align*}
U_n \bar Y_m = \bar Y_n.
\end{align*}
\end{proposition}

\begin{proof}
Indeed, 
$
U_n \bar Y_m = m^{-1}\sum_{j=1}^m U_n Y_j = \frac{m\bar Y_n}{m}=\bar Y_n.
$

\end{proof}

Let $h: F\times \dots \times F\mapsto {\mathbb R}$ be a symmetric function of $m$ variables.

\begin{proposition}
\label{U_L_p}
Let $p\geq 1$ and assume that ${\mathbb E}|h(Y_1,\dots, Y_m)|^p<\infty.$ Then
\begin{align*}
\|(U_n h) (Y_1,\dots, Y_n) - {\mathbb E}h(Y_1,\dots, Y_m)\|_{L_p} \leq \|h(Y_1,\dots, Y_m)- {\mathbb E}h(Y_1,\dots, Y_m)\|_{L_p}.
\end{align*}
\end{proposition}

\begin{proof}
By Jensen's inequality,
\begin{align*}
&
{\mathbb E}\Bigl|(U_n h) (Y_1,\dots, Y_n) - {\mathbb E}h(Y_1,\dots, Y_m)\Bigr|^p
= {\mathbb E}\Bigl| {\mathbb E}((h(Y_1,\dots, Y_m) - {\mathbb E}h(Y_1,\dots, Y_m))|{\mathcal F}_{{\rm sym}})\Bigr|^p
\\
&
\leq 
 {\mathbb E}{\mathbb E}\Bigl(\Bigl|h(Y_1,\dots, Y_m) - {\mathbb E}h(Y_1,\dots, Y_m)\Bigr|^p|{\mathcal F}_{{\rm sym}}\Bigr)
 = {\mathbb E}\Bigl|h(Y_1,\dots, Y_m) - {\mathbb E}h(Y_1,\dots, Y_m)\Bigr|^p,
\end{align*}
implying the claim.

\end{proof}

\begin{proposition}
\label{U_L_p_taylor}
Let $f: F\mapsto {\mathbb R}$ be a continuously differentiable functional. 
Let $p\geq 1$ and suppose that 
\begin{align*}
{\mathbb E}|S_f({\mathbb E}Y; \bar Y_m- {\mathbb E}Y)|^p<\infty.
\end{align*}
Then
\begin{align*}
&
\Bigl\|U_n f(\bar Y_m) - {\mathbb E} U_n f(\bar Y_m) - \langle \bar Y_n- {\mathbb E}Y, f'({\mathbb E}Y)\rangle\Bigr\|_{L_p}
\leq 
\Bigl\|S_f({\mathbb E}Y; \bar Y_m- {\mathbb E}Y)- {\mathbb E} S_f({\mathbb E}Y; \bar Y_m- {\mathbb E}Y)\Bigr\|_{L_p}.
\end{align*}.
\end{proposition}

\begin{proof}
Let $h(Y_1,\dots, Y_m):= f(\bar Y_m).$ 
Then we have 
\begin{align*}
&
h(Y_1,\dots, Y_m)-{\mathbb E} h(Y_1,\dots, Y_m)
\\
&
= \langle \bar Y_m- {\mathbb E}Y, f'({\mathbb E}Y)\rangle + S_f({\mathbb E}Y; \bar Y_m- {\mathbb E}Y)-{\mathbb E} S_f({\mathbb E}Y; \bar Y_m- {\mathbb E}Y)\end{align*}
and 
\begin{align*}
&
(U_n h)(Y_1,\dots, Y_n)-{\mathbb E} (U_n h)(Y_1,\dots, Y_n)
\\
&
= \langle U_n \bar Y_m- {\mathbb E}Y, f'({\mathbb E}Y)\rangle + U_n S_f({\mathbb E}Y; \bar Y_m- {\mathbb E}Y)- {\mathbb E} U_n S_f({\mathbb E}Y; \bar Y_m- {\mathbb E}Y)
\\
&
= 
\langle \bar Y_n- {\mathbb E}Y, f'({\mathbb E}Y)\rangle + U_n S_f({\mathbb E}Y; \bar Y_m- {\mathbb E}Y)- {\mathbb E} U_n S_f({\mathbb E}Y; \bar Y_m- {\mathbb E}Y).
\end{align*}
Therefore, by Proposition \ref{U_L_p},
\begin{align*}
&
\Bigl\|U_n f(\bar Y_m) - {\mathbb E} U_n f(\bar Y_m) - \langle \bar Y_n- {\mathbb E}Y, f'({\mathbb E}Y)\rangle\Bigr\|_{L_p}
\\
&
=
\Bigl\|U_n S_f({\mathbb E}Y; \bar Y_m- {\mathbb E}Y)- {\mathbb E} U_n S_f({\mathbb E}Y; \bar Y_m- {\mathbb E}Y)\Bigr\|_{L_p}
\\
&
\leq 
\Bigl\|S_f({\mathbb E}Y; \bar Y_m- {\mathbb E}Y)- {\mathbb E} S_f({\mathbb E}Y; \bar Y_m- {\mathbb E}Y)\Bigr\|_{L_p}.
\end{align*}

\end{proof}

In view of Assumption \ref{assume_on_C_j},  we easily get the next statements.

\begin{proposition}
\label{Lp_for_T_f}
Let $p\geq 1$ and suppose that ${\mathbb E}|f(\bar Y_n)|^p<\infty, n\geq 1.$ Then
\begin{align*}
\Bigl\| \hat T_f(Y_1,\dots, Y_n)- {\mathbb E}\hat T_f(Y_1,\dots, Y_n)\Bigr\|_{L_p}
\lesssim \max_{1\leq j\leq k} \|f(\bar Y_{n_j})-{\mathbb E}f(\bar Y_{n_j})\|_{L_p}
\end{align*}
and 
\begin{align*}
\Bigl\| \check T_f(Y_1,\dots, Y_n)- {\mathbb E}\check T_f(Y_1,\dots, Y_n)\Bigr\|_{L_p}
\lesssim \max_{1\leq j\leq k} \|f(\bar Y_{n_j})-{\mathbb E}f(\bar Y_{n_j})\|_{L_p}
\end{align*}
\end{proposition}

\begin{proof}
The claims immediately follow from Assumption \ref{assume_on_C_j} and Proposition \ref{U_L_p}.

\end{proof}

\begin{proposition}
\label{L_p-linear}
Let $f:F\mapsto {\mathbb R}$ be a continuously differentiable functional. Let $p\geq 1$ and suppose that 
\begin{align*}
{\mathbb E}|S_f({\mathbb E}, \bar Y_n-{\mathbb E})|^p<\infty, n\geq 1. 
\end{align*}
Then
\begin{align*}
&
\Bigl\| \check T_f(Y_1,\dots, Y_n)- {\mathbb E}\check T_f(Y_1,\dots, Y_n)- \langle \bar Y_n-{\mathbb E}Y, f'({\mathbb E} Y)\rangle\Bigr\|_{L_p}
\lesssim 
\max_{1\leq j\leq k}\Bigl\|S_f({\mathbb E}Y; \bar Y_{n_j}- {\mathbb E}Y)- {\mathbb E} S_f({\mathbb E}Y; \bar Y_{n_j}- {\mathbb E}Y)\Bigr\|_{L_p}.
\end{align*}
\end{proposition}

\begin{proof}
Indeed, since $\sum_{j=1}^k C_j=1$ and in view of \eqref{U-Stat},  
\begin{align*}
&
\check T_f(Y_1,\dots, Y_n)- {\mathbb E}\check T_f(Y_1,\dots, Y_n)- \langle \bar Y_n-{\mathbb E}Y, f'({\mathbb E} Y)\rangle
\\
&
= \sum_{j=1}^k C_j \Bigl(U_n f(\bar Y_{n_j}) - {\mathbb E} U_n f(\bar Y_{n_j})- \langle \bar Y_n-{\mathbb E}Y, f'({\mathbb E})\rangle\Bigr).
\end{align*}
Thus, by Proposition \ref{U_L_p_taylor} and  Assumption \ref{assume_on_C_j}, 
\begin{align*}
&
\Bigl\| \check T_f(Y_1,\dots, Y_n)- {\mathbb E}\check T_f(Y_1,\dots, Y_n)- \langle \bar Y_n-{\mathbb E}Y, f'({\mathbb E} Y)\rangle\Bigr\|_{L_p}
\\
&
\leq 
\Bigl\|\sum_{j=1}^k C_j \Bigl(U_n f(\bar Y_{n_j}) - {\mathbb E} U_n f(\bar Y_{n_j})- \langle \bar Y_n-{\mathbb E}Y, f'({\mathbb E})\rangle\Bigr)\Bigr\|_{L_p}
\\
&
\leq \sum_{j=1}^k |C_j| \Bigl\|U_n f(\bar Y_{n_j}) - {\mathbb E} U_n f(\bar Y_{n_j})- \langle \bar Y_n-{\mathbb E}Y, f'({\mathbb E})\rangle\Bigr\|_{L_p}
\\
&
\lesssim 
\max_{1\leq j\leq k}\Bigl\|S_f({\mathbb E}Y; \bar Y_{n_j}- {\mathbb E}Y)- {\mathbb E} S_f({\mathbb E}Y; \bar Y_{n_j}- {\mathbb E}Y)\Bigr\|_{L_p}.
\end{align*}
\end{proof}

\section{Concentration of smooth functionals of sample covariance operator}
\label{sec:conc}

In this section, we derive several concentration bounds for functionals of sample covariance $\hat \Sigma_n$ playing a basic role in the proofs of the main results.
In particular, we need the bounds on the $L_p$-norms $\|f(\hat \Sigma_n)-{\mathbb E}f(\hat \Sigma_n)\|_{L_p}$ for all $p\geq 1$ 
that imply concentration 
with exponential tails of $f(\hat \Sigma_n)$ around its expectation. We also need similar concentration bounds for the remainder of Taylor expansion 
$S_f(\Sigma,\hat \Sigma_n-\Sigma)$ for a smooth functional $f.$ Some concentration inequalities of this type have been previously obtained in \cite{Koltchinskii_2017, Koltchinskii_Zhilova_19} in the case of sample covariances in Hilbert spaces. Here we extend these results to general Banach spaces and, in fact, we also simplify their proofs. Moreover, the inequalities we obtain easily imply concentration inequalities 
for $\|\hat \Sigma_n-\Sigma\|$ initially proved in \cite{Koltchinskii_Lounici} (see also \cite{Adamczak} for another proof). General results on concentration of smooth functions of r.v. 
satisfying Sobolev type inequalities are presented in \cite{Adamczak_Wolff}.

\begin{theorem}
\label{main_conc}
Let $f: L(E^{\ast}, E)\mapsto {\mathbb R}$ be a Lipschitz functional. Then, for all $p\geq 1,$ 
\begin{align*}
\|f(\hat \Sigma_n)-{\mathbb E}f(\hat \Sigma_n)\|_{L_p} \lesssim 
\|f\|_{{\rm Lip}} \|\Sigma\|
\biggl(\biggl(\sqrt{\frac{{\bf r}(\Sigma)}{n}}\vee 1\biggr)\sqrt{\frac{p}{n}}+ \frac{p}{n}\biggr).
\end{align*}
\end{theorem}

\begin{proof}
It is well known that a centered Gaussian r.v. $X$ in a separable Banach space can be represented by the following random series:
\begin{align}
\label{series_X}
X= \sum_{j=1}^{\infty} Z_j x_j,
\end{align}
where $\{Z_j: j\geq 1\}$ are i.i.d. standard normal r.v. and $\{x_j :j\geq 1\}$ are vectors in $E$ such that the series $\sum_{j=1}^{\infty} Z_j x_j$ converges in $E$ a.s. 
and $\sum_{j\geq 1}\|x_j\|^2<\infty$ (see, e.g., \cite{Kwapien}).  

Let $X^{(N)}:= \sum_{j=1}^{N} Z_j x_j.$ 
Denote by $\Sigma^{(N)}$ the covariance operator of $X^{(N)}.$
Since $\|X-X^{(N)}\|\to 0$ as $N\to\infty$ a.s., it easily follows from the Gaussian concentration that, for all $p\geq 1,$ 
${\mathbb E}\|X-X^{(N)}\|^p\to 0$ as $N\to\infty.$

\begin{lemma}
\label{N_infty}
The following statements hold:
\begin{enumerate}[(i)]
\item The sequence $\|\Sigma^{(N)}\|, N\geq 1$ is nondecreasing and $\|\Sigma^{(N)}\| \to \|\Sigma\|$ as $N\to\infty.$ 
\item For all $p\geq 1,$
the sequence ${\mathbb E}\|X^{(N)}\|^p, N\geq 1$ is nondecreasing and 
\begin{align*}
{\mathbb E}\|X^{(N)}\|^p\to {\mathbb E}\|X\|^p\ {\rm as}\ N\to\infty.
\end{align*}
\item As a consequence, 
\begin{align*}
{\bf r}(\Sigma^{(N)}) \to {\bf r}(\Sigma)\ {\rm as}\ N\to\infty.
\end{align*}
\end{enumerate}
\end{lemma}

\begin{proof}
Indeed, 
\begin{align*}
\|\Sigma^{(N)}\| = \sup_{\|u\|\leq 1}\sum_{j=1}^N \langle x_j,u\rangle^2,
\end{align*}
which is a nondecreasing sequence w.r.t. $N.$ 
Moreover, 
\begin{align*}
|\|\Sigma^{(N)}\|- \|\Sigma\|| \leq \sup_{\|u\|\leq 1} \sum_{j=N+1}^{\infty} \langle x_j,u\rangle^2 \leq \sum_{j=N+1}^{\infty}\|x_j\|^2
\to 0\ {\rm as}\ N\to\infty.
\end{align*}

By Jensen's inequality,
\begin{align*}
&
{\mathbb E}\|X^{(N)}\|^p= {\mathbb E} \biggl\|{\mathbb E}_{Z_{N+1}} \Bigl(\sum_{j=1}^N Z_j x_j + Z_{N+1} x_{N+1}\Bigr)\biggr\|^p
\\
&
\leq {\mathbb E} {\mathbb E}_{Z_{N+1}}\biggl\|\sum_{j=1}^N Z_j x_j + Z_{N+1} x_{N+1}\biggr\|^p
= {\mathbb E}\|X^{(N+1)}\|^p
\end{align*}
implying that the sequence ${\mathbb E}\|X^{(N)}\|^2, N\geq 1$ is nondecreasing. 
A similar argument shows that ${\mathbb E}\|X^{(N)}\|^p\leq {\mathbb E}\|X\|^p, N\geq 1.$
Since $\|X^{(N)}\|^p \to \|X\|^p$ as $N\to \infty$ a.s. and, for all $p'>p,$  
\begin{align*}
\sup_{N\geq 1}{\mathbb E}\|X^{(N)}\|^{p'} \leq {\mathbb E}\|X\|^{p'}<\infty,
\end{align*}
we can conclude that ${\mathbb E}\|X^{(N)}\|^p\to {\mathbb E}\|X\|^p$ as $N\to\infty.$

The last claim of the lemma easily follows from the definition of ${\bf r}(\Sigma).$ 
\end{proof}

For $X_k, k\geq 1,$ we have
\begin{align*}
X_k = \sum_{j=1}^{\infty} Z_{j,k} x_j,
\end{align*}
where $Z_{j,k}, j\geq 1, k=1,\dots, n$ are i.i.d. standard normal r.v. and $\sum_{j=1}^{\infty} Z_{j,k} x_j$ converges in $E$ a.s. for all $k=1,\dots, n.$ 
Let 
\begin{align*}
X_k^{(N)} = \sum_{j=1}^{N} Z_{j,k} x_j, k=1,\dots, n, N\geq 1.
\end{align*}
Denote by $\hat \Sigma_n^{(N)}$ the sample covariance based on $X_1^{(N)}, \dots, X_n^{(N)}.$
Since the mapping $E\ni x\mapsto x\otimes x\in L(E^{\ast}, L)$ is continuous, we easily get that
\begin{align*}
\hat \Sigma_n^{(N)}= n^{-1}\sum_{k=1}^n X_k^{(N)}\otimes X_k^{(N)} \to \hat \Sigma_n=n^{-1}\sum_{k=1}^n X_k\otimes X_k\ {\rm as}\ N\to \infty\ {\rm a.s.}
\end{align*}
in the space $L(E^{\ast}, E)$ (equipped with the operator norm). 
Note that, for all $p\geq 1,$
\begin{align*}
{\mathbb E}\|\hat \Sigma_n\|^p \leq {\mathbb E}\biggl(n^{-1}\sum_{k=1}^n \|X_k\|^2\biggr)^p
\leq n^{-1} \sum_{k=1}^n {\mathbb E}\|X_k\|^{2p} = {\mathbb E}\|X\|^{2p}<\infty
\end{align*}
and, in view of Lemma \ref{N_infty} (ii), we also get
\begin{align*}
\sup_{N\geq 1}{\mathbb E}\|\hat \Sigma_n^{(N)}\|^p \leq \sup_{N\geq 1}{\mathbb E}\|X^{(N)}\|^{2p}= {\mathbb E}\|X\|^{2p}<\infty.
\end{align*}
Since we have $\|\hat \Sigma_n^{(N)}-\hat \Sigma_n\|^{p}\to 0$ as $N\to\infty$ a.s. and, for all $p'>p,$ 
\begin{align*}
\sup_{N\geq 1}{\mathbb E}\|\hat \Sigma_n^{(N)}-\hat \Sigma_n\|^{p'} \lesssim_{p'} {\mathbb E}\|X\|^{2p'}<\infty,
\end{align*}
we can conclude that, for all $p\geq 1,$
\begin{align*}
{\mathbb E} \|\hat \Sigma_n^{(N)}-\hat \Sigma_n\|^{p}\to 0\ {\rm as}\ N\to\infty.
\end{align*}
Since $f$ is Lipschitz, this implies that 
\begin{align*}
{\mathbb E} |f(\hat \Sigma_n^{(N)})-f(\hat \Sigma_n)|^p \to 0\ {\rm as}\ N\to \infty
\end{align*}
and
\begin{align}
\label{approx_N}
\Bigl\|(f(\hat \Sigma_n^{(N)})-f(\hat \Sigma_n))-{\mathbb E}(f(\hat \Sigma_n^{(N)})-f(\hat \Sigma_n))\Bigr\|_{L_p} \to 0\ {\rm as}\ N\to \infty.
\end{align}

Thus, to obtain a bound on $\|f(\hat \Sigma_n)- {\mathbb E}f(\hat \Sigma_n)\|_{L_p},$ it is now enough to bound 
$\|f(\hat \Sigma_n^{(N)})- {\mathbb E}f(\hat \Sigma_n^{(N)})\|_{L_p}$
and pass to the limit as $N\to\infty.$ To this end, recall that $\hat \Sigma_n^{(N)}=\hat \Sigma_n^{(N)} ({\mathcal Z})$ depends on ${\mathcal Z}:= (Z_{j,k})_{j=1,\dots N; k=1,\dots n}\in {\mathbb R}^{N\times n}.$ Our main tool is the following well known form of Gaussian concentration inequalities (it could be proved, for instance, using Maurey-Pisier argument, see, e.g., \cite{Gine}, Theorem 2.1.7).
For a locally Lipschitz function $g:{\mathbb R}^m \mapsto {\mathbb R},$ denote by 
\begin{align*}
(Lg)(z):= \inf_{U\ni z} \sup_{z',z''\in U}\frac{|f(z')-f(z'')|}{\|z'-z''\|_{\ell_2}}
\end{align*}
the local Lipschitz constant of $g$ at point $z$ (with the infimum taken over all neighborhoods $U$ of point $z$). 
Note that this definition could be extended to local Lipschitz functions from a metric space into another metric space.

\begin{proposition}
\label{Gauss_conc}
Let ${\mathcal Z}\sim N(0,I_m)$ be a standard normal r.v. in ${\mathbb R}^m$ and let $g:{\mathbb R}^m \mapsto {\mathbb R}$ be a locally Lipschitz 
function. Then, for all $p\geq 1,$
\begin{align*}
\|g({\mathcal Z})-{\mathbb E}g({\mathcal Z})\|_{L_p} \lesssim \sqrt{p} \|(Lg)({\mathcal Z})\|_{L_p}.
\end{align*}
\end{proposition}

We will apply the bound of Proposition \ref{Gauss_conc} to 
$
g({\mathcal Z}) := f(\hat \Sigma_n^{(N)}({\mathcal Z})).
$
For ${\mathcal Z}$ and $\tilde {\mathcal Z}=(\tilde Z_{j,k})_{j=1,\dots N; k=1,\dots n},$ $\hat \Sigma_n^{(N)}= \hat \Sigma_n^{(N)} ({\mathcal Z})$ and $\tilde \Sigma_n^{(N)}:=\hat \Sigma_n^{(N)}(\tilde{\mathcal Z}),$ 
we then have 
\begin{align*}
&
\|\hat \Sigma_n^{(N)}-\tilde \Sigma_n^{(N)}\|= \sup_{\|u\|,\|v\|\leq 1} \biggl|n^{-1}\sum_{k=1}^n \langle X_k^{(N)},u\rangle \langle X_k^{(N)},v\rangle -n^{-1}\sum_{k=1}^n \langle \tilde X_k^{(N)},u\rangle \langle \tilde X_k^{(N)},v\rangle\biggr|
\\
&
\leq \sup_{\|u\|,\|v\|\leq 1} \biggl|n^{-1}\sum_{k=1}^n \langle X_k^{(N)}-\tilde X_k^{(N)},u\rangle \langle X_k^{(N)},v\rangle\biggr|
+\sup_{\|u\|,\|v\|\leq 1}
\biggl|n^{-1}\sum_{k=1}^n \langle \tilde X_k^{(N)},u\rangle \langle X_k^{(N)}-\tilde X_k^{(N)},v\rangle\biggr|
\\
&
\leq \sup_{\|u\|\leq 1} \biggl(n^{-1}\sum_{k=1}^n \langle X_k^{(N)}-\tilde X_k^{(N)},u\rangle^2\biggr)^{1/2}
\biggl[\sup_{\|v\|\leq 1} \biggl(n^{-1}\sum_{k=1}^n \langle X_k^{(N)},v\rangle^2\biggr)^{1/2} + \sup_{\|v\|\leq 1} \biggl(n^{-1}\sum_{k=1}^n \langle \tilde X_k^{(N)},v\rangle^2\biggr)^{1/2}\biggr] 
\\
&
\leq (\|\hat \Sigma_n^{(N)}\|^{1/2}+\|\tilde \Sigma_n^{(N)}\|^{1/2}) \sup_{\|u\|\leq 1} \biggl(n^{-1}\sum_{k=1}^n \langle X_k^{(N)}-\tilde X_k^{(N)},u\rangle^2\biggr)^{1/2}.
\end{align*}
Note that 
\begin{align*}
\langle X_k^{(N)}-\tilde X_k^{(N)},u\rangle^2 &= \biggl(\sum_{j=1}^N (Z_{j,k}-\tilde Z_{j,k})\langle x_j,u\rangle\biggr)^2 
\leq \sum_{j=1}^N (Z_{j,k}-\tilde Z_{j,k})^2 \sum_{j=1}^N \langle x_j,u\rangle^2 
= \langle \Sigma^{(N)}u,u\rangle \sum_{j=1}^N (Z_{j,k}-\tilde Z_{j,k})^2.
\end{align*}
Therefore,
\begin{align}
\label{Lip_Sigma}
&
\nonumber
\|\hat \Sigma_n^{(N)}-\tilde \Sigma_n^{(N)}\|
\leq 
(\|\hat \Sigma_n^{(N)}\|^{1/2}+\|\tilde \Sigma_n^{(N)}\|^{1/2}) \|\Sigma^{(N)}\|^{1/2}
\biggl(n^{-1}\sum_{k=1}^n \sum_{j=1}^N (Z_{j,k}-\tilde Z_{j,k})^2\biggr)^{1/2}
\\
&
\leq 
(\|\hat \Sigma_n^{(N)}\|^{1/2}+\|\tilde \Sigma_n^{(N)}\|^{1/2}) \frac{\|\Sigma^{(N)}\|^{1/2}}{\sqrt{n}}\|{\mathcal Z}-\tilde {\mathcal Z}\|_{\ell_2}.
\end{align}
The last inequality implies the following bound on the local Lipschitz constant of the mapping ${\mathbb R}^{N\times n}\ni {\mathcal Z}\mapsto 
\hat \Sigma_n^{(N)}({\mathcal Z}):$
\begin{align}
\label{bd_LhatSigma}
(L\hat \Sigma_n^{(N)})({\mathcal Z}) \leq \frac{2\|\Sigma^{(N)}\|^{1/2}}{\sqrt{n}}\|\hat \Sigma_n^{(N)}({\mathcal Z})\|^{1/2}.
\end{align}
It immediately implies the bound on the local Lipschitz constant 
of function ${\mathcal Z}\mapsto f(\hat \Sigma_n^{(N)}({\mathcal Z})):$ 
\begin{align*}
(L f(\hat \Sigma_n^{(N)}))({\mathcal Z})\leq 2\|f\|_{{\rm Lip}} \frac{\|\Sigma^{(N)}\|^{1/2}}{\sqrt{n}}\|\hat \Sigma_n^{(N)}({\mathcal Z})\|^{1/2}.
\end{align*}
By Proposition \ref{Gauss_conc}, we get, that for all $p\geq 1,$
\begin{align}
\label{conc_odin}
\|f(\hat \Sigma_n^{(N)})-{\mathbb E}f(\hat \Sigma_n^{(N)})\|_{L_p} \lesssim \sqrt{p}\|f\|_{{\rm Lip}} \frac{\|\Sigma^{(N)}\|^{1/2}}{\sqrt{n}}
\Bigl\|\|\hat \Sigma_n^{(N)}({\mathcal Z})\|^{1/2}\Bigr\|_{L_p}.
\end{align}
Note also that 
\begin{align*}
\Bigl|\|\hat \Sigma_n^{(N)}({\mathcal Z})\|^{1/2}-\|\hat \Sigma_n^{(N)}(\tilde{\mathcal Z})\|^{1/2}\Bigr|
\leq \frac{\Bigl\|\hat \Sigma_n^{(N)}({\mathcal Z})-\hat \Sigma_n^{(N)}(\tilde{\mathcal Z})\Bigr\|}{\|\hat \Sigma_n^{(N)}({\mathcal Z})\|^{1/2}+\|\hat \Sigma_n^{(N)}(\tilde {\mathcal Z})\|^{1/2}},
\end{align*}
which together with \eqref{Lip_Sigma} implies that ${\mathcal Z}\mapsto \|\hat \Sigma_n^{(N)}({\mathcal Z})\|^{1/2}$
is a Lipschitz function with constant $\frac{\|\Sigma^{(N)}\|^{1/2}}{\sqrt{n}}.$ Again, by Proposition \ref{Gauss_conc},
\begin{align}
\label{conc_dva}
\Bigl\|  \|\hat \Sigma_n^{(N)}({\mathcal Z})\|^{1/2}-  {\mathbb E}\|\hat \Sigma_n^{(N)}({\mathcal Z})\|^{1/2}\Bigr\|_{L_p}
\lesssim \sqrt{p} \frac{\|\Sigma^{(N)}\|^{1/2}}{\sqrt{n}}.
\end{align}
Combining \eqref{conc_odin} and \eqref{conc_dva} yields the bound
\begin{align*}
\|f(\hat \Sigma_n^{(N)})-{\mathbb E}f(\hat \Sigma_n^{(N)})\|_{L_p} \lesssim 
\sqrt{p}\|f\|_{{\rm Lip}} \frac{\|\Sigma^{(N)}\|^{1/2}}{\sqrt{n}}{\mathbb E}\|\hat \Sigma_n^{(N)}\|^{1/2}
+ p \|f\|_{{\rm Lip}} \frac{\|\Sigma^{(N)}\|}{n}.
\end{align*}
Note also that, by Lemma \ref{N_infty}, $\|\Sigma^{(N)}\|\leq \|\Sigma\|$ and ${\mathbb E}\|X^{(N)}\|^2\leq {\mathbb E}\|X\|^2.$
Therefore, using the bound of Theorem \ref{KL_1},
\begin{align*}
&
{\mathbb E}\|\hat \Sigma_n^{(N)}\|^{1/2}
\leq 
{\mathbb E}^{1/2}\|\hat \Sigma_n^{(N)}\| \leq \|\Sigma^{(N)}\|^{1/2} + {\mathbb E}^{1/2}\|\hat \Sigma_n^{(N)}-\Sigma^{(N)}\|.
\\
&
\leq \|\Sigma^{(N)}\|^{1/2} + \biggl(\|\Sigma^{(N)}\|^{1/2}\sqrt{\frac{{\mathbb E}\|X^{(N)}\|^2}{n}} \bigvee \frac{{\mathbb E}\|X^{(N)}\|^2}{n}\biggr)^{1/2}
\\
&
\leq \|\Sigma\|^{1/2} + C\biggl(\|\Sigma\|^{1/2}\sqrt{\frac{{\mathbb E}\|X\|^2}{n}} \bigvee \frac{{\mathbb E}\|X\|^2}{n}\biggr)^{1/2}
\\
&
=\|\Sigma\|^{1/2}\biggl(1+ C\biggl(\sqrt{\frac{{\bf r}(\Sigma)}{n}}\vee \frac{{\bf r}(\Sigma)}{n}\biggr)^{1/2}\biggr)\lesssim 
\|\Sigma\|^{1/2}\biggl(\sqrt{\frac{{\bf r}(\Sigma)}{n}}\vee 1\biggr),
\end{align*}
and we get the following bound:
\begin{align*}
\|f(\hat \Sigma_n^{(N)})-{\mathbb E}f(\hat \Sigma_n^{(N)})\|_{L_p} \lesssim 
\|f\|_{{\rm Lip}} \|\Sigma\|
\biggl(\biggl(\sqrt{\frac{{\bf r}(\Sigma)}{n}}\vee 1\biggr)\sqrt{\frac{p}{n}}+ \frac{p}{n}\biggr).
\end{align*}
Passing to the limit as $N\to\infty$ and using \eqref{approx_N}, we get  
\begin{align*}
\|f(\hat \Sigma_n)-{\mathbb E}f(\hat \Sigma_n)\|_{L_p} \lesssim 
\|f\|_{{\rm Lip}} \|\Sigma\|
\biggl(\biggl(\sqrt{\frac{{\bf r}(\Sigma)}{n}}\vee 1\biggr)\sqrt{\frac{p}{n}}+ \frac{p}{n}\biggr).
\end{align*}
\end{proof}

Theorem \ref{main_conc} immediately implies the following concentration bound.

\begin{corollary}
For all $t\geq 1$ with probability at least $1-e^{-t}$
\begin{align*}
|f(\hat \Sigma_n)-{\mathbb E}f(\hat \Sigma_n)| \lesssim 
\|f\|_{{\rm Lip}} \|\Sigma\|
\biggl(\biggl(\sqrt{\frac{{\bf r}(\Sigma)}{n}}\vee 1\biggr)\sqrt{\frac{t}{n}}+ \frac{t}{n}\biggr).
\end{align*}
\end{corollary}

\begin{proof}
Indeed, by Markov inequality,
\begin{align*}
{\mathbb P}\Bigl\{|f(\hat \Sigma_n)-{\mathbb E}f(\hat \Sigma_n)|\geq e \|f(\hat \Sigma_n)-{\mathbb E}f(\hat \Sigma_n)\|_{L_t}\Bigr\}
\leq \frac{\|f(\hat \Sigma_n)-{\mathbb E}f(\hat \Sigma_n)\|_{L_t}^t}{e^t\|f(\hat \Sigma_n)-{\mathbb E}f(\hat \Sigma_n)\|_{L_t}^t}=e^{-t},
\end{align*}
implying the claim.
\end{proof}

The next corollary is also obvious.

\begin{corollary}
The following bound holds:
\begin{align*}
\|f(\hat \Sigma_n)-{\mathbb E}f(\hat \Sigma_n)\|_{\psi_1} \lesssim 
\|f\|_{{\rm Lip}} \frac{\|\Sigma\|}{\sqrt{n}}
\biggl(\sqrt{\frac{{\bf r}(\Sigma)}{n}}\vee 1\biggr).
\end{align*}
\end{corollary}

Note also that the functional $L(E^{\ast}; E)\ni B\mapsto \|B-\Sigma\|$ is Lipschits with constant $1.$ Applying the above statements 
to this function, we get the following results initially proved in a more complicated way in \cite{Koltchinskii_Lounici} (see also theorems \ref{KL_1} and \ref{KL_2}).

\begin{proposition}
\label{K_L-A}
The following bounds hold:
\begin{enumerate}[(i)]
\item for all $p\geq 1$
\begin{align*}
\Bigl\|\|\hat \Sigma_n-\Sigma\|-{\mathbb E}\|\hat \Sigma_n-\Sigma\|\Bigr\|_{L_p}
\lesssim \|\Sigma\|
\biggl(\biggl(\sqrt{\frac{{\bf r}(\Sigma)}{n}}\vee 1\biggr)\sqrt{\frac{p}{n}}+ \frac{p}{n}\biggr);
\end{align*}
\item 
\begin{align*}
\Bigl\|\|\hat \Sigma_n-\Sigma\|-{\mathbb E}\|\hat \Sigma_n-\Sigma\|\Bigr\|_{\psi_1}
\lesssim 
\frac{\|\Sigma\|}{\sqrt{n}}
\biggl(\sqrt{\frac{{\bf r}(\Sigma)}{n}}\vee 1\biggr);
\end{align*}
\item for all $t\geq 1,$
\begin{align*}
\Bigl|\|\hat \Sigma_n-\Sigma\|-{\mathbb E}\|\hat \Sigma_n-\Sigma\|\Bigr|
\lesssim \|\Sigma\|
\biggl(\biggl(\sqrt{\frac{{\bf r}(\Sigma)}{n}}\vee 1\biggr)\sqrt{\frac{t}{n}}+ \frac{t}{n}\biggr).
\end{align*}
\end{enumerate}
\end{proposition}

We will also need concentration bounds on the remainder of the first order Taylor expansion 
\begin{align*}
S_f(\Sigma; \hat \Sigma_n-\Sigma) = f(\hat \Sigma_n) - f(\Sigma)- \langle \hat \Sigma_n-\Sigma, f'(\Sigma)\rangle
\end{align*}
for a functional $f$ with $\|f'\|_{{\rm Lip}_{\rho}}<\infty$ for some $\rho\in (0,1].$

\begin{theorem}
\label{S_f_conc}
Suppose that $f'\in {\rm Lip}_{\rho}(L(E^{*}, E))$ and that ${\bf r}(\Sigma)\lesssim n.$ Then, for all $p\geq 1,$ 
\begin{align*}
&
\Bigl\|S_f(\Sigma; \hat \Sigma_n-\Sigma)- {\mathbb E}S_f(\Sigma; \hat \Sigma_n-\Sigma)\Bigr\|_{L_p}
\\
&
\lesssim 
\|f'\|_{{\rm Lip}_{\rho}}\|\Sigma\|^{1+\rho}\biggl(\sqrt{\frac{p}{n}}
\Bigl(\sqrt{\frac{{\bf r}(\Sigma)}{n}}\Bigr)^{\rho} + \Bigl(\frac{p}{n}\Bigr)^{(1+\rho)/2} + \Bigl(\frac{p}{n}\Bigr)^{1+\rho} \biggr). 
\end{align*}
In particular, it implies that 
\begin{align*}
\Bigl\|S_f(\Sigma; \hat \Sigma_n-\Sigma)- {\mathbb E}S_f(\Sigma; \hat \Sigma_n-\Sigma)\Bigr\|_{\psi_{1/(1+\rho)}}
\lesssim \|f'\|_{{\rm Lip}_{\rho}} 
\frac{\|\Sigma\|^{1+\rho}}{\sqrt{n}}
\Bigl(\sqrt{\frac{{\bf r}(\Sigma)}{n}}\Bigr)^{\rho}.
\end{align*}
\end{theorem}

\begin{proof}
As in the proof of Theorem \ref{main_conc}, it is enough 
to obtain the concentration bound for $S_f(\Sigma; \hat \Sigma_n^{(N)}-\Sigma)$ and then pass to the limit as $N\to \infty.$
First note that, by the second bound of Proposition \ref{Taylor_rem}, the local Lipschitz constant of the mapping 
$H\mapsto S_f(\Sigma; H)$ is bounded from above as follows:  
\begin{align*}
(L S_f(\Sigma; \cdot))(H)\lesssim \|f'\|_{{\rm Lip}_{\rho}} \|H\|^{\rho}.
\end{align*}
Together with \eqref{bd_LhatSigma}, this implies the following bound on the local Lipschitz constant of 
$
{\mathcal Z}\mapsto S_f(\Sigma; \hat \Sigma_n^{(N)}({\mathcal Z})-\Sigma):
$
\begin{align}
\label{LS_f_bd}
&
\nonumber
(L S_f(\Sigma; \hat \Sigma_n^{(N)}(\cdot)-\Sigma))({\mathcal Z})
\\
&
\nonumber
\lesssim  \|f'\|_{{\rm Lip}_{\rho}}\|\hat \Sigma_n^{(N)}({\mathcal Z})-\Sigma\|^{\rho}\frac{\|\Sigma^{(N)}\|^{1/2}}{\sqrt{n}}\|\hat \Sigma_n^{(N)}({\mathcal Z})\|^{1/2}
\\
&
\nonumber
\lesssim  
\|f'\|_{{\rm Lip}_{\rho}}\frac{\|\Sigma^{(N)}\|^{1/2}\|\Sigma\|^{1/2}}{\sqrt{n}}
\|\hat \Sigma_n^{(N)}({\mathcal Z})-\Sigma\|^{\rho}
+
\|f'\|_{{\rm Lip}_{\rho}}\frac{\|\Sigma^{(N)}\|^{1/2}}{\sqrt{n}}\|\hat \Sigma_n^{(N)}({\mathcal Z})-\Sigma\|^{\rho+1/2}
\\
&
\nonumber
\lesssim 
\|f'\|_{{\rm Lip}_{\rho}}\frac{\|\Sigma^{(N)}\|^{1/2}\|\Sigma\|^{1/2}}{\sqrt{n}}
\|\hat \Sigma_n^{(N)}({\mathcal Z})-\Sigma^{(N)}\|^{\rho}
+
\|f'\|_{{\rm Lip}_{\rho}}\frac{\|\Sigma^{(N)}\|^{1/2}\|\Sigma\|^{1/2}}{\sqrt{n}}
\|\Sigma^{(N)}-\Sigma\|^{\rho}
\\
&
+
\|f'\|_{{\rm Lip}_{\rho}}\frac{\|\Sigma^{(N)}\|^{1/2}}{\sqrt{n}}\|\hat \Sigma_n^{(N)}({\mathcal Z})-\Sigma^{(N)}\|^{\rho+1/2}
+ \|f'\|_{{\rm Lip}_{\rho}}\frac{\|\Sigma^{(N)}\|^{1/2}}{\sqrt{n}}\|\Sigma^{(N)}-\Sigma\|^{\rho+1/2}.
\end{align}
By Proposition \ref{Gauss_conc}, we get 
\begin{align*}
&
\Bigl\|S_f(\Sigma; \hat \Sigma_n^{(N)}-\Sigma)- {\mathbb E}S_f(\Sigma; \hat \Sigma_n^{(N)}-\Sigma)\Bigr\|_{L_p}
\\
&
\lesssim \sqrt{p} \|f'\|_{{\rm Lip}_{\rho}}
\biggl(
\frac{\|\Sigma^{(N)}\|^{1/2}\|\Sigma\|^{1/2}}{\sqrt{n}}\Bigl\|\|\hat \Sigma_n^{(N)}-\Sigma^{(N)}\|^{\rho}\Bigr\|_{L_p}
+\frac{\|\Sigma^{(N)}\|^{1/2}\|\Sigma\|^{1/2}}{\sqrt{n}}\|\Sigma^{(N)}-\Sigma\|^{\rho}
\\
&
+ \frac{\|\Sigma^{(N)}\|^{1/2}}{\sqrt{n}}\Bigl\|\|\hat \Sigma_n^{(N)}-\Sigma^{(N)}\|^{\rho+1/2}\Bigr\|_{L_p}
+\frac{\|\Sigma^{(N)}\|^{1/2}}{\sqrt{n}}\|\Sigma^{(N)}-\Sigma\|^{\rho+1/2}
\biggr)
\end{align*}
By Proposition \ref{K_L-A} (i), 
\begin{align}
\label{bd_rho}
&
\nonumber
\Bigl\|\|\hat \Sigma_n^{(N)}-\Sigma^{(N)}\|^{\rho}\Bigr\|_{L_p}
\leq \Bigl\|\|\hat \Sigma_n^{(N)}-\Sigma^{(N)}\|\Bigr\|_{L_p}^{\rho}
\\
&
\nonumber
\leq \Bigl({\mathbb E} \|\hat \Sigma_n^{(N)}-\Sigma^{(N)}\|\Bigr)^{\rho} 
+ \Bigl\|\|\hat \Sigma_n^{(N)}-\Sigma^{(N)}\|-{\mathbb E} \|\hat \Sigma_n^{(N)}-\Sigma^{(N)}\|\Bigr\|_{L_p}^{\rho}
\\
&
\lesssim \|\Sigma^{(N)}\|^{\rho}\Bigl(\sqrt{\frac{{\bf r}(\Sigma^{(N)})}{n}}\vee \frac{{\bf r}(\Sigma^{(N)})}{n}\Bigr)^{\rho} 
+ \|\Sigma^{(N)}\|^{\rho}
\biggl(\biggl(\sqrt{\frac{{\bf r}(\Sigma^{(N)})}{n}}\vee 1\biggr)\sqrt{\frac{p}{n}}+ \frac{p}{n}\biggr)^{\rho}.
\end{align}
Since $\rho+1/2<2,$ we have  
\begin{align*}
\Bigl\|\|\hat \Sigma_n^{(N)}-\Sigma^{(N)}\|^{\rho+1/2}\Bigr\|_{L_p}\leq \Bigl\|\|\hat \Sigma_n^{(N)}-\Sigma^{(N)}\|\Bigr\|_{L_{2p}}^{\rho+1/2},
\end{align*}
and, similarly to \eqref{bd_rho}, we get 
\begin{align*}
&
\Bigl\|\|\hat \Sigma_n^{(N)}-\Sigma^{(N)}\|^{\rho+1/2}\Bigr\|_{L_p}
\\
&
\lesssim 
\|\Sigma^{(N)}\|^{\rho+1/2}\Bigl(\sqrt{\frac{{\bf r}(\Sigma^{(N)})}{n}}\vee \frac{{\bf r}(\Sigma^{(N)})}{n}\Bigr)^{\rho+1/2} 
+ \|\Sigma^{(N)}\|^{\rho+1/2}
\biggl(\biggl(\sqrt{\frac{{\bf r}(\Sigma^{(N)})}{n}}\vee 1\biggr)\sqrt{\frac{p}{n}}+ \frac{p}{n}\biggr)^{\rho+1/2}.
\end{align*}
Therefore, 
\begin{align*}
&
\Bigl\|S_f(\Sigma; \hat \Sigma_n^{(N)}-\Sigma)- {\mathbb E}S_f(\Sigma; \hat \Sigma_n^{(N)}-\Sigma)\Bigr\|_{L_p}
\\
&
\lesssim \|f'\|_{{\rm Lip}_{\rho}}
\biggl(
p^{1/2}\frac{\|\Sigma^{(N)}\|^{1/2+\rho}\|\Sigma\|^{1/2}}{\sqrt{n}}
\Bigl(\sqrt{\frac{{\bf r}(\Sigma^{(N)})}{n}}\vee \frac{{\bf r}(\Sigma^{(N)})}{n}\Bigr)^{\rho} 
\\
&
+ \|\Sigma^{(N)}\|^{1/2+\rho}\|\Sigma\|^{1/2}
\biggl(\Bigl(\frac{p}{n}\Bigr)^{(1+\rho)/2}\biggl(\sqrt{\frac{{\bf r}(\Sigma^{(N)})}{n}}\vee 1\biggr)^{\rho}+ \Bigl(\frac{p}{n}\Bigr)^{\rho+1/2}\biggr)
\\
&
+ p^{1/2}\frac{\|\Sigma^{(N)}\|^{1+\rho}}{\sqrt{n}}\Bigl(\sqrt{\frac{{\bf r}(\Sigma^{(N)})}{n}}\vee \frac{{\bf r}(\Sigma^{(N)})}{n}\Bigr)^{\rho+1/2} 
\\
&
+ \|\Sigma^{(N)}\|^{1+\rho}
\biggl(\Bigl(\frac{p}{n}\Bigr)^{\rho/2+3/4}\biggl(\sqrt{\frac{{\bf r}(\Sigma^{(N)})}{n}}\vee 1\biggr)^{\rho+1/2}+ \Bigl(\frac{p}{n}\Bigr)^{1+\rho}\biggr)
\\
&
+\frac{\|\Sigma^{(N)}\|^{1/2}\|\Sigma\|^{1/2}}{\sqrt{n}}\|\hat \Sigma_n^{(N)}-\Sigma\|^{\rho}
+\frac{\|\Sigma^{(N)}\|^{1/2}}{\sqrt{n}}\|\Sigma^{(N)}-\Sigma\|^{\rho+1/2}
\biggr).
\end{align*}
Passing to the limit as $N\to\infty,$ we get 
\begin{align*}
&
\Bigl\|S_f(\Sigma; \hat \Sigma_n-\Sigma)- {\mathbb E}S_f(\Sigma; \hat \Sigma_n-\Sigma)\Bigr\|_{L_p}
\\
&
\lesssim \|f'\|_{{\rm Lip}_{\rho}}
\biggl(
p^{1/2}\frac{\|\Sigma\|^{1+\rho}}{\sqrt{n}}
\Bigl(\sqrt{\frac{{\bf r}(\Sigma)}{n}}\vee \frac{{\bf r}(\Sigma)}{n}\Bigr)^{\rho} 
+ \|\Sigma\|^{1+\rho}
\biggl(\Bigl(\frac{p}{n}\Bigr)^{(1+\rho)/2}\biggl(\sqrt{\frac{{\bf r}(\Sigma)}{n}}\vee 1\biggr)^{\rho}+ \Bigl(\frac{p}{n}\Bigr)^{\rho+1/2}\biggr)
\\
&
+ p^{1/2}\frac{\|\Sigma\|^{1+\rho}}{\sqrt{n}}\Bigl(\sqrt{\frac{{\bf r}(\Sigma)}{n}}\vee \frac{{\bf r}(\Sigma)}{n}\Bigr)^{\rho+1/2} 
+ \|\Sigma\|^{1+\rho}
\biggl(\Bigl(\frac{p}{n}\Bigr)^{\rho/2+3/4}\biggl(\sqrt{\frac{{\bf r}(\Sigma)}{n}}\vee 1\biggr)^{\rho+1/2}+ \Bigl(\frac{p}{n}\Bigr)^{1+\rho}\biggr)\biggr).
\end{align*}
Assuming that ${\bf r}(\Sigma)\lesssim n,$ the above bound simplifies as follows:
\begin{align}
\label{bd_last}
&
\nonumber
\Bigl\|S_f(\Sigma; \hat \Sigma_n-\Sigma)- {\mathbb E}S_f(\Sigma; \hat \Sigma_n-\Sigma)\Bigr\|_{L_p}
\lesssim \|f'\|_{{\rm Lip}_{\rho}}
\biggl(
p^{1/2}\frac{\|\Sigma\|^{1+\rho}}{\sqrt{n}}
\Bigl(\sqrt{\frac{{\bf r}(\Sigma)}{n}}\Bigr)^{\rho} 
\\
&
\nonumber
+ \|\Sigma\|^{1+\rho}
\biggl(\Bigl(\frac{p}{n}\Bigr)^{(1+\rho)/2}+ \Bigl(\frac{p}{n}\Bigr)^{\rho+1/2} + \Bigl(\frac{p}{n}\Bigr)^{\rho/2+3/4}+ \Bigl(\frac{p}{n}\Bigr)^{1+\rho}\biggr)\biggr)
\\
&
\lesssim \|f'\|_{{\rm Lip}_{\rho}} \|\Sigma\|^{1+\rho}
\biggl(
\sqrt{\frac{p}{n}}
\Bigl(\sqrt{\frac{{\bf r}(\Sigma)}{n}}\Bigr)^{\rho} 
+ 
\Bigl(\frac{p}{n}\Bigr)^{(1+\rho)/2}+ \Bigl(\frac{p}{n}\Bigr)^{1+\rho}\biggr).
\end{align}
\end{proof}

The following corollary is obvious.

\begin{corollary}
Suppose that $f'\in {\rm Lip}_{\rho}$ and ${\bf r}(\Sigma)\lesssim n.$ Then
\begin{align*}
\Bigl\|f(\hat \Sigma_n)-{\mathbb E}f(\hat \Sigma_n)- \langle \hat \Sigma_n-\Sigma, f'(\Sigma)\rangle\Bigr\|_{\psi_{1/(1+\rho)}}
\lesssim \|f'\|_{{\rm Lip}_{\rho}} 
\frac{\|\Sigma\|^{1+\rho}}{\sqrt{n}}
\Bigl(\sqrt{\frac{{\bf r}(\Sigma)}{n}}\Bigr)^{\rho}.
\end{align*}
\end{corollary}

\begin{proof}
It immediately follows from the fact that the right hand side of \eqref{bd_last} is bounded from above by 
$$
\lesssim 
\|f'\|_{{\rm Lip}_{\rho}} p^{1+\rho}
\frac{\|\Sigma\|^{1+\rho}}{\sqrt{n}}
\Bigl(\sqrt{\frac{{\bf r}(\Sigma)}{n}}\Bigr)^{\rho}.
$$
\end{proof}

\section{Proofs of the main results}
\label{sec:proof_main}

In this section, we provide the proofs of the main results stated in Section \ref{sec:Main_th}.

\subsection{Upper risk bounds and normal approximation}

We start with the proof of theorems \ref{main_th_0} and \ref{main_th_1}.

\begin{proof}
To prove the first bound of Theorem \ref{main_th_0}, note that, for $\rho\in (0,1]$ and $p\geq 1,$ 
\begin{align*}
\Bigl\|f(\hat \Sigma_n)-f(\Sigma)\Bigr\|_{L_p} \leq \|f\|_{{\rm Lip}_{\rho}}\Bigl\| \|\hat \Sigma_n-\Sigma\|^{\rho}\Bigr\|_{L_p}
\leq \|f\|_{{\rm Lip}_{\rho}}\Bigl\| \|\hat \Sigma_n-\Sigma\|\Bigr\|_{L_p}^{\rho}.
\end{align*}
Using Proposition \ref{K_L-A}, (i) and Theorem \ref{KL_1}, we easily get 
\begin{align*}
\Bigl\| \|\hat \Sigma_n-\Sigma\|\Bigr\|_{L_p} \lesssim \|\Sigma\| \biggl[\biggl(\sqrt{\frac{{\bf r}(\Sigma)}{n}}\vee \frac{{\bf r}(\Sigma)}{n}\biggr)
+ \biggl(\sqrt{\frac{{\bf r}(\Sigma)}{n}}\vee 1\biggr) \sqrt{\frac{p}{n}} +\frac{p}{n}\bigg],
\end{align*}
which implies 
\begin{align*}
\Bigl\|f(\hat \Sigma_n)-f(\Sigma)\Bigr\|_{L_p}\lesssim \|f\|_{{\rm Lip}_{\rho}}\|\Sigma\|^{\rho}
\biggl[\biggl(\sqrt{\frac{{\bf r}(\Sigma)}{n}}\vee \frac{{\bf r}(\Sigma)}{n}\biggr)^{\rho}
+ \biggl(\sqrt{\frac{{\bf r}(\Sigma)}{n}}\vee 1\biggr)^{\rho} \biggl(\sqrt{\frac{p}{n}}\Biggr)^{\rho} +\biggl(\frac{p}{n}\biggr)^{\rho}\biggr].
\end{align*}

To prove the second bound, note that 
\begin{align*} 
{\mathbb E} f(\hat \Sigma_n) -f(\Sigma) = {\mathbb E}\langle \hat \Sigma_n-\Sigma, f'(\Sigma)\rangle + {\mathbb E}S_f(\Sigma, \hat \Sigma_n-\Sigma)
= {\mathbb E}S_f(\Sigma, \hat \Sigma_n-\Sigma).
\end{align*}
Using the first bound of Proposition \ref{Taylor_rem} along with Theorem \ref{KL_1} and Proposition \ref{K_L-A} (i) (for $p=1+\rho$), we get 
\begin{align}
\label{bias_bd_XYZ}
&
\nonumber
|{\mathbb E} f(\hat \Sigma_n) -f(\Sigma)|\lesssim \|f'\|_{{\rm Lip}_{\rho}} {\mathbb E}\|\hat \Sigma_n-\Sigma\|^{1+\rho}
\\
& 
\nonumber
\lesssim \|f'\|_{{\rm Lip}_{\rho}} \|\Sigma\|^{1+\rho} \biggl(\sqrt{\frac{{\bf r}(\Sigma)}{n}}\vee \frac{{\bf r}(\Sigma)}{n}\biggr)^{1+\rho}
+ \|f'\|_{{\rm Lip}_{\rho}} \biggl(\frac{\|\Sigma\|}{\sqrt{n}}\biggr)^{1+\rho}\biggl(\sqrt{\frac{{\bf r}(\Sigma)}{n}}\vee 1\biggr)^{1+\rho}
\\
&
\lesssim 
\|f'\|_{{\rm Lip}_{\rho}} \|\Sigma\|^{1+\rho} \biggl(\sqrt{\frac{{\bf r}(\Sigma)}{n}}\vee \frac{{\bf r}(\Sigma)}{n}\biggr)^{1+\rho},
\end{align}
where we also used the fact that ${\bf r}(\Sigma)\geq 1.$
It remains to combine the last bound with concentration bound of Theorem \ref{main_conc} to complete the proof of Theorem \ref{main_th_0}.

To prove Theorem \ref{main_th_1}, we will apply the results of Section \ref{jackknife} with the space $L(E^{\ast}, E)$ playing the role of $F$ and $Y=X\otimes X,$ $Y_j=X_j\otimes X_j, j=1,\dots, n.$
Then $\Sigma={\mathbb E}Y$ and $\hat \Sigma_n= \bar Y_n.$ Also, $T_f^{(1)}(X_1,\dots, X_n)= \hat T_f (Y_1,\dots, Y_n)$
and $T_f^{(2)}(X_1,\dots, X_n)= \check T_f (Y_1,\dots, Y_n).$ By Proposition \ref{bias_jack}, we get the following bounds on the bias of these 
estimators: for $i=1,2,$
\begin{align*}
|{\mathbb E} T_f^{(i)}(X_1,\dots, X_n) - f(\Sigma)|\lesssim_{k,\rho} \|f^{(k)}\|_{{\rm Lip}_{\rho}} \max_{1\leq j\leq k} {\mathbb E}\|\hat \Sigma_{n_j}- \Sigma\|^{k+\rho}.
\end{align*}
Using the bounds of Theorem \ref{KL_1} and Proposition \ref{K_L-A} (i), we get that, for all $j=1,\dots, k,$
\begin{align*}
\Bigl\|\|\hat \Sigma_{n_j}-\Sigma\|\Bigr\|_{L_{k+\rho}}
\lesssim 
\|\Sigma\|
\biggl(\sqrt{\frac{{\rm \bf r}(\Sigma)}{n_j}}+\frac{{\rm \bf r}(\Sigma)}{n_j}+\biggl(\sqrt{\frac{{\bf r}(\Sigma)}{n_j}}\vee 1\biggr)\sqrt{\frac{k+\rho}{n_j}}+ \frac{k+\rho}{n_j}\biggr).
\end{align*}
Since $n/c\leq n_j\le n, j=1,\dots, k,$ this implies that
\begin{align*}
\max_{1\leq j\leq k}\Bigl\|\|\hat \Sigma_{n_j}-\Sigma\|\Bigr\|_{L_{k+\rho}}
&
\lesssim 
\|\Sigma\|
\biggl(\sqrt{\frac{{\rm \bf r}(\Sigma)}{n}}+\frac{{\rm \bf r}(\Sigma)}{n}+\biggl(\sqrt{\frac{{\bf r}(\Sigma)}{n}}\vee 1\biggr)\sqrt{\frac{k+\rho}{n}}+ \frac{k+\rho}{n}\biggr)
\\
&
\lesssim_{k,\rho} \|\Sigma\|\biggl(\sqrt{\frac{{\rm \bf r}(\Sigma)}{n}}\vee \frac{{\rm \bf r}(\Sigma)}{n}\biggr).
\end{align*}
Therefore, for $i=1,2,$
\begin{align}
\label{bias_AAA}
|{\mathbb E} T_f^{(i)}(X_1,\dots, X_n) - f(\Sigma)|\lesssim_{k,\rho} \|f^{(k)}\|_{{\rm Lip}_{\rho}} \|\Sigma\|^{k+\rho}\biggl(\sqrt{\frac{{\rm \bf r}(\Sigma)}{n}}\vee \frac{{\rm \bf r}(\Sigma)}{n}\biggr)^{k+\rho}.
\end{align}
On the other hand, by Proposition \ref{Lp_for_T_f}, for all $p\geq 1$ and $i=1,2,$ 
\begin{align*}
\Bigl\|T_f^{(i)}(X_1,\dots, X_n)-{\mathbb E}T_f^{(i)}(X_1,\dots, X_n)\Bigr\|_{L_p}
\lesssim \max_{1\leq j\leq k} \Bigl\|f(\hat \Sigma_{n_j})-{\mathbb E}f(\hat \Sigma_{n_j})\Bigr\|_{L_p}
\end{align*}
and, by Theorem \ref{main_conc} for $j=1,\dots, k,$
\begin{align*}
\label{conc_AAA}
\|f(\hat \Sigma_{n_j})-{\mathbb E}f(\hat \Sigma_{n_j})\|_{L_p} \lesssim 
\|f\|_{{\rm Lip}} \|\Sigma\|
\biggl(\biggl(\sqrt{\frac{{\bf r}(\Sigma)}{n_j}}\vee 1\biggr)\sqrt{\frac{p}{n_j}}+ \frac{p}{n_j}\biggr).
\end{align*}
Since $n/c\leq n_j\le n, j=1,\dots, k,$ we get 
\begin{align}
\Bigl\|T_f^{(i)}(X_1,\dots, X_n)-{\mathbb E}T_f^{(i)}(X_1,\dots, X_n)\Bigr\|_{L_p}
\lesssim \|f\|_{{\rm Lip}} \|\Sigma\|
\biggl(\biggl(\sqrt{\frac{{\bf r}(\Sigma)}{n}}\vee 1\biggr)\sqrt{\frac{p}{n}}+ \frac{p}{n}\biggr).
\end{align}
Combining bounds \eqref{bias_AAA} and \eqref{conc_AAA} yields the first bound of Theorem \ref{main_th_1}.
\end{proof}

Next we provide the proof of theorems \ref{main_th_2_0}, \ref{main_th_2} and Corollary \ref{cor_main_th_2}.

\begin{proof}
Since 
\begin{align*}
f(\hat \Sigma_n)-f(\Sigma) = \langle \hat \Sigma_n-\Sigma, f'(\Sigma)\rangle + 
S_f(\Sigma, \hat \Sigma_n-\Sigma)- {\mathbb E}S_f(\Sigma, \hat \Sigma_n-\Sigma)
+{\mathbb E} f(\hat \Sigma_n)-f(\Sigma),
\end{align*}
it is enough to use bound on the bias \eqref{bias_bd_XYZ} and the first concentration bound of Theorem \ref{S_f_conc} 
to complete the proof of Theorem \ref{main_th_2_0}.

It follows from Proposition \ref{L_p-linear} that, for all $p\geq 1,$
\begin{align*}
&
\Bigl\| T_f^{(2)}(X_1,\dots, X_n)- {\mathbb E} T_f^{(2)}(X_1,\dots, X_n)- \langle \hat \Sigma_n-\Sigma, f'(\Sigma)\rangle\Bigr\|_{L_p}
\\
&
\lesssim 
\max_{1\leq j\leq k}\Bigl\|S_f(\Sigma; \hat \Sigma_{n_j}- \Sigma)- {\mathbb E} S_f(\Sigma; \hat \Sigma_{n_j}- \Sigma)\Bigr\|_{L_p}.
\end{align*}
Thus, it remains to use the first bound of Theorem \ref{S_f_conc} to get that, for all $p\geq 1,$ 
\begin{align*}
&
\Bigl\| T_f^{(2)}(X_1,\dots, X_n)- {\mathbb E} T_f^{(2)}(X_1,\dots, X_n)- \langle \hat \Sigma_n-\Sigma, f'(\Sigma)\rangle\Bigr\|_{L_p}
\\
&
\lesssim_{k,\rho,\gamma} \max_{1\leq j\leq k} \|f'\|_{{\rm Lip}_{\gamma}}\|\Sigma\|^{1+\gamma}\sqrt{\frac{p}{n_j}}
\Bigl(\sqrt{\frac{{\bf r}(\Sigma)}{n_j}}\Bigr)^{\gamma} + \|f'\|_{{\rm Lip}_{\gamma}}\|\Sigma\|^{1+\gamma}
\biggl(\Bigl(\frac{p}{n_j}\Bigr)^{(1+\gamma)/2} + \Bigl(\frac{p}{n_j}\Bigr)^{1+\gamma}\biggr)
\\
&
\lesssim_{k,\rho,\gamma} \|f'\|_{{\rm Lip}_{\gamma}}\|\Sigma\|^{1+\gamma}\sqrt{\frac{p}{n}}
\Bigl(\sqrt{\frac{{\bf r}(\Sigma)}{n}}\Bigr)^{\gamma} + \|f'\|_{{\rm Lip}_{\gamma}}\|\Sigma\|^{1+\gamma}
\biggl(\Bigl(\frac{p}{n}\Bigr)^{(1+\gamma)/2} + \Bigl(\frac{p}{n}\Bigr)^{1+\gamma}\biggr),
\end{align*}
where we also use the condition that $n/c\leq n_j\leq n, j=1,\dots, k.$
Combining the last bound with bound \eqref{bias_AAA} for $i=2$ completes the proof of Theorem \ref{main_th_2}.

To deduce the bound of Corollary \ref{cor_main_th_2}, it is enough to observe that, for all $\gamma \in (0,1],$ $\|f'\|_{{\rm Lip}_{\gamma}}\leq 2\|f'\|_{C^{1}}$
and to use the bound of Theorem \ref{main_th_2} for $\gamma=\frac{1}{\beta}-1.$
\end{proof}

We turn now to the proof of Corollary \ref{cor_norm_approx}.

\begin{proof}
We will use the following lemma.

\begin{lemma}
\label{bd_on_psi_1}
The following bound holds:
\begin{align*}
\Bigl\|\langle X\otimes X, f'(\Sigma)\rangle\Bigr\|_{\psi_1} \lesssim \|\Sigma\| \|f'(\Sigma)\|.
\end{align*}
\end{lemma}

\begin{proof}
First note that, in view of \eqref{series_X}, 
\begin{align*}
&
\langle X\otimes X, f'(\Sigma)\rangle 
= \biggl\langle 
\biggl(\sum_{j=1}^{\infty} Z_j x_j\biggr)\otimes \biggl(\sum_{j=1}^{\infty} Z_j x_j\biggr), 
f'(\Sigma)
\biggr\rangle
= \sum_{i,j=1}^{\infty} Z_i Z_j \langle x_i\otimes x_j, f'(\Sigma)\rangle 
= \sum_{i,j=1}^{\infty} a_{ij} Z_i Z_j,
\end{align*}
where 
\begin{align*}
a_{ij} := \frac{1}{2}\Bigl(\langle x_i\otimes x_j, f'(\Sigma)\rangle+  \langle x_j\otimes x_i, f'(\Sigma)\rangle\Bigr), i,j\geq 1.
\end{align*}
The symmetric matrix $(a_{ij})_{i,j\geq 1}$ defines in a standard way an operator $A:\ell_2\mapsto \ell_2,$
which is bounded and self-adjoint. Its boundedness easily follows from the fact that it is Hilbert--Schmidt since
\begin{align*}
\|A\|_{2}^2=\sum_{i,j\geq 1} a_{ij}^2 \leq \sum_{i,j\geq 1} \langle x_i\otimes x_j, f'(\Sigma)\rangle^2 
\leq \|f'(\Sigma)\|^2 \biggl(\sum_{j=1}^{\infty} \|x_j\|^2\biggr)^2<\infty.
\end{align*}
We will prove that, in fact, $A$ is even a nuclear operator with $\|A\|_1 \leq \|\Sigma\| \|f'(\Sigma)\|.$ To this end, note that, for a self-adjoint 
operator $B:\ell_2\mapsto \ell_2$ with $\|B\|\leq 1,$ we have  
\begin{align*}
\langle A,B\rangle &= \sum_{i,j\geq 1} a_{ij}b_{ij} 
= \sum_{i,j\geq 1} \frac{1}{2}\Bigl(\langle x_i\otimes x_j, f'(\Sigma)\rangle+  \langle x_j\otimes x_i, f'(\Sigma)\rangle\Bigr) b_{ij}
= \biggl\langle \frac{1}{2}\sum_{i,j\geq 1} b_{ij}(x_i\otimes x_j + x_j\otimes x_i), f'(\Sigma)\biggr\rangle,
\end{align*}
which implies 
\begin{align*}
&
|\langle A,B\rangle| \leq  \frac{1}{2} \| f'(\Sigma)\| \biggl\|\sum_{i,j\geq 1} b_{ij}(x_i\otimes x_j + x_j\otimes x_i)\biggr\|
\\
&
=  \frac{1}{2} \|f'(\Sigma)\| \sup_{\|u\|, \|v\|\leq 1} \biggl|\sum_{i,j\geq 1} b_{ij} \langle x_i, u\rangle \langle x_j,v\rangle+b_{ij} \langle x_j, u\rangle \langle x_i,v\rangle \biggr|
\\
&
\leq \|f'(\Sigma)\| \sup_{\|u\|, \|v\|\leq 1} \biggl|\sum_{i,j\geq 1} b_{ij} \langle x_i, u\rangle \langle x_j,v\rangle \biggr|
\leq \|f'(\Sigma)\| \|B\| \sup_{\|u\|\leq 1} \sum_{j\geq 1} \langle x_j,u\rangle^2
\\
&
\leq \|f'(\Sigma)\| \sup_{\|u\|\leq 1}\langle \Sigma u,u\rangle \leq \|\Sigma\| \|f'(\Sigma)\|.
\end{align*}
By the duality between the nuclear norm and the operator norm, we have 
\begin{align*}
\|A\|_1 = \sup\Bigl\{|\langle A,B\rangle|: B\ {\rm is\ compact},\ \|B\|\leq 1\Bigr\} \leq \|\Sigma\| \|f'(\Sigma)\|.
\end{align*}

To prove the bound of the lemma, it is enough to consider the case when $X$ is represented by a finite sum $X=\sum_{j=1}^N Z_j x_j$ and then pass to the limit 
as $N\to \infty.$ In this case,
\begin{align*}
\langle X\otimes X, f'(\Sigma)\rangle = \sum_{i,j=1}^{N} a_{ij} Z_i Z_j \overset{d}{=} \sum_{k=1}^N \lambda_k g_k^2,
\end{align*}
where $\lambda_k$ are the eigenvalues of the matrix $(a_{ij})_{i,j=1}^N$ and $g_1,\dots, g_{N}$ are i.i.d. $N(0,1)$ r.v. 
Therefore, we have 
\begin{align*}
\Bigl\|\langle X\otimes X, f'(\Sigma)\rangle\Bigr\|_{\psi_1} = \Bigl\|\sum_{k=1}^N \lambda_k g_k^2\Bigr\|_{\psi_1}
\leq \sum_{k=1}^N |\lambda_k| \|g_k^2\|_{\psi_1}\lesssim \sum_{k=1}^N |\lambda_k| = \|A\|_1 \leq \|\Sigma\| \|f'(\Sigma)\|,
\end{align*}
which completes the proof.
\end{proof}

Note that it was also possible to use Hanson-Wright inequality in the proof of the above lemma, but we prefer a more direct argument.

We will also use the following normal approximation bounds proved in \cite{Rio} (see Theorem 4.1 and equation (4.3); Theorem 2.1). 
Let $\eta, \eta_1, \eta_2, \dots$ be i.i.d. r.v. with ${\mathbb E}\eta=0,$ ${\mathbb E}\eta^2=1$ and let $Z\sim N(0,1).$ Then, for all $r\in (1,2],$
\begin{align*}
W_r\biggl(\frac{\eta_1+\dots+\eta_n}{\sqrt{n}}, Z\biggr) \lesssim \frac{{\mathbb E}^{1/r}|\eta|^{r+2}}{\sqrt{n}}.
\end{align*}
Moreover, if $\|\eta\|_{\psi_1}<\infty,$ then 
\begin{align*}
W_{\psi_1}\biggl(\frac{\eta_1+\dots+\eta_n}{\sqrt{n}}, Z\biggr)\lesssim \frac{C(\|\eta\|_{\psi_1})}{\sqrt{n}}
\end{align*}
with some constant $C(\|\eta\|_{\psi_1})<\infty$ depending only on $\|\eta\|_{\psi_1}.$

We will apply this bounds to r.v. 
$$\eta := \frac{\langle X\otimes X, f'(\Sigma)\rangle- {\mathbb E}\langle X\otimes X, f'(\Sigma)\rangle}{\sigma_f(\Sigma)}.$$
By Lemma \ref{bd_on_psi_1}, we have 
\begin{align*}
\|\eta\|_{\psi_1} \lesssim \frac{\|\Sigma\| \|f'(\Sigma)\|}{\sigma_f(\Sigma)}.
\end{align*}
This also implies that 
\begin{align*}
{\mathbb E}^{1/2}|\eta|^{4} \lesssim \frac{\|\Sigma\|^2 \|f'(\Sigma)\|^2}{\sigma_f^2(\Sigma)}.
\end{align*}
Since 
\begin{align*}
\frac{\langle \hat \Sigma_n-\Sigma,f'(\Sigma)\rangle}{\sigma_f(\Sigma)} = \frac{1}{n}\sum_{j=1}^n \eta_j,
\end{align*}
we get 
\begin{align}
\label{W_2_A}
W_2 \Bigl(\frac{\sqrt{n}\langle \hat \Sigma_n-\Sigma, f'(\Sigma)\rangle}{\sigma_f(\Sigma)}, Z\Bigr)\lesssim \frac{\|\Sigma\|^2 \|f'(\Sigma)\|^2}{\sigma_f^2(\Sigma)}\frac{1}{\sqrt{n}}
\end{align}
and 
\begin{align}
\label{W_psi_B}
W_{\psi_1} \Bigl(\frac{\sqrt{n}\langle \hat \Sigma_n-\Sigma, f'(\Sigma)\rangle}{\sigma_f(\Sigma)}, Z\Bigr)\lesssim C\biggl(\frac{\|\Sigma\| \|f'(\Sigma)\|}{\sigma_f(\Sigma)}\biggr)\frac{1}{\sqrt{n}}.
\end{align}
Combining bound \eqref{W_2_A} with the bound of Theorem \ref{main_th_2} for $p=2,$ we get, under the assumptions of this theorem, that 
\begin{align*}
&
W_2 \biggl(\frac{\sqrt{n}(T_f^{(2)}(X_1,\dots, X_n)-f(\Sigma))}{\sigma_f(\Sigma)}, Z\biggr)
\lesssim_{k,\rho, \gamma}
\frac{\|\Sigma\|^2 \|f'(\Sigma)\|^2}{\sigma_f^2(\Sigma)}\frac{1}{\sqrt{n}}
\\
&
+
\frac{\|f'\|_{{\rm Lip}_{\gamma}}\|\Sigma\|^{1+\gamma}}{\sigma_f(\Sigma)}
\Bigl(\sqrt{\frac{{\bf r}(\Sigma)}{n}}\Bigr)^{\gamma} 
+ \frac{\|f^{(k)}\|_{{\rm Lip}_{\rho}} \|\Sigma\|^{k+\rho}}{\sigma_f(\Sigma)}
\sqrt{n}\biggl(\sqrt{\frac{{\bf r}(\Sigma)}{n}}\biggr)^{k+\rho}.
\end{align*}
Similarly, combining bound \eqref{W_psi_B} with bound of Corollary \ref{cor_main_th_2}, we get under the assumptions of this corollary that, for all $\beta\in [1/2,1),$ 
\begin{align*}
&
W_{\psi_{\beta}} \biggl(\frac{\sqrt{n}(T_f^{(2)}(X_1,\dots, X_n)-f(\Sigma))}{\sigma_f(\Sigma)}, Z\biggr)
\lesssim_{k,\rho} 
C\biggl(\frac{\|\Sigma\| \|f'(\Sigma)\|}{\sigma_f(\Sigma)}\biggr)\frac{1}{\sqrt{n}}
\\
&
+
\frac{\|f'\|_{C^1} \|\Sigma\|^{1/\beta}}{\sigma_f(\Sigma)}
\biggl(\sqrt{\frac{{\bf r}(\Sigma)}{n}}\biggr)^{1/\beta-1}
+ \frac{\|f^{(k)}\|_{{\rm Lip}_{\rho}} \|\Sigma\|^{k+\rho}}{\sigma_f(\Sigma)}\sqrt{n}\biggl(\sqrt{\frac{{\bf r}(\Sigma)}{n}}\biggr)^{k+\rho},
\end{align*}
which completes the proof of Claim (ii) of Corollary \ref{cor_norm_approx}.
The proof of Claim (i) is similar.
\end{proof}

\subsection{Local versions of upper bounds}

We now turn to the proof of Theorem \ref{main_th_1_local}.

\begin{proof}
Since functional $f:L(E^{\ast}, E)\mapsto {\mathbb R}$ is Lipschitz, we can still use the concentration bound of Theorem \ref{main_conc}. 
However, since $f$ is smooth only locally in the neighborhood $U=B(\Sigma, \delta),$ we have to modify the bounds on the bias based on 
propositions \ref{bias_jack} and \ref{bias_jack_one}. To this end, it is enough to modify the bound on the remainder $R$ in Lemma \ref{bias_control},
which is based on global differentiability. This requires bounding ${\mathbb E} S_f^{(k)}(\Sigma,\hat \Sigma_n-\Sigma).$ We have 
\begin{align*}
{\mathbb E} S_f^{(k)}(\Sigma,\hat \Sigma_n-\Sigma)= {\mathbb E} S_f^{(k)}(\Sigma,\hat \Sigma_n-\Sigma)I(\|\hat \Sigma_n-\Sigma\|<\delta)
+ {\mathbb E} S_f^{(k)}(\Sigma,\hat \Sigma_n-\Sigma)I(\|\hat \Sigma_n-\Sigma\|\geq \delta).
\end{align*}
The first term in the right hand side is still controlled using the bound on the remainder of Taylor expansion of Proposition \ref{Taylor_rem_high}:
\begin{align}
\label{S_f<delta}
\nonumber
|{\mathbb E} S_f^{(k)}(\Sigma,\hat \Sigma_n-\Sigma)I(\|\hat \Sigma_n-\Sigma\|<\delta)| & \lesssim \|f^{(k)}\|_{{\rm Lip}_{\rho}(U)} {\mathbb E}\|\hat \Sigma_n-\Sigma\|^{k+\rho}
\\
&
\lesssim_{k,\rho} \|f^{(k)}\|_{{\rm Lip}_{\rho}(U)} \|\Sigma\|^{k+\rho}\biggl(\sqrt{\frac{{\bf r}(\Sigma)}{n}}\biggr)^{k+\rho},
\end{align}
where we used the fact that ${\bf r}(\Sigma)\lesssim n$ and bounds of Proposition \ref{K_L-A} and Theorem \ref{KL_1}.

To control the second term, note that 
\begin{align*}
|S_f^{(k)}(\Sigma,\hat \Sigma_n-\Sigma)|& \leq |f(\hat \Sigma_n)-f(\Sigma)|+ \sum_{j=1}^k \frac{\|f^{(j)}(\Sigma)\|}{j!} \|\hat \Sigma_n-\Sigma\|^j
\\
&
\leq \|f\|_{{\rm Lip}} \|\hat \Sigma_n-\Sigma\|+  \sum_{j=1}^k \frac{\|f^{(j)}(\Sigma)\|}{j!} \|\hat \Sigma_n-\Sigma\|^j.
\end{align*}
Thus
\begin{align}
\label{bdd_S_f_k}
&
\nonumber
|{\mathbb E} S_f^{(k)}(\Sigma,\hat \Sigma_n-\Sigma)I(\|\hat \Sigma_n-\Sigma\|\geq \delta)| 
\\
&
\leq \|f\|_{{\rm Lip}} {\mathbb E}\|\hat \Sigma_n-\Sigma\|I(\|\hat \Sigma_n-\Sigma\|\geq \delta)
+  \sum_{j=1}^k \frac{\|f^{(j)}(\Sigma)\|}{j!} {\mathbb E}\|\hat \Sigma_n-\Sigma\|^j I(\|\hat \Sigma_n-\Sigma\|\geq \delta).
\end{align}
Note that 
\begin{align}
\label{bbc_odin}
{\mathbb E}\|\hat \Sigma_n-\Sigma\|^j I(\|\hat \Sigma_n-\Sigma\|\geq \delta) & \leq {\mathbb E}^{1/2}\|\hat \Sigma_n-\Sigma\|^{2j}\  {\mathbb P}^{1/2}\{\|\hat \Sigma_n-\Sigma\|\geq \delta\}.
\end{align}
Using the bound of Theorem \ref{KL_1} and bound (i) of Proposition \ref{K_L-A}, we easily get that, under the assumption ${\bf r}(\Sigma)\lesssim n$ for all $p\geq 1,$
\begin{align}
\label{L_p_norm_hat_Sigma}
&
\Bigl\|\|\hat \Sigma_n-\Sigma\|\Bigr\|_{L_p} \leq {\mathbb E}\|\hat \Sigma_n-\Sigma\| + \Bigl\|\|\hat \Sigma_n-\Sigma\|-{\mathbb E}\|\hat \Sigma_n-\Sigma\|\Bigr\|_{L_p}
\leq 
C\|\Sigma\|
\Bigl(
\sqrt{\frac{{\rm \bf r}(\Sigma)}{n}} + \sqrt{\frac{p}{n}}+ \frac{p}{n}
\Bigr)
\end{align}
with some numerical constant $C>0.$ 
In particular, this implies that 
\begin{align}
\label{bbc_dva}
{\mathbb E}^{1/2}\|\hat \Sigma_n-\Sigma\|^{2j} \lesssim_j \|\Sigma\|^{j} \Bigl(\sqrt{\frac{{\rm \bf r}(\Sigma)}{n}}\Bigr)^j.
\end{align}
If ${\mathbb E}\|\hat \Sigma_n-\Sigma\|\leq \delta/2$ and $\delta \leq \|\Sigma\|$ (which hold under the assumptions of the theorem),
we have that, for a small enough numerical constant $c>0,$
\begin{align}
\label{bbc_tri}
&
\nonumber
{\mathbb P}\{\|\hat \Sigma_n-\Sigma\|> \delta\}\leq {\mathbb P}\{|\|\hat \Sigma_n-\Sigma\|- {\mathbb E}\|\hat \Sigma_n-\Sigma\||> \delta/2\}
\\
&
\leq 
\exp\biggl\{- 4cn \biggl(\frac{\delta}{\|\Sigma\|}\wedge \frac{\delta^2}{\|\Sigma\|^2}\biggr)\biggr\}
\leq 
\exp\biggl\{- \frac{4c n \delta^2}{\|\Sigma\|^2}\biggr\},
\end{align}
where we used bound (iii) of Proposition \ref{K_L-A} with 
\begin{align*}
t:=4cn \biggl(\frac{\delta}{\|\Sigma\|}\wedge \frac{\delta^2}{\|\Sigma\|^2}\biggr).
\end{align*}
It follows from \eqref{bbc_odin}, \eqref{bbc_dva} and \eqref{bbc_tri} that 
\begin{align}
\label{bbc_4}
\nonumber
{\mathbb E}\|\hat \Sigma_n-\Sigma\|^j I(\|\hat \Sigma_n-\Sigma\|\geq \delta) 
&\lesssim_j  \|\Sigma\|^{j} \Bigl(\sqrt{\frac{{\rm \bf r}(\Sigma)}{n}}\Bigr)^j
\exp\biggl\{- \frac{2c n \delta^2}{\|\Sigma\|^2}\biggr\}
\\ 
&
\nonumber
\lesssim_j  \|\Sigma\|^{j} \Bigl(\sqrt{\frac{{\rm \bf r}(\Sigma)}{n}}\Bigr)^j \Bigl(\frac{\|\Sigma\|}{\delta\sqrt{n}}\Bigr)^{j} \exp\biggl\{- \frac{c n \delta^2}{\|\Sigma\|^2}\biggr\}
\\
&
\nonumber
\lesssim_j  \|\Sigma\|^{j} \Bigl(\sqrt{\frac{{\rm \bf r}(\Sigma)}{n}}\Bigr)^j  \Bigl(\frac{1}{\sqrt{{\bf r}(\Sigma)}}\Bigr)^j \exp\biggl\{- \frac{c n \delta^2}{\|\Sigma\|^2}\biggr\}
\\
&
\lesssim_j \Bigl(\frac{\|\Sigma\|}{\sqrt{n}}\Bigr)^j \exp\biggl\{- \frac{c n \delta^2}{\|\Sigma\|^2}\biggr\}.
\end{align}
In the last bounds, we also used the condition that $\delta \geq C\|\Sigma\|\sqrt{\frac{{\bf r}(\Sigma)}{n}}$ and 
the fact that $e^{-x^2} \lesssim_j x^{-j}$ for $x\gtrsim 1,$ implying that 
$
\exp\{- \frac{c n \delta^2}{\|\Sigma\|^2}\} \lesssim_j (\frac{\|\Sigma\|}{\delta \sqrt{n}})^j.$
Substituting \eqref{bbc_4} into bound \eqref{bdd_S_f_k} (and using again bounds \eqref{bbc_dva} for $j=1$ and \eqref{bbc_tri} to control the first term),
we get 
\begin{align*}
|{\mathbb E} S_f^{(k)}(\Sigma,\hat \Sigma_n-\Sigma)I(\|\hat \Sigma_n-\Sigma\|\geq \delta)| 
&
\lesssim_k 
\|f\|_{{\rm Lip}} \frac{\|\Sigma\|}{\sqrt{n}}\exp\biggl\{- \frac{c n \delta^2}{\|\Sigma\|^2}\biggr\}+ 
\max_{1\leq j\leq k}\|f^{(j)}(\Sigma)\| \Bigl(\frac{\|\Sigma\|}{\sqrt{n}}\Bigr)^j \exp\biggl\{- \frac{c n \delta^2}{\|\Sigma\|^2}\biggr\}.
\end{align*}
Recalling \eqref{S_f<delta}, we get the following bound
\begin{align}
\label{bd_previous}
&
\nonumber
|{\mathbb E} S_f^{(k)}(\Sigma,\hat \Sigma_n-\Sigma)| 
\\
&
\lesssim_k \|f^{(k)}\|_{{\rm Lip}_{\rho}(U)}\|\Sigma\|^{k+\rho}\biggl(\sqrt{\frac{{\bf r}(\Sigma)}{n}}\biggr)^{k+\rho} +  
\|f\|_{{\rm Lip}} \frac{\|\Sigma\|}{\sqrt{n}}\exp\biggl\{- \frac{c n \delta^2}{\|\Sigma\|^2}\biggr\}+ 
\max_{1\leq j\leq k}\|f^{(j)}(\Sigma)\|  \Bigl(\frac{\|\Sigma\|}{\sqrt{n}}\Bigr)^j \exp\biggl\{- \frac{c n \delta^2}{\|\Sigma\|^2}\biggr\}.
\end{align}
Note also that it is enough to extend the maximum in the third term in the right hand side of bound \eqref{bd_previous} from $j=2$ to $k$ since the term in the maximum 
corresponding to $j=1$ is dominated by the second term in the right hand side.
Bound \eqref{bd_previous} implies the following bounds on the bias of estimators $T_f^{(1)}(X_1,\dots, X_n), T_f^{(2)}(X_1,\dots, X_n)$
for $i=1,2:$
\begin{align}
\label{bd_bias_T_1,2}
&
\nonumber
|{\mathbb E}T_{f}^{(i)}(X_1,\dots, X_n)-f(\Sigma)|
\\
&
\lesssim_k 
\|f^{(k)}\|_{{\rm Lip}_{\rho}(U)}\|\Sigma\|^{k+\rho}\biggl(\sqrt{\frac{{\bf r}(\Sigma)}{n}}\biggr)^{k+\rho} +  
\|f\|_{{\rm Lip}} \frac{\|\Sigma\|}{\sqrt{n}}\exp\biggl\{- \frac{c n \delta^2}{\|\Sigma\|^2}\biggr\}+ 
\max_{2\leq j\leq k}\|f^{(j)}(\Sigma)\|  \Bigl(\frac{\|\Sigma\|}{\sqrt{n}}\Bigr)^j \exp\biggl\{- \frac{c n \delta^2}{\|\Sigma\|^2}\biggr\}.
\end{align}
The rest of the proof is the same as in the case of Theorem \ref{main_th_1}.
\end{proof}

We now prove Theorem \ref{main_th_2_local}.

\begin{proof} 
To proof the theorem, we need to modify the concentration bounds of Theorem \ref{S_f_conc}
in the case when $f$ is smooth only in the neighborhood $U.$ Namely, the following lemma will be proved. 

\begin{lemma}
\label{conc_local_S_f}
Suppose ${\bf r}(\Sigma)\lessim n$ and $f\in {\rm Lip}(L(E^{\ast}, E)),$ $f'\in {\rm Lip}_{\gamma}(U)$ for some $\gamma\in (0,1]$ and 
$U=B(\Sigma,\delta),$ $\delta\in (0,1].$ Suppose also that $\|\Sigma\|\geq \delta \geq C\|\Sigma\|\sqrt{\frac{{\bf r}(\Sigma)}{n}}$ for a sufficiently large 
constant $C>0.$
Then, for all $p\geq 1,$
\begin{align*}
&
\biggl\|S_f(\Sigma; \hat \Sigma_n-\Sigma)- 
{\mathbb E}S_f(\Sigma; \hat \Sigma_n-\Sigma)\biggr\|_{L_p}
\\
&
\lesssim
(\|f\|_{{\rm Lip}}\vee \|f'\|_{{\rm Lip}_{\gamma}(U)}) 
\frac{\|\Sigma\|^{1+\gamma}}{\delta^{\gamma}}
\biggl(\Bigl(\sqrt{\frac{{\bf r}(\Sigma)}{n}}\Bigr)^{\gamma} \sqrt{\frac{p}{n}} + \Bigl(\frac{p}{n}\Bigr)^{(1+\gamma)/2}
+ \Bigl(\frac{p}{n}\Bigr)^{1+\gamma/2}\biggr).
\end{align*}
\end{lemma}

\begin{proof}
As in the proof of Theorem \ref{S_f_conc}, r.v. $S_f(\Sigma;\hat \Sigma_n-\Sigma)$ will be approximated by $S_f(\Sigma;\hat \Sigma_n^{(N)}-\Sigma).$
We will assume that $N$ is large enough so that $\|\Sigma^{(N)}-\Sigma\|<\delta.$ We have to approximate r.v. $S_f(\Sigma;\hat \Sigma_n^{(N)}-\Sigma)$
further to be able to use the bounds on the remainder of Taylor expansion of Lemma \ref{Taylor_rem}.
To this end, consider a function $\varphi: {\mathbb R}\mapsto [0,1]$
such that $\varphi(y)=1, y\leq 1,$ $\varphi (y)=0, y\geq 2$ and $\varphi $ is Lipschitz with constant $1,$ 
and approximate r.v. $S_f(\Sigma;\hat \Sigma_n^{(N)}-\Sigma)$ by r.v. $S_f(\Sigma;\hat \Sigma_n^{(N)}-\Sigma)\varphi\Bigl(\frac{4\|\hat \Sigma_n^{(N)}-\Sigma\|}{\delta}\Bigr).$ Note that this r.v. coincides with $S_f(\Sigma;\hat \Sigma_n^{(N)}-\Sigma)$ when 
$\|\hat \Sigma_n^{(N)}-\Sigma\|\leq \delta/4$ and it is equal to $0$ when $\|\hat \Sigma_n^{(N)}-\Sigma\| >\delta/2.$ Using Proposition \ref{Taylor_rem}, it will be possible to control the local Lipschitz 
constant of the function 
\begin{align}
\label{loc_Lip_func}
{\mathcal Z}\mapsto S_f(\Sigma;\hat \Sigma_n^{(N)}({\mathcal Z})-\Sigma)\varphi\Bigl(\frac{4\|\hat \Sigma_n^{(N)}({\mathcal Z})-\Sigma\|}{\delta}\Bigr)=: g({\mathcal Z})
\end{align}
as we did for the function $S_f(\Sigma;\hat \Sigma_n^{(N)}({\mathcal Z})-\Sigma)$ in the proof of Theorem \ref{S_f_conc}. 
This will lead to concentration bounds for r.v. $S_f(\Sigma;\hat \Sigma_n^{(N)}-\Sigma)\varphi\Bigl(\frac{4\|\hat \Sigma_n^{(N)}-\Sigma\|}{\delta}\Bigr)$
and, in the limit as $N\to\infty,$ to concentration bounds for r.v. $S_f(\Sigma;\hat \Sigma_n-\Sigma)\varphi\Bigl(\frac{4\|\hat \Sigma_n-\Sigma\|}{\delta}\Bigr).$
Before doing this, we will obtain bounds on the approximation error
\begin{align*}
S_f(\Sigma;\hat \Sigma_n-\Sigma)-S_f(\Sigma;\hat \Sigma_n-\Sigma)\varphi\Bigl(\frac{4\|\hat \Sigma_n-\Sigma\|}{\delta}\Bigr) = 
S_f(\Sigma;\hat \Sigma_n-\Sigma)\biggl(1-\varphi\Bigl(\frac{4\|\hat \Sigma_n-\Sigma\|}{\delta}\Bigr)\biggr).
\end{align*}
Note that it is equal to $0$ when $\|\hat \Sigma_n-\Sigma\|\leq \delta/4$ and, otherwise, it is bounded by 
\begin{align*}
|S_f(\Sigma;\hat \Sigma_n-\Sigma)|\leq |f(\hat \Sigma_n)-f(\Sigma)| + |\langle \hat \Sigma_n-\Sigma, f'(\Sigma)\rangle|
\leq 2\|f\|_{{\rm Lip}(U)}\|\hat \Sigma_n-\Sigma\|.
\end{align*}
Thus, 
\begin{align}
\label{bd_S_f_X_1}
&
\Bigl| S_f(\Sigma;\hat \Sigma_n-\Sigma)-S_f(\Sigma;\hat \Sigma_n-\Sigma)\varphi\Bigl(\frac{4\|\hat \Sigma_n-\Sigma\|}{\delta}\Bigr)\Bigr|
\leq  2\|f\|_{{\rm Lip}}
\|\hat \Sigma_n-\Sigma\| I(\|\hat \Sigma_n-\Sigma\|> \delta/4).
\end{align}
Next we have 
\begin{align}
\label{bd_S_f_X_2}
&
\nonumber
\Bigl\|\|\hat \Sigma_n-\Sigma\| I(\|\hat \Sigma_n-\Sigma\|> \delta/4)\Bigr\|_{L_p} \leq \Bigl\|\|\hat \Sigma_n-\Sigma\|\Bigr\|_{L_{2p}} \Bigl\|I(\|\hat \Sigma_n-\Sigma\|> \delta/4)\Bigr\|_{L_{2p}}
\\
&
=\Bigl\|\|\hat \Sigma_n-\Sigma\|\Bigr\|_{L_{2p}} {\mathbb P}^{1/2p}\{\|\hat \Sigma_n-\Sigma\|> \delta/4\}.
\end{align}
If ${\mathbb E}\|\hat \Sigma_n-\Sigma\|\leq \delta/8$ and $\delta \leq \|\Sigma\|$ (which holds under the assumptions of the lemma), then, similarly to bound \eqref{bbc_tri} for a small enough constant $c>0,$
\begin{align*}
&
{\mathbb P}\{\|\hat \Sigma_n-\Sigma\|> \delta/4\}\leq 
\exp\biggl\{- \frac{2c n \delta^2}{\|\Sigma\|^2}\biggr\},
\end{align*}
Since, for all $\lambda\in (0,1],$ $e^{-x}\lesssim x^{-\lambda/2}, x>0,$ we get 
\begin{align}
\label{plambda}
{\mathbb P}^{1/2p}\{\|\hat \Sigma_n-\Sigma\|> \delta/4\}\leq  \exp\biggl\{- \frac{ c n \delta^2}{p \|\Sigma\|^2}\biggr\}
\lesssim \frac{\|\Sigma\|^{\lambda}}{\delta^{\lambda}} \biggl(\frac{p}{n}\biggr)^{\lambda/2}.
\end{align}
Also, it follows from \eqref{L_p_norm_hat_Sigma} that
\begin{align}
\label{bd_S_f_X_3}
\Bigl\|\|\hat \Sigma_n-\Sigma\|\Bigr\|_{L_{2p}} 
\lesssim \|\Sigma\| \sqrt{\frac{{\bf r}(\Sigma)}{n}} + \|\Sigma\| \sqrt{\frac{p}{n}}+ \|\Sigma\| \frac{p}{n}.
\end{align}
Using \eqref{plambda} with $\lambda=1,$ we get 
\begin{align*}
\|\Sigma\| \sqrt{\frac{{\bf r}(\Sigma)}{n}}\ {\mathbb P}^{1/2p}\{\|\hat \Sigma_n-\Sigma\|> \delta/4\}
\lessim \frac{\|\Sigma\|^2}{\delta}\sqrt{\frac{{\bf r}(\Sigma)}{n}} \sqrt{\frac{p}{n}},
\end{align*}
and, using the same bound with $\lambda=\gamma,$ we get 
\begin{align*}
\|\Sigma\|\sqrt{\frac{p}{n}}\ {\mathbb P}^{1/2p}\{\|\hat \Sigma_n-\Sigma\|> \delta/4\}
\lessim \frac{\|\Sigma\|^{1+\gamma}}{\delta^{\gamma}} \Bigl(\frac{p}{n}\Bigr)^{(1+\gamma)/2}
\end{align*}
and 
\begin{align*}
\|\Sigma\|\frac{p}{n}\ {\mathbb P}^{1/2p}\{\|\hat \Sigma_n-\Sigma\|> \delta/4\}
\lessim \frac{\|\Sigma\|^{1+\gamma}}{\delta^{\gamma}} \Bigl(\frac{p}{n}\Bigr)^{1+\gamma/2}.
\end{align*}
It now follows from \eqref{bd_S_f_X_1}, \eqref{bd_S_f_X_2} and \eqref{bd_S_f_X_3} that
\begin{align}
\label{S_f_approx_S_f}
&
\nonumber
\Bigl\| S_f(\Sigma;\hat \Sigma_n-\Sigma)-S_f(\Sigma;\hat \Sigma_n-\Sigma)\varphi\Bigl(\frac{4\|\hat \Sigma_n-\Sigma\|}{\delta}\Bigr)\Bigr\|_{L_p}
\\
&
\lesssim 
\|f\|_{{\rm Lip}} \biggl( \frac{\|\Sigma\|^2}{\delta}\sqrt{\frac{{\bf r}(\Sigma)}{n}} \sqrt{\frac{p}{n}} + \frac{\|\Sigma\|^{1+\gamma}}{\delta^{\gamma}} \Bigl(\frac{p}{n}\Bigr)^{(1+\gamma)/2}
+ \frac{\|\Sigma\|^{1+\gamma}}{\delta^{\gamma}} \Bigl(\frac{p}{n}\Bigr)^{1+\gamma/2}\biggr).
\end{align}

We will now control the local Lipschitz constant of function $g({\mathcal Z})$ defined by \eqref{loc_Lip_func}.
Since $g({\mathcal Z})$ is a continuous function and it is equal to $0$ on the open set $\{{\mathcal Z}: \|\hat \Sigma_n^{(N)}({\mathcal Z})-\Sigma\|>\delta/2\},$
its local Lipschitz constant $(Lg)({\mathcal Z})$ is equal to $0$ on this set. Thus, it would be enough to control $(Lg)({\mathcal Z})$ on the open 
set $\{{\mathcal Z}: \|\hat \Sigma_n^{(N)}({\mathcal Z})-\Sigma\|<\delta\}.$ Since $\|f'\|_{{\rm Lip}_{\gamma}(U)}<\infty,$ we can use on the above set Proposition \ref{Taylor_rem} to show that, similarly to \eqref{LS_f_bd}, 
\begin{align}
\label{loc_lip_S_f_UUU}
&
\nonumber
(L S_f(\Sigma; \hat \Sigma_n^{(N)}(\cdot)-\Sigma))({\mathcal Z})
\\
&
\nonumber
\lesssim 
\|f'\|_{{\rm Lip}_{\gamma}(U)}\frac{\|\Sigma^{(N)}\|^{1/2}\|\Sigma\|^{1/2}}{\sqrt{n}}
\|\hat \Sigma_n^{(N)}({\mathcal Z})-\Sigma^{(N)}\|^{\gamma}
+
\|f'\|_{{\rm Lip}_{\gamma}(U)}\frac{\|\Sigma^{(N)}\|^{1/2}\|\Sigma\|^{1/2}}{\sqrt{n}}
\|\Sigma^{(N)}-\Sigma\|^{\gamma}
\\
&
+
\|f'\|_{{\rm Lip}_{\gamma}(U)}\frac{\|\Sigma^{(N)}\|^{1/2}}{\sqrt{n}}\|\hat \Sigma_n^{(N)}({\mathcal Z})-\Sigma^{(N)}\|^{\gamma+1/2}
+ \|f'\|_{{\rm Lip}_{\gamma}(U)}\frac{\|\Sigma^{(N)}\|^{1/2}}{\sqrt{n}}\|\Sigma^{(N)}-\Sigma\|^{\gamma+1/2}.
\end{align}
In addition, on the same set, we have 
\begin{align} 
\label{bd_S_f_WWW}
\nonumber
\Bigl|S_f(\Sigma; \hat \Sigma_n^{(N)}({\mathcal Z})-\Sigma)\Bigr| 
&
\lesssim \|f'\|_{{\rm Lip}_{\gamma}(U)}\|\hat \Sigma_n^{(N)}({\mathcal Z})-\Sigma\|^{1+\gamma}
\\
&
\lesssim \|f'\|_{{\rm Lip}_{\gamma}(U)}\|\hat \Sigma_n^{(N)}({\mathcal Z})-\Sigma^{(N)}\|^{1+\gamma}+ 
\|f'\|_{{\rm Lip}_{\gamma}(U)}\|\Sigma^{(N)}-\Sigma\|^{1+\gamma}.
\end{align}
Since $\varphi $ is Lipschitz with constant $1,$ we can use bound \eqref{bd_LhatSigma} to get that 
\begin{align}
\label{Lip_varphi}
\nonumber
\Bigl(L\varphi\Bigl(\frac{4\|\hat \Sigma_n^{(N)}(\cdot)-\Sigma\|}{\delta}\Bigr)\Bigr)({\mathcal Z}) 
&
\leq \frac{8}{\delta}\frac{\|\Sigma^{(N)}\|^{1/2}}{\sqrt{n}}\|\hat \Sigma_n^{(N)}({\mathcal Z})\|^{1/2}
\\
&
\leq \frac{8}{\delta}\frac{\|\Sigma^{(N)}\|^{1/2}}{\sqrt{n}}\|\hat \Sigma_n^{(N)}({\mathcal Z})-\Sigma^{(N)}\|^{1/2} +
\frac{8}{\delta}\frac{\|\Sigma^{(N)}\|}{\sqrt{n}}.
\end{align}
Also, recall that $\varphi$ is bounded by $1.$ 
Then, it follows from \eqref{loc_lip_S_f_UUU}, \eqref{bd_S_f_WWW} and \eqref{Lip_varphi} that 
\begin{align}
\label{Lg_bd_prel}
(Lg)({\mathcal Z}) &\lesssim 
\nonumber
\|f'\|_{{\rm Lip}_{\gamma}(U)}\frac{\|\Sigma^{(N)}\|^{1/2}\|\Sigma\|^{1/2}}{\sqrt{n}}
\|\hat \Sigma_n^{(N)}({\mathcal Z})-\Sigma^{(N)}\|^{\gamma}
\\
&
\nonumber
+
\|f'\|_{{\rm Lip}_{\gamma}(U)}\frac{\|\Sigma^{(N)}\|^{1/2}\|\Sigma\|^{1/2}}{\sqrt{n}}
\|\Sigma^{(N)}-\Sigma\|^{\gamma}
\\
&
\nonumber
+
\|f'\|_{{\rm Lip}_{\gamma}(U)}\frac{\|\Sigma^{(N)}\|^{1/2}}{\sqrt{n}}\|\hat \Sigma_n^{(N)}({\mathcal Z})-\Sigma^{(N)}\|^{\gamma+1/2}
\\
&
\nonumber
+ \|f'\|_{{\rm Lip}_{\gamma}(U)}\frac{\|\Sigma^{(N)}\|^{1/2}}{\sqrt{n}}\|\Sigma^{(N)}-\Sigma\|^{\gamma+1/2}
\\
&
\nonumber
+ \frac{\|f'\|_{{\rm Lip}_{\gamma}(U)}}{\delta}\frac{\|\Sigma^{(N)}\|}{\sqrt{n}}\|\hat \Sigma_n^{(N)}({\mathcal Z})-\Sigma^{(N)}\|^{\gamma+1}  
\\
&
\nonumber
+\frac{\|f'\|_{{\rm Lip}_{\gamma}(U)}}{\delta} \frac{\|\Sigma^{(N)}\|}{\sqrt{n}}\|\Sigma^{(N)}-\Sigma\|^{\gamma+1}.
\\
&
\nonumber
+\frac{\|f'\|_{{\rm Lip}_{\gamma}(U)}}{\delta}\frac{\|\Sigma^{(N)}\|^{1/2}}{\sqrt{n}}
\|\hat \Sigma_n^{(N)}({\mathcal Z})-\Sigma^{(N)}\|^{\gamma+3/2} 
\\
&
+\frac{\|f'\|_{{\rm Lip}_{\gamma}(U)}}{\delta} \frac{\|\Sigma^{(N)}\|^{1/2}}{\sqrt{n}} \|\Sigma^{(N)}-\Sigma\|^{\gamma+1}\|\hat \Sigma_n^{(N)}({\mathcal Z})-\Sigma^{(N)}\|^{1/2}. 
\end{align}
Since we are assuming that $\|\Sigma^{(N)}-\Sigma\|<\delta$ and 
$\|\hat \Sigma_n^{(N)}-\Sigma\|<\delta,$ it follows that $\|\hat \Sigma_n^{(N)}-\Sigma^{(N)}\|<2\delta.$ Recall also that $\|\Sigma^{(N)}\|\leq \|\Sigma\|.$
Therefore, the last four terms in the right hand side of bound \eqref{Lg_bd_prel} are dominated by the first four terms (up to constants),
and we get 
\begin{align}
\label{Lg_bd_fin}
(Lg)({\mathcal Z}) &\lesssim 
\nonumber
\|f'\|_{{\rm Lip}_{\gamma}(U)}\frac{\|\Sigma^{(N)}\|^{1/2}\|\Sigma\|^{1/2}}{\sqrt{n}}
\|\hat \Sigma_n^{(N)}({\mathcal Z})-\Sigma^{(N)}\|^{\gamma}
\\
&
\nonumber
+
\|f'\|_{{\rm Lip}_{\gamma}(U)}\frac{\|\Sigma^{(N)}\|^{1/2}\|\Sigma\|^{1/2}}{\sqrt{n}}
\|\Sigma^{(N)}-\Sigma\|^{\gamma}
\\
&
\nonumber
+
\|f'\|_{{\rm Lip}_{\gamma}(U)}\frac{\|\Sigma^{(N)}\|^{1/2}}{\sqrt{n}}\|\hat \Sigma_n^{(N)}({\mathcal Z})-\Sigma^{(N)}\|^{\gamma+1/2}
\\
&
+ \|f'\|_{{\rm Lip}_{\gamma}(U)}\frac{\|\Sigma^{(N)}\|^{1/2}}{\sqrt{n}}\|\Sigma^{(N)}-\Sigma\|^{\gamma+1/2}.
\end{align}
Note that this bound is similar to \eqref{LS_f_bd}. Thus, we can repeat the concentration argument of Theorem \ref{S_f_conc} 
to get a bound similar to \eqref{bd_last}:
\begin{align}
\label{bd_last_last}
&
\nonumber
\biggl\|S_f(\Sigma; \hat \Sigma_n-\Sigma)\varphi\Bigl(\frac{4\|\hat \Sigma_n-\Sigma\|}{\delta}\Bigr)- 
{\mathbb E}S_f(\Sigma; \hat \Sigma_n-\Sigma)\varphi\Bigl(\frac{4\|\hat \Sigma_n-\Sigma\|}{\delta}\Bigr)\biggr\|_{L_p}
\\
&
\lesssim \|f'\|_{{\rm Lip}_{\gamma}(U)}\|\Sigma\|^{1+\gamma}
\biggl(
\Bigl(\sqrt{\frac{{\bf r}(\Sigma)}{n}}\Bigr)^{\gamma} 
\sqrt{\frac{p}{n}}
+ 
\Bigl(\frac{p}{n}\Bigr)^{(1+\gamma)/2}+ \Bigl(\frac{p}{n}\Bigr)^{1+\gamma}\biggr).
\end{align}
Combining bounds \eqref{bd_last_last} and \eqref{S_f_approx_S_f} easily yilelds
\begin{align*}
&
\Bigl\|S_f(\Sigma; \hat \Sigma_n-\Sigma)- 
{\mathbb E}S_f(\Sigma; \hat \Sigma_n-\Sigma)\Bigr\|_{L_p}
\\
&
\lesssim \|f'\|_{{\rm Lip}_{\gamma}(U)}
\biggl(
\|\Sigma\|^{1+\gamma}\Bigl(\sqrt{\frac{{\bf r}(\Sigma)}{n}}\Bigr)^{\gamma} 
\sqrt{\frac{p}{n}}
+ 
\|\Sigma\|^{1+\gamma}\Bigl(\frac{p}{n}\Bigr)^{(1+\gamma)/2}+ \|\Sigma\|^{1+\gamma}\Bigl(\frac{p}{n}\Bigr)^{1+\gamma}\biggr)
\\
&
+\|f\|_{{\rm Lip}} \biggl( \frac{\|\Sigma\|^2}{\delta}\sqrt{\frac{{\bf r}(\Sigma)}{n}} \sqrt{\frac{p}{n}} + \frac{\|\Sigma\|^{1+\gamma}}{\delta^{\gamma}} \Bigl(\frac{p}{n}\Bigr)^{(1+\gamma)/2}
+ \frac{\|\Sigma\|^{1+\gamma}}{\delta^{\gamma}} \Bigl(\frac{p}{n}\Bigr)^{1+\gamma/2}\biggr).
\end{align*}
Note that the assumption $\delta\gtrsim \|\Sigma\|\sqrt{\frac{{\bf r}(\Sigma)}{n}}$ implies that 
\begin{align}
\label{bd_2_gamma}
\nonumber
\frac{\|\Sigma\|^2}{\delta}\sqrt{\frac{{\bf r}(\Sigma)}{n}} 
&= \frac{\|\Sigma\|^{1+\gamma}}{\delta^{\gamma}} \Bigl(\sqrt{\frac{{\bf r}(\Sigma)}{n}}\Bigr)^{\gamma} 
\  \frac{\|\Sigma\|^{1-\gamma}}{\delta^{1-\gamma}} \Bigl(\sqrt{\frac{{\bf r}(\Sigma)}{n}}\Bigr)^{1-\gamma}
\\
&
\lesssim \frac{\|\Sigma\|^{1+\gamma}}{\delta^{\gamma}} \Bigl(\sqrt{\frac{{\bf r}(\Sigma)}{n}}\Bigr)^{\gamma}.
\end{align}
Therefore, for $\delta\leq 1,$ we get 
\begin{align*}
&
\biggl\|S_f(\Sigma; \hat \Sigma_n-\Sigma)- 
{\mathbb E}S_f(\Sigma; \hat \Sigma_n-\Sigma)\biggr\|_{L_p}
\\
&
\lesssim
(\|f\|_{{\rm Lip}}\vee \|f'\|_{{\rm Lip}_{\gamma}(U)}) 
\frac{\|\Sigma\|^{1+\gamma}}{\delta^{\gamma}}
\biggl(\Bigl(\sqrt{\frac{{\bf r}(\Sigma)}{n}}\Bigr)^{\gamma} \sqrt{\frac{p}{n}} + \Bigl(\frac{p}{n}\Bigr)^{(1+\gamma)/2}
+ \Bigl(\frac{p}{n}\Bigr)^{1+\gamma/2}\biggr).
\end{align*}
\end{proof}

The rest of the proof is similar to the proof of Theorem \ref{main_th_2}. It utilizes concentration bound of Lemma \ref{conc_local_S_f},
bound \eqref{bd_bias_T_1,2} on the bias of estimator $T_f^{(2)}(X_1,\dots, X_n)$ from the proof of Theorem \ref{main_th_1_local}
and bound \eqref{bd_2_gamma}. It also uses the fact that, for all $\gamma\in (0,1],$ $\|f'\|_{{\rm Lip}_{\gamma}(U)}\lesssim \|f'\|_{C^{k-1+\rho}(U)}.$
\end{proof}

\subsection{Minimax lower bounds: proof of Theorem \ref{loc_min_max}}
\label{min_max_loc_proof}

Let $L\subset {\mathbb H}$ be the linear span of eigenvectors of $\Sigma_0$ corresponding to the eigenvalues $\lambda, \mu$ and let $P_L$ denote the orthogonal projection 
onto $L.$
Clearly, ${\rm dim}(L)=d=[r].$ Let $u\in L, \|u\|=1$ be a unit eigenvector of $\Sigma_0$ corresponding to its top eigenvalue $\lambda$ (it is unique up to its sign).   
Then
\begin{align*}
\Sigma_0 = \Sigma_{u,\lambda,\mu} := \lambda (u\otimes u)+ \mu (P_L-u\otimes u)= (\lambda-\mu)(u\otimes u) + \mu P_L.
\end{align*}
Since $d=[r]\leq r$ and $\lambda\leq a,$ we have $\Sigma_0\in {\mathcal S}(a,r).$ 
Moreover, recall that $\lambda=\gamma_1 a, \mu=\gamma_2 a$ for $\gamma_1, \gamma_2\in (0,1), \gamma_1>\gamma_2.$ If $\Sigma$ is a covariance operator such ${\rm Im}(\Sigma)=L$ and 
$\|\Sigma-\Sigma_0\|<\delta,$ where $\delta\leq \kappa a$ with 
$
\kappa=(1-\gamma_1)\wedge (\gamma_1-\gamma_2)\wedge \gamma_2, 
$
then $\Sigma\in {\mathcal S}(a,r).$ 
This implies that 
\begin{align}
\label{bd_init_min_max}
&
\nonumber
\inf_{T}\sup_{\Sigma\in {\mathcal S}(a,r), \|\Sigma-\Sigma_0\|<\delta} {\mathbb E}_{\Sigma} (T(X_1,\dots, X_n)-f(\Sigma))^2
\\
&
\geq  \inf_{T}
\sup_{{\rm Im}(\Sigma)=L,\|\Sigma-\Sigma_0\|<\delta} {\mathbb E}_{\Sigma} (T(X_1,\dots, X_n)-f(\Sigma))^2.
\end{align}

Subspace $L$ will be fixed throughout the proof. It will be also convenient to fix an orthonormal basis 
$e_1,\dots, e_d$ of $L$ and, since $L\ni x\mapsto (\langle x, e_1\rangle, \dots, \langle x, e_d\rangle)\in {\mathbb R}^d$
is an isometry, to identify $L$ with ${\mathbb R}^d$ and its vectors with their coordinate representations.

We start with the following simple proposition.

\begin{proposition}
\label{prop_min_max_simple}
Suppose that $c\gamma_1\frac{a}{\sqrt{n}}<\delta \leq \kappa a \wedge 1$ for some sufficiently large constant $c>0.$ Let $U:=B(\Sigma_0, 2\delta).$ Then, for $s\geq 1,$ 
\begin{align*}
\sup_{\|f\|_{C^{s,a}(U)}\leq 1} \inf_{T} \sup_{\Sigma\in {\mathcal S}(a,r),\|\Sigma-\Sigma_0\|<\delta} \|T(X_1,\dots, X_n)-f(\Sigma)\|_{L_2({\mathbb P}_{\Sigma})} \gtrsim \frac{\gamma_1}{\sqrt{n}}.
\end{align*}
For $s<1,$
\begin{align*}
\sup_{\|f\|_{C^{s,a}(U)}\leq 1} \inf_{T} \sup_{\Sigma\in {\mathcal S}(a,r),\|\Sigma-\Sigma_0\|<\delta} \|T(X_1,\dots, X_n)-f(\Sigma)\|_{L_2({\mathbb P}_{\Sigma})} \gtrsim \gamma_1^s \Bigl(\frac{1}{\sqrt{n}}\Bigr)^s.
\end{align*}
\end{proposition}

\begin{proof}
Let 
\begin{align*}
f(\Sigma):=\frac{\langle (\Sigma-\Sigma_0) u,u\rangle}{a}.
\end{align*}
Clearly, this is a linear functional of $\Sigma,$ so its derivatives of order two and higher are equal to $0.$ 
Moreover, since $\kappa\leq 1/3,$ we have 
\begin{align*}
|f(\Sigma)|\leq \frac{\|\Sigma-\Sigma_0\|}{a}\leq \frac{2\delta}{a} \leq 2\kappa \leq 1, \Sigma\in U.
\end{align*}
Finally, $f$ is a Lipschitz functional with constant $1/a$ on set $U.$ Therefore, $\|f\|_{C^{s,a}(U)}\leq 1.$
Let $\Sigma_1= \lambda_1 (u\otimes u)+ \mu (P_L-u\otimes u)$ with $\lambda_1:= \gamma_1 a+ c\gamma_1\frac{a}{\sqrt{n}}.$
Note that, since $c\gamma_1\frac{a}{\sqrt{n}}< \delta\leq \kappa a\wedge 1,$ we have $\Sigma_1\in B(\Sigma_0,\delta)$ and $\Sigma_1\in {\mathcal S}(a,r).$ 
Note also that 
\begin{align*}
|f(\Sigma_1)-f(\Sigma_0)| = c\frac{\gamma_1}{\sqrt{n}}.
\end{align*}
To bound from below the risk of an arbitrary estimator $T(X_1,\dots, X_n)$ of $f(\Sigma),$ we will use ``the two hypotheses" method (see \cite{Tsybakov}, Theorem 2.1, Theorem 2.2).  
Using \eqref{bd_init_min_max},
we get
\begin{align}
\label{bd_init_min_max_xyz}
&
\nonumber
\inf_{T}\sup_{\Sigma\in {\mathcal S}(a,r), \|\Sigma-\Sigma_0\|<\delta} {\mathbb E}_{\Sigma} (T(X_1,\dots, X_n)-f(\Sigma))^2
\\
&
\geq  \inf_{T}
\Bigl({\mathbb E}_{\Sigma_0} (T(X_1,\dots, X_n)-f(\Sigma_0))^2 \bigvee
{\mathbb E}_{\Sigma_1} (T(X_1,\dots, X_n)-f(\Sigma_1))^2\Bigr).
\end{align}
Clearly, to minimize the right hand side with respect to $T,$ it is enough to consider the estimators $T(X_1,\dots, X_n)$ taking values in the interval $[f(\Sigma_0), f(\Sigma_1)].$
For such an estimator, we get 
\begin{align}
\label{bd_T_phi}
&
\nonumber
{\mathbb E}_{\Sigma_0} (T(X_1,\dots, X_n)-f(\Sigma_0))^2 \bigvee
{\mathbb E}_{\Sigma_1} (T(X_1,\dots, X_n)-f(\Sigma_1))^2
\\
&
\nonumber
\geq
{\mathbb E}_{\Sigma_0} (T(X_1,\dots, X_n)-f(\Sigma_0))^2 I\Bigl(|T(X_1,\dots, X_n) - f(\Sigma_0)|\geq \frac{c\gamma_1}{2\sqrt{n}}\Bigr)
\\
&
\nonumber
\ \ \ \ \bigvee
{\mathbb E}_{\Sigma_1} (T(X_1,\dots, X_n)-f(\Sigma_1))^2 I\Bigl(|T(X_1,\dots, X_n) - f(\Sigma_1)|> \frac{c\gamma_1}{2\sqrt{n}}\Bigr)
\\
&
\nonumber
\geq \frac{c^2\gamma_1^2}{4}\frac{1}{n} \Bigl({\mathbb P}_{\Sigma_0}\Bigl\{|T(X_1,\dots, X_n) - f(\Sigma_0)|\geq \frac{c\gamma_1}{2\sqrt{n}}\Bigr\}
\bigvee  {\mathbb P}_{\Sigma_1}\Bigl\{|T(X_1,\dots, X_n) - f(\Sigma_1)|> \frac{c\gamma_1}{2\sqrt{n}}\Bigr\}\Bigr)
\\
&
= \frac{c^2\gamma_1^2}{4}\frac{1}{n} \Bigl({\mathbb E}_{\Sigma_0} \phi(X_1,\dots, X_n) \bigvee {\mathbb E}_{\Sigma_1} (1-\phi(X_1,\dots, X_n))\Bigr),
\end{align}
where 
\begin{align*}
\phi(X_1,\dots, X_n) :=  I\Bigl(|T(X_1,\dots, X_n) - f(\Sigma_0)|\geq \frac{c\gamma_1}{2\sqrt{n}}\Bigr)
\end{align*}
could be viewed as a test for the hypothesis $\Sigma=\Sigma_0$ against the alternative $\Sigma=\Sigma_1.$

Next, we will control the KL-divergence between Gaussian distributions $N(0,\Sigma_0)$ and $N(0,\Sigma_1):$
\begin{align*}
K(N(0,\Sigma_0)\|N(0,\Sigma_1))= \frac{1}{2} \Bigl({\rm tr}(\Sigma_1^{-1} \Sigma_0)-d-\log {\rm det}(\Sigma_1^{-1} \Sigma_0)\Bigr),
\end{align*}
where it is assumed that $\Sigma_0, \Sigma_1$ are operators acting in space $L$ (where they are nonsingular). 
We have 
\begin{align*}
\Sigma_1^{-1} \Sigma_0 = \frac{\lambda}{\lambda_1} (u\otimes u) + (P_L-u\otimes u),
\end{align*}
implying that ${\rm tr}(\Sigma_1^{-1} \Sigma_0)-d= \frac{\lambda}{\lambda_1}-1$ and ${\rm det}(\Sigma_1^{-1} \Sigma_0)=\frac{\lambda}{\lambda_1}.$
Hence 
\begin{align*}
K(N(0,\Sigma_0)\|N(0,\Sigma_1))= \frac{\lambda}{\lambda_1}-1 - \log\frac{\lambda}{\lambda_1}.
\end{align*}
Since 
\begin{align*}
\frac{\lambda}{\lambda_1}-1 = \frac{\gamma_1 a}{\gamma_1a+c\gamma_1\frac{a}{\sqrt{n}}}-1 = \frac{1}{1+\frac{c}{\sqrt{n}}}-1= -\frac{\frac{c}{\sqrt{n}}}{1+\frac{c}{\sqrt{n}}}, 
\end{align*}
we get 
\begin{align*}
K(N(0,\Sigma_0)\|N(0,\Sigma_1))= -\frac{\frac{c}{\sqrt{n}}}{1+\frac{c}{\sqrt{n}}}  - \log\biggl(1-\frac{\frac{c}{\sqrt{n}}}{1+\frac{c}{\sqrt{n}}}\biggr) \asymp \frac{c^2}{n},
\end{align*}
provided that $\frac{c}{\sqrt{n}}<1/4.$ This implies that 
\begin{align*}
K(N(0,\Sigma_0)^{\otimes n}\|N(0,\Sigma_1)^{\otimes n})
=n K(N(0,\Sigma_0)\|N(0,\Sigma_1)) \asymp c^2.
\end{align*}
Thus, if $c^2$ is small enough, then by \cite{Tsybakov}, Theorem 2.1, Theorem 2.2(iii), we can conclude that 
\begin{align*}
{\mathbb E}_{\Sigma_0} \phi(X_1,\dots, X_n) \bigvee {\mathbb E}_{\Sigma_1} (1-\phi(X_1,\dots, X_n)) \gtrsim 1.
\end{align*}
Combining this with \eqref{bd_init_min_max_xyz} and \eqref{bd_T_phi} we get that
\begin{align*}
\inf_{T}\sup_{\Sigma\in {\mathcal S}(a,r), \|\Sigma-\Sigma_0\|<\delta} {\mathbb E}_{\Sigma} (T(X_1,\dots, X_n)-f(\Sigma))^2
\gtrsim \frac{\gamma_1^2}{n},
\end{align*}
which implies the first claim of the proposition.

The proof of the second claim is similar, but the functional $f$ is now defined as 
\begin{align*}
f(\Sigma):=\frac{|\langle (\Sigma-\Sigma_0) u,u\rangle|^s}{a^s}.
\end{align*}
\end{proof}

The proof of the next proposition is much more involved.

\begin{proposition}
\label{prop_min_max_hard}
Let $s>0.$
Suppose $\delta$ satisfies the condition 
\begin{align}
\label{cond_prop_4.2}
c_1 \sqrt{\gamma_1 \gamma_2} a \sqrt{\frac{r}{n}}\leq \delta \leq c_2(\gamma_1-\gamma_2) a \wedge 1
\end{align}
for a sufficiently large constant $c_1\geq 1$
and a sufficiently small constant $c_2\in (0,1].$ Let $U=B(\Sigma_0, 2\delta).$ Then 
\begin{align*}
\sup_{\|f\|_{C^{s,a}(U)}\leq 1} \inf_{T} \sup_{\Sigma\in {\mathcal S}(a,r), \|\Sigma-\Sigma_0\|<\delta} \Bigl\|T(X_1,\dots, X_n)-f(\Sigma)\Bigr\|_{L_2({\mathbb P}_{\Sigma})} \gtrsim 
\gamma_1^{s/2}\gamma_2^{s/2} \Bigl(\sqrt{\frac{r}{n}}\Bigr)^s.
\end{align*} 
\end{proposition}

\begin{proof}
It will be enough to prove the bound for a sufficiently large $r,$ say, $r\geq r_0,$ where $r_0$ is a numerical constant. 
If $r\leq r_0,$ then the bound easily follows from Proposition \ref{prop_min_max_simple} since, for $s\geq 1,$
\begin{align*}
\gamma_1^{s/2}\gamma_2^{s/2} \Bigl(\sqrt{\frac{r}{n}}\Bigr)^s \leq r_0^{s/2} \frac{\gamma_1}{\sqrt{n}} 
\end{align*}
and, for $s<1,$
\begin{align*}
\gamma_1^{s/2}\gamma_2^{s/2} \Bigl(\sqrt{\frac{r}{n}}\Bigr)^s \leq r_0^{s/2} \gamma_1^s \Bigl(\frac{1}{\sqrt{n}}\Bigr)^s.
\end{align*}
Note also that, under condition \eqref{cond_prop_4.2}, 
\begin{align}
\label{UP_r}
r\leq \frac{c_2^2}{c_1^2} \frac{(\gamma_1-\gamma_2)^2}{\gamma_1 \gamma_2} n \leq \frac{(\gamma_1-\gamma_2)^2}{\gamma_1 \gamma_2}n,
\end{align}
so, it is enough to prove the bound under assumption \eqref{UP_r}.

Recall that $d=[r].$ In what follows, we will use certain ``well separated" finite sets of unit vectors $\theta_{\omega}\in L$
parametrized by vertices $\omega$ of binary cube $\{-1,1\}^d.$ Recall that $u\in L$ is a unit vector and $\Sigma_0=\Sigma_{u,\lambda,\mu}.$  
Let $\eps \in (0,1].$ For $\omega\in \{-1,1\}^d,$ define 
\begin{align*}
t_{\omega} := \eps \frac{\omega}{\sqrt{d}} + \sqrt{1-\eps^2} u
\end{align*} 
(recall that vectors are replaced by their coordinate representations). 
Then 
\begin{align}
\label{t_omega_square}
\|t_{\omega}\|^2 = 1+ \frac{2\eps\sqrt{1-\eps^2}}{\sqrt{d}}\langle u, \omega\rangle.
\end{align}

Let 
\begin{align*}
h(\omega, \omega') := \sum_{j=1}^d I(\omega_j\neq \omega_j'), \omega, \omega'\in \{-1,1\}^d
\end{align*}
be the Hamming distance in the binary cube $\{-1,1\}^d.$
We will use the following simple lemma (which is a modification of Varshamov-Gilbert bound). 

\begin{lemma}
\label{pack_hamming}
There exists a subset $B\subset \{-1,1\}^d$ such that
\begin{enumerate}[(i)] 
\item ${\rm card}(B)\geq \frac{e^{d/8}}{2};$
\item $|\langle u,\omega\rangle| <2, \omega\in B;$
\item $h(\omega, \omega') \geq d/4,\ \omega, \omega'\in B, \omega\neq \omega'.$
\end{enumerate}
\end{lemma}

\begin{proof}
Assuming that ${\mathbb P}$ is a uniform probability distribution on the binary cube $\{-1,1\}^d,$
$\langle u, \omega\rangle$ is a subgaussian r.v. with parameter $1,$ so, we have 
\begin{align*}
{\mathbb P}\{\langle u, \omega\rangle \geq t\}\leq e^{-t^2/2}.
\end{align*}
Therefore,
\begin{align*}
2^{-d} {\rm card}\{\omega\in \{-1,1\}^d: |\langle u, \omega\rangle|\geq t\}\leq 2e^{-t^2/2}.
\end{align*}
For $t=2,$ we have $2e^{-t^2/2} < 1/2.$ This implies that 
\begin{align*}
{\rm card}\{\omega\in \{-1,1\}^d: |\langle u, \omega\rangle|<2\}>2^{d-1}.
\end{align*}
Note that the cardinality of any ball of radius $m$ with respect to the Hamming distance in the $d$-dimensional 
binary cube is equal to 
\begin{align*}
{d\choose {\leq m}} = \sum_{j=0}^m {d\choose j}.
\end{align*}
Let $B\subset \{\omega\in \{-1,1\}^d: |\langle u, \omega\rangle|<2\}$ be an $m$-separated subset w.r.t. the Hamming distance of 
maximal cardinality. Then $\{\omega\in \{-1,1\}^d: |\langle u, \omega\rangle|<2\}$ is covered by the union of balls of radius $m$ 
with centers in $B,$ implying that ${\rm card}(B) {d\choose {\leq m}}\geq 2^{d-1},$ or
\begin{align*}
{\rm card}(B) \frac{{d\choose {\leq m}}}{2^d}\geq 1/2.
\end{align*}
If $\nu\sim B(d, 1/2)$ denotes the binomial r.v. with parameters $d$ and $1/2,$ then the last inequality can be interpreted 
as 
\begin{align*} 
{\rm card}(B) {\mathbb P}\{\nu \leq m\} \geq 1/2.
\end{align*}
We will choose $m=d/4.$ Then, by Hoeffding's inequality,
$
{\mathbb P}\{\nu \leq d/4\} \leq e^{-d/8}. 
$
Thus, for the corresponding set $B,$ we have
$
{\rm card}(B) \geq \frac{e^{d/8}}{2}.
$
\end{proof}

Denote 
\begin{align*}
\theta_{\omega} := \frac{t_{\omega}}{\|t_{\omega}\|},\ \omega \in B. 
\end{align*}

\begin{lemma}
\label{prop_theta_omega}
Assume that $\frac{\eps}{\sqrt{d}}\leq 1/8.$ Then, for all $\omega, \omega'\in B,$
\begin{align*}
\frac{\eps}{\sqrt{d}} (\sqrt{h(\omega,\omega')} - 4)\leq \|\theta_{\omega}-\theta_{\omega'}\|\leq \frac{8\eps}{\sqrt{d}} \sqrt{h(\omega,\omega')}.
\end{align*}
Moreover, if $d$ is sufficiently large (namely, $d\geq 16^2$), then, for all $\omega, \omega'\in B, \omega\neq \omega',$
\begin{align*}
\frac{\eps}{2\sqrt{d}} \sqrt{h(\omega,\omega')}\leq \|\theta_{\omega}-\theta_{\omega'}\|\leq \frac{8\eps}{\sqrt{d}} \sqrt{h(\omega,\omega')},
\end{align*}
which implies that 
\begin{align*}
\frac{\eps}{4}\leq \|\theta_{\omega}-\theta_{\omega'}\| \leq 8\eps ,\ \omega, \omega'\in B, \omega\neq \omega'.
\end{align*}
\end{lemma}

\begin{proof}
Note that, in view of \eqref{t_omega_square} for all $\omega\in B,$ we have 
\begin{align*}
1-\frac{4\eps}{\sqrt{d}} \leq \|t_{\omega}\|^2 \leq 1+\frac{4\eps}{\sqrt{d}},
\end{align*}
which also implies that 
\begin{align}
\label{t_omega_omega'_1}
1-\frac{4\eps}{\sqrt{d}} \leq \|t_{\omega}\| \leq 1+\frac{4\eps}{\sqrt{d}},
\end{align}
and, under the assumption that $\eps/\sqrt{d}\leq 1/8,$ we have $1/2\leq \|t_{\omega}\|\leq 3/2.$

Since $t_{\omega}-t_{\omega'}= \eps \frac{\omega-\omega'}{\sqrt{d}},$ 
we have 
\begin{align}
\label{t_omega_hamming}
\|t_{\omega}-t_{\omega'}\| = \frac{2\eps}{\sqrt{d}} \sqrt{h(\omega,\omega')}.
\end{align}

Note also that
\begin{align*}
&
\|\theta_{\omega}-\theta_{\omega'}\|\leq \Bigl\|\frac{t_{\omega}}{\|t_{\omega}\|}- \frac{t_{\omega'}}{\|t_{\omega}\|}\Bigr\|
+ \Bigl\|\frac{t_{\omega'}}{\|t_{\omega}\|}- \frac{t_{\omega'}}{\|t_{\omega'}\|}\Bigr\|
\\
&
\leq \frac{\|t_{\omega}-t_{\omega'}\|}{\|t_{\omega}\|} + \|t_{\omega'}\|\Bigl|\frac{1}{\|t_{\omega}\|}- \frac{1}{\|t_{\omega'}\|}\Bigr|
\leq \frac{2\|t_{\omega}-t_{\omega'}\|}{\|t_{\omega}\|} \leq 4\|t_{\omega}-t_{\omega'}\|.
\end{align*}
On the other hand, it follows from \eqref{t_omega_omega'_1} that 
$
|\|t_{\omega}\|- \|t_{\omega'}\|| \leq \frac{8\eps}{\sqrt{d}},
$
and we have
\begin{align*}
\|t_{\omega}-t_{\omega'}\| \leq \|t_{\omega}\|\|\theta_{\omega}-\theta_{\omega'}\|+ \|\theta_{\omega'}\||\|t_{\omega}\|- \|t_{\omega'}\||
\leq 2\|\theta_{\omega}-\theta_{\omega'}\| + \frac{8\eps}{\sqrt{d}}.
\end{align*}
Thus, under the assumption $\frac{\eps}{\sqrt{d}}\leq 1/8,$ we get that for all $\omega, \omega'\in B,$
\begin{align*}
\frac{1}{2}\|t_{\omega}-t_{\omega'}\| -\frac{4\eps}{\sqrt{d}}\leq \|\theta_{\omega}-\theta_{\omega'}\|\leq 4 \|t_{\omega}-t_{\omega'}\|.
\end{align*}
Combining this with \eqref{t_omega_hamming} yields the first inequality of the lemma.
For the second inequality, note that, if $d\geq 16^2,$ then, in view of bound (iii) of Lemma \ref{pack_hamming},
$\sqrt{h(\omega, \omega')}\geq 8.$
\end{proof}

\begin{lemma}
\label{SIgma_theta_omega_u}
If $d\geq 8^2,$ $\eps\in (0,1]$ and $\frac{\eps}{\sqrt{d}}\leq 1/8,$ then, for all $\omega\in B,$ 
\begin{align*}
\|\Sigma_{\theta_{\omega},\lambda,\mu}-\Sigma_{u,\lambda,\mu}\| \leq 5\sqrt{2}|\lambda-\mu| \eps.
\end{align*}
\end{lemma}

\begin{proof}
Note that  
$
\|t_{\omega}-u\| \leq \eps + |1-\sqrt{1-\eps^2}| \leq 2\eps.
$
Using bound \eqref{t_omega_omega'_1}, it is easy to get that, under the assumptions $d\geq 8^2$ and $\frac{\eps}{\sqrt{d}}\leq 1/8,$ 
\begin{align*}
\|\theta_{\omega}-u\| \leq \frac{\|t_{\omega}-u\|}{\|t_{\omega}\|} + \frac{|\|t_{\omega}\|-1|}{\|t_{\omega}\|}
\leq 4\eps + \frac{8\eps}{\sqrt{d}}\leq 5\eps,\ \omega\in B. 
\end{align*}
Hence, for all $\omega\in B,$
\begin{align*}
\|\Sigma_{\theta_{\omega},\lambda,\mu}-\Sigma_{u,\lambda,\mu}\|&= |\lambda-\mu| \|\theta_{\omega}\otimes \theta_{\omega}-u\otimes u\|
\leq |\lambda-\mu| \|\theta_{\omega}\otimes \theta_{\omega}-u\otimes u\|_2 
\\
&
= |\lambda-\mu| \sqrt{2-2 \langle \theta_{\omega}, u\rangle^2} 
= |\lambda-\mu| \sqrt{2(1- \langle \theta_{\omega}, u\rangle)(1+\langle \theta_{\omega}, u\rangle)} 
\\
&
\leq |\lambda-\mu| \sqrt{4(1- \langle \theta_{\omega}, u\rangle)} = \sqrt{2}|\lambda-\mu| \|\theta_{\omega}-u\|
\\
&
\leq 5\sqrt{2}|\lambda-\mu| \eps.
\end{align*}
\end{proof}

We need the following simple lemma. 

\begin{lemma}
\label{KL_lem}
Let $u_1,u_2\in L$ with $\|u_1\|=1, \|u_2\|=1$ and let $\lambda>\mu>0.$
The following formula holds for the Kullback-Leibler distance between $N(0, \Sigma_{u_1,\lambda,\mu})$ and  $N(0, \Sigma_{u_2,\lambda,\mu}):$
\begin{align*}
K(N(0, \Sigma_{u_1,\lambda,\mu})\| N(0, \Sigma_{u_2,\lambda,\mu})) = \frac{(\lambda-\mu)(\mu^{-1}-\lambda^{-1})}{2}\|u_1\otimes u_1-u_2\otimes u_2\|_2^2.
\end{align*}
\end{lemma}

\begin{proof}
Indeed, taking into account that $\Sigma_{u_1,\lambda,\mu}$ and $\Sigma_{u_2,\lambda,\mu}$ have the same spectrum and 
using well known formulas for the KL-distance between two multivariate normal distributions, we have 
\begin{align*}
K(N(0, \Sigma_{u_1,\lambda,\mu})\| N(0, \Sigma_{u_2,\lambda,\mu}))= {\rm tr}(\Sigma_{u_1,\lambda,\mu}^{-1}(\Sigma_{u_2,\lambda,\mu}-\Sigma_{u_1,\lambda,\mu})),
\end{align*}
where $\Sigma_{u_1,\lambda,\mu}^{-1}$ is the inverse of operator $\Sigma_{u_1,\lambda,\mu}:L\mapsto L.$
Let $P_1= u_1\otimes u_1$ $P_2= u_2\otimes u_2,$ $P=P_L.$
Note that $\Sigma_{u_2,\lambda,\mu}-\Sigma_{u_1,\lambda,\mu}=(\lambda-\mu)(P_2-P_1)$
and 
\begin{align*}
&
\Sigma_{u_1,\lambda,\mu}^{-1}= \frac{1}{\lambda} P_1 + \frac{1}{\mu} (P-P_1)= \Bigl(\frac{1}{\lambda} -\frac{1}{\mu}\Bigr)P_1 +  \frac{1}{\mu}P.
\end{align*}
Therefore,
\begin{align*}
&
{\rm tr}(\Sigma_{u_1,\lambda,\mu}^{-1}(\Sigma_{u_2,\lambda,\mu}-\Sigma_{u_1,\lambda,\mu}))= (\lambda-\mu) {\rm tr}\Bigl(\Bigl(\Bigl(\frac{1}{\lambda} -\frac{1}{\mu}\Bigr)P_1 +  \frac{1}{\mu}P\Bigr)(P_2-P_1)\Bigr)
\\
&
=(\lambda-\mu)(\lambda^{-1}-\mu^{-1}) {\rm tr}(P_1 P_2- P_1^2) + (\lambda-\mu)\mu^{-1} {\rm tr}(P(P_2-P_1))
\\
&
=(\lambda-\mu)(\mu^{-1}-\lambda^{-1}) (1- \langle P_1,P_2\rangle) = \frac{(\lambda-\mu)(\mu^{-1}-\lambda^{-1})}{2}\|P_1-P_2\|_2^2,
\end{align*}
implying the claim.
\end{proof}

Let $X_1,\dots, X_n$ be i.i.d. $N(0, \Sigma_{\theta, \lambda, \mu}),$ where $\lambda>\mu>0$ are known numbers and $\theta\in L, \|\theta\|=1$
is an unknown parameter to be estimated based on observations $X_1,\dots, X_n.$ Further assume that $\theta \in \Theta_{\eps}:= \{\theta_{\omega}: \omega \in B\}.$
The following fact and its proof are straightforward modifications of well known minimax bounds on the error rate for principal eigenvector in spiked covariance model. 

\begin{lemma}
\label{min_max_theta}
Assume that 
\begin{align*}
16^2\leq d\leq (\lambda-\mu)(\mu^{-1}-\lambda^{-1})n.
\end{align*}
Also suppose that, for a sufficiently small constant $c\in (0,1],$ 
\begin{align}
\label{assume_eps_lambda_mu}
\eps^2 \leq \frac{c}{(\lambda-\mu)(\mu^{-1}-\lambda^{-1})} 
\frac{d}{n}.
\end{align}
Then 
\begin{align*}
\inf_{\hat \theta} \max_{\theta\in \Theta_{\eps}} {\mathbb E}_{\theta} \|\hat \theta-\theta\|^2 \gtrsim \eps^2.
\end{align*}
\end{lemma}  

\begin{proof}
First note that, by Lemma \ref{KL_lem} for all $\theta, \theta' \in \Theta_{\eps},$
\begin{align*}
K\Bigl(N(0, \Sigma_{\theta,\lambda,\mu})^{\otimes n}\| N(0, \Sigma_{\theta',\lambda,\mu})^{\otimes n}\Bigr) 
&
= n K\Bigl(N(0, \Sigma_{\theta,\lambda,\mu})\| N(0, \Sigma_{\theta',\lambda,\mu})\Bigr) 
\\
&
= \frac{(\lambda-\mu)(\mu^{-1}-\lambda^{-1})}{2} n \|\theta\otimes \theta-\theta'\otimes \theta'\|_2^2
\\
&
=  (\lambda-\mu)(\mu^{-1}-\lambda^{-1}) n (1-\langle \theta, \theta'\rangle^2).
\end{align*}
On the other hand, by the last bound of Lemma \ref{prop_theta_omega}, for all $\theta,\theta'\in \Theta_{\eps},$
\begin{align*}
\langle \theta, \theta'\rangle = \frac{2-\|\theta-\theta'\|^2}{2} \geq 1- \frac{8^2 \eps^2}{2},
\end{align*}
implying that 
\begin{align*}
1 - \langle \theta, \theta'\rangle^2 \leq 1- \Bigl(1- \frac{8^2 \eps^2}{2}\Bigr)^2 = \frac{8^2 \eps^2}{2} \Bigl(2-\frac{8^2 \eps^2}{2}\Bigr) 
\leq 8^2 \eps^2
\end{align*}
(note that, under conditions of the lemma, $8^2 \eps^2\leq 1$ provided that constant $c$ is small enough).
Therefore, for all $\theta,\theta'\in \Theta_{\eps},$
\begin{align*}
K\Bigl(N(0, \Sigma_{\theta,\lambda,\mu})^{\otimes n}\| N(0, \Sigma_{\theta',\lambda,\mu})^{\otimes n}\Bigr) 
\leq 8^2 (\lambda-\mu)(\mu^{-1}-\lambda^{-1}) n \eps^2.
\end{align*}
Note also that, by Lemma \ref{pack_hamming} (i) under the assumptions $d\geq 16^2$ and \eqref{assume_eps_lambda_mu}
with a sufficiently small constant $c,$ we have 
\begin{align*}
\frac{1}{10}\log {\rm card}(\Theta_{\eps}) &=\frac{1}{10}\log {\rm card}(B) \geq 
\frac{1}{10}(d/8 - \log 2) \geq 
8^2 (\lambda-\mu)(\mu^{-1}-\lambda^{-1}) n \eps^2
\\
&
\geq \max_{\theta,\theta'\in \Theta_{\eps}}K\Bigl(N(0, \Sigma_{\theta,\lambda,\mu})^{\otimes n}\| N(0, \Sigma_{\theta',\lambda,\mu})^{\otimes n}\Bigr).
\end{align*}
The claim of the lemma now follows by an application of a standard KL-lower bound based on many hypotheses (see, e.g., \cite{Tsybakov}, Theorem 2.5).
\end{proof}

Let $\Sigma_0:= \Sigma_{u,\lambda,\mu}$ and let $\delta<(\lambda-\mu)/2.$
Then, for any self-adjoint operator $\Sigma$ such that $\|\Sigma-\Sigma_0\|<\delta,$
the top eigenvalue $\lambda(\Sigma)=\|\Sigma\|$ of $\Sigma$ is simple and $|\lambda(\Sigma)-\lambda|<\delta.$
Moreover, if $\theta(\Sigma)$ is a unit eigenvector of $\Sigma$ corresponding to $\lambda(\Sigma),$  
then, by standard perturbation bounds (see, e.g., \cite{Koltchinskii_Lounici_bilinear}, bound (2.7) of Lemma 1) 
\begin{align}
\label{perturb_123}
\|\theta(\Sigma)\otimes \theta(\Sigma)-u\otimes u\| \leq \frac{4\delta}{\lambda-\mu}.
\end{align}
Since eigenvector $\theta(\Sigma)$ is defined up to its sign, we will assume that $\langle \theta(\Sigma),u\rangle\geq 0.$

Denote $P(\Sigma) :=\theta(\Sigma)\otimes \theta(\Sigma).$
Let $\bar \delta:=(\lambda-\mu)/8$ and $U:=B(\Sigma_0, 2\delta)$ with $\delta<\bar \delta/2.$

We will need the following analytic lemmas (we will only sketch their proofs, skipping the details). 

\begin{lemma}
\label{PSigma_smooth}
For all $k\geq 0,$ the mapping $U\ni\Sigma\mapsto P(\Sigma)$ is $k$ times Fr\'echet differentiable and 
$$\|P^{(k)}\|_{L_{\infty}(U)}\lesssim_k \bar \delta^{-k}.$$ 
Moreover, for all $\rho\in (0,1],$
$$
\|P^{(k)}\|_{{\rm Lip}_{\rho}(U)}\lesssim_{k,\rho} \bar \delta^{-k-\rho}.
$$
\end{lemma}

\begin{proof} 
First note that, by Riesz formula, the spectral projection $P(\Sigma), \Sigma\in U$ can be written as 
\begin{align*}
P(\Sigma) = -\frac{1}{2\pi i} \oint_{\gamma} R_{\Sigma}(z) dz,
\end{align*}
where $R_{\Sigma}(z):= (\Sigma - z I)^{-1}$ is the resolvent of $\Sigma$ and $\gamma:=\{z: |z-\lambda|=4\bar \delta\}$ is the circle centered at $\lambda$ of radius $4\bar \delta$ with a counterclockwise orientation. 
Using the formula 
\begin{align*}
R_{\Sigma+H}(z) = (I+ R_{\Sigma}(z)H)^{-1} R_{\Sigma}(z) = \sum_{j=0}^{\infty}(-1)^j (R_{\Sigma}(z) H)^{j} R_{\Sigma}(z)
\end{align*}
that holds for all $z\in \gamma$ and for all operators $H$ with small enough operator norm $\|H\|,$ it is not hard to show the differentiability of $P(\Sigma)$ and to derive the following 
formula for its $k$-th order Fr\'echet derivative:
\begin{align}
\label{D^kP}
 P^{(k)}(\Sigma)[H_1,\dots, H_k] = \frac{(-1)^{k+1}}{2\pi i} \sum_{\tau \in S_k}  \oint_{\gamma} R_{\Sigma}(z) H_{\tau(1)} R_{\Sigma}(z) \dots R_{\Sigma}(z) H_{\tau(k)}R_{\Sigma}(z)dz,
 \end{align} 
 where the summation is over all permutations $\tau$ of indices $1,\dots, k.$ Since $\|R_{\Sigma}(z)\|\lesssim \bar \delta^{-1}, z\in \gamma$ and the radius of $\gamma$ is $4\bar \delta,$
 it is easy to prove that 
 \begin{align*}
 \|P^{(k)}(\Sigma)\| \lesssim_k \bar \delta^{-k}, \Sigma \in U,
 \end{align*}
 which implies the desired bound on $\|P^{(k)}\|_{L_{\infty}(U)}.$ 
 Also, it is easy to see that, for $\|H_1\|< 2\delta, \|H_2\|< 2\delta,$ 
 \begin{align*}
 \|R_{\Sigma_0+H_1}(z)-R_{\Sigma_0+H_2}(z)\| \lesssim \frac{\|H_1-H_2\|}{\bar \delta^2}, z\in \gamma.
 \end{align*}
 Since $\|H_1-H_2\|\lesssim \bar \delta,$ we also have 
 \begin{align*}
 \|R_{\Sigma_0+H_1}(z)-R_{\Sigma_0+H_2}(z)\| \lesssim \frac{\|H_1-H_2\|^{\rho}}{\bar \delta^{1+\rho}}, z\in \gamma,
 \end{align*} 
 for $\rho \in (0,1],$ and, using formula \eqref{D^kP},
 it is easy to prove the desired bound on   
 $\|P^{(k)}\|_{{\rm Lip}_{\rho}(U)}.$
\end{proof}

\begin{remark}
\label{PSigma_smooth_R}
\normalfont
Note that Lemma \ref{PSigma_smooth} holds in a more general setting, for instance, when $\Sigma_0$ is an arbitrary covariance operator in a separable Hilbert space ${\mathbb H}$ with eigenvalues 
$\lambda_1=\dots=\lambda_l >\lambda_{l+1}\geq \lambda_{l+2}\geq \dots$ and with $\bar \delta= \frac{\lambda_l -\lambda_{l+1}}{8}.$ 
For $\delta < \bar \delta/2$ and for $U:=B(\Sigma_0, 2\delta),$ one can still define a $C^{\infty}$ projection valued function $P$ on $U$ such that $P(\Sigma)$ is the projection 
on the linear span of eigenvectors corresponding to the first $l$ eigenvalues of $\Sigma$ and the bounds on the derivatives of $P$ of Lemma \ref{PSigma_smooth} hold.
\end{remark}

\begin{lemma}
\label{Fcircg}
For all $k\geq 0,$
the mapping $U\ni\Sigma\mapsto \theta(\Sigma)$ 
is $k$ times Fr\'echet differentiable and 
$$\|\theta^{(k)}\|_{L_{\infty}(U)}\lesssim_k \bar \delta^{-k}.$$ 
Moreover, for all $\rho\in (0,1],$
$$\|\theta^{(k)}\|_{{\rm Lip}_{\rho}(U)}\lesssim_{k,\rho} \bar \delta^{-k-\rho}.$$
\end{lemma}

\begin{proof}
Let $P_0:=P(\Sigma_0)=u\otimes u.$
Note that bound \eqref{perturb_123} implies that $P(U)=\{P(\Sigma): \Sigma\in U\}\subset B(P_0,1/2).$
Consider now the following mapping $A\mapsto F(A)$ from the ball $B(P_0,1/2)$ into ${\mathbb H}:$
\begin{align*}
F(A):= \frac{Au}{\sqrt{\langle A u,u\rangle}}, A\in B(P_0,1/2).
\end{align*}
Note that, for all $A\in B(P_0,1/2),$ we have $\langle Au,u\rangle>1/2,$ so, the above mapping is well defined. Moreover, since 
$A\mapsto Au$ and $A\mapsto \langle A u,u\rangle$ are both linear mappings with operator norms bounded by $1$
and $\langle A u,u\rangle >1/2$ for $A\in B(P_0,1/2),$ it is easy to check that, for all $k\geq 1,$ the mapping $A\mapsto F(A)$ 
is $k$ times Fr\'echet differentiable and we have $\|F^{(k)}\|_{L_{\infty}(B(P_0,1/2))}\lesssim_k 1.$
Finally, note that for any unit vector $v\in {\mathbb H}$
with $v\otimes v \in B(P_0,1/2)$ and $\langle v,u\rangle\geq 0,$ we have $F(v\otimes v)=v.$ 
Therefore, $\theta(\Sigma) = F(P(\Sigma)), \Sigma \in B(\Sigma_0, 2\delta)$ and, using Fa\`a di Bruno type calculus (see, e.g., \cite{Koltchinskii_2021}, Section 3.1), it is easy 
to bound the Fr\'echet derivatives of this superposition, yielding the bounds of the lemma.
\end{proof}

We now turn to the main part of the argument. 
Our first goal is to construct a set of functionals $f_k\in C^s(U), k=1,\dots, d$ that are ``hard" to estimate. They will be defined as 
$$f_k(\Sigma)= (\gamma_1-\gamma_2)^s h_k(\theta(\Sigma)),$$ 
for properly chosen $h_k\in C^s({\mathbb H}).$ 
Each functional $h_k$ will be a sum of ``bumps" with disjoint supports around well separated vectors $\theta_{\omega}, \omega \in B.$
Namely, let $\varphi: {\mathbb R} \mapsto [0,1]$ be a $C^{\infty}$ function with support in $[-1,1]$ and with $\varphi (0)>0.$ Suppose that 
$\|\varphi\|_{C^s({\mathbb R})}\lesssim 1.$  Let $\phi(u):= \varphi(\|u\|^2), u\in {\mathbb H}.$ Then $\|\phi\|_{C^s({\mathbb H})}\lesssim 1.$
Define 
\begin{align*}
h_k(\theta) := \sum_{\omega \in B} \omega_k \eps^s \phi \Bigl(\frac{\theta-\theta_{\omega}}{c\eps}\Bigr), \theta \in {\mathbb H}, k=1,\dots, d,
\end{align*}
where $c\in (0, 1/16)$ is a small enough constant. The following facts are straightforward:

\begin{enumerate}[(i)]
\item each functional $\theta\mapsto \phi \Bigl(\frac{\theta-\theta_{\omega}}{c\eps}\Bigr)$ is supported in a ball 
of radius $c\eps$ centered at $\theta_{\omega}.$

\item  Functionals $\phi \Bigl(\frac{\theta-\theta_{\omega}}{c\eps}\Bigr), \omega\in B$ have disjoint supports. Moreover,
the supports of any two of those functionals are separated by distance $\geq \eps/8.$  

\item For all $\omega\in B,$ 
\begin{align*}
\Bigl\| \eps^s \phi \Bigl(\frac{\cdot-\theta_{\omega}}{c\eps}\Bigr)\Bigr\|_{C^s({\mathbb H})}\lesssim 1.
\end{align*}
\end{enumerate}

The next proposition easily follows from the above properties. 

\begin{proposition}
\label{prop_coding}
\begin{enumerate}[(i)]
\item For all $k=1,\dots, d,$ 
\begin{align*}
\|h_k\|_{C^s({\mathbb H})}\lesssim 1.
\end{align*}
\item For all $k=1,\dots, d$ and all $\omega\in B,$
\begin{align*}
h_k(\theta_{\omega})= \omega_k \eps^s \varphi(0).
\end{align*}
\end{enumerate}
\end{proposition}

The second claim of Proposition \ref{prop_coding} means that the values of functionals $h_k, k=1,\dots, d$
at the vectors $\theta_{\omega}, \omega\in B$ could be used to recover the vectors $\omega\in B$ and, hence, 
the vectors $\theta_{\omega}$ themselves. This property plays a crucial role in our argument (its idea goes back to \cite{Nemirovski_1990}).

As in the statement of the theorem, we fix $\lambda=\gamma_1 a$ and $\mu=\gamma_2 a$ for some numerical constants $\gamma_1>\gamma_2>0$ and denote $\Sigma_{\theta}:= \Sigma_{\theta, \lambda, \mu}.$
Recall that $\Sigma_0= \Sigma_u$ and also recall the statement and the notations of  Lemma \ref{Fcircg}, in particular, that 
$$
\bar \delta =\frac{\lambda-\mu}{8}= \frac{\gamma_1-\gamma_2}{8} a.
$$
Define functionals  
\begin{align*}
f_k(\Sigma):=(\gamma_1-\gamma_2)^s h_k(\theta(\Sigma)), k=1,\dots, d, \Sigma\in U=B(\Sigma_0, 2\delta).
\end{align*}
Let $s=m+\rho,$ $m\geq 0,$ $\rho\in (0,1].$ Using the fact that $\|h_k\|_{C^s({\mathbb H})}\lesssim 1$ and the bounds of Lemma \ref{Fcircg}, it is not hard to show that 
\begin{align*}
\|(h_k\circ \theta)^{(j)}\|_{L_{\infty}(U)} \lesssim_j \bar \delta^{-j}, j=1,\dots, m\ {\rm and}\  \|(h_k\circ \theta)^{(m)}\|_{{\rm Lip}_{\rho}(U)} \lesssim_{m,\rho} \bar \delta^{-s}
\end{align*}
(the proof is again based on Fa\`a di Bruno type calculus, see, e.g., \cite{Koltchinskii_2021}, Section 3.1).
Therefore, 
\begin{align*}
&
\|f_k^{(j)}\|_{L_{\infty}(U)} \lesssim_j (\gamma_1-\gamma_2)^s\bar \delta^{-j}\lesssim_j (\gamma_1-\gamma_2)^{s-j}a^{-j}\lesssim_j a^{-j}, j=1,\dots, m\ {\rm and}\  
\\
&
\|f_k^{(m)}\|_{{\rm Lip}_{\rho}(U)} \lesssim_{m,\rho} (\gamma_1-\gamma_2)^s \bar \delta^{-s}\lesssim_s a^{-s}.
\end{align*}
It easily follows that $\|f_k\|_{C^{s,a}(U)} \lesssim_s 1$ and, moreover, by using $c' \varphi$ with a small enough $c'>0$ instead of $\varphi,$ we can achieve the bounds 
\begin{align*}
\|f_k\|_{C^{s,a}(U)} \leq 1, k=1,\dots, d. 
\end{align*}
Note also that 
\begin{align*}
f_k(\Sigma_{\theta_{\omega}})=(\gamma_1-\gamma_2)^s h_k(\theta_{\omega})=\varphi(0) (\gamma_1-\gamma_2)^s \eps^s \omega_k, k=1,\dots, d.
\end{align*}
The last relationship clearly means that the values $f_k(\Sigma_{\theta_{\omega}}), k=1,\dots, d$ could be used to recover vectors $\omega\in B$
and $\theta_{\omega}, \omega\in B.$
Moreover, define the following distance  
\begin{align*}
\tau(\omega,\omega') := \biggl(\frac{1}{d}\sum_{k=1}^d (f_k(\Sigma_{\theta_{\omega}})- f_k(\Sigma_{\theta_{\omega'}}))^2\biggr)^{1/2},\ \omega, \omega'\in B.
\end{align*}
Then
\begin{align*}
\tau (\omega, \omega') = \frac{\varphi(0) (\gamma_1-\gamma_2)^s \eps^s}{\sqrt{d}} \sqrt{h(\omega, \omega')},\ \omega, \omega'\in B.
\end{align*}
The right hand side of the identity above can be now used to extend $\tau$ to a distance on the whole binary cube $\{-1,1\}^d.$

By the second bound of Lemma \ref{prop_theta_omega}, we easily get the following statement:

\begin{lemma}
\label{tau_hamm}
Suppose $d\geq 16^2$ and $\frac{\eps}{\sqrt{d}}\leq 1/8.$
Then, for all $\omega, \omega'\in B,$
\begin{align*}
\frac{1}{8}\varphi(0) (\gamma_1-\gamma_2)^s \eps^{s-1}\|\theta_{\omega}-\theta_{\omega'}\|
\leq 
\tau (\omega, \omega')\leq 2\varphi(0) (\gamma_1-\gamma_2)^s \eps^{s-1}\|\theta_{\omega}-\theta_{\omega'}\|.
\end{align*}
\end{lemma}

Consider now arbitrary estimators $T_k(X_1,\dots, X_n)$ of functionals $f_k(\Sigma).$ For $k=1,\dots, d,$ define 
$\tilde \omega_k:= {\rm sign}(T_k(X_1,\dots, X_n)),$ $\tilde \omega:=(\tilde \omega_1, \dots, \tilde \omega_d)$ 
and 
\begin{align*}
\tilde T_k(X_1,\dots, X_n) := \varphi(0) (\gamma_1-\gamma_2)^s \eps^s \tilde \omega_k.
\end{align*}
Then, for all $\omega\in B$ and $k=1,\dots, d,$
\begin{align}
\label{T_k_tilde_T_k}
{\mathbb E}_{\Sigma_{\theta_{\omega}}} (\tilde T_k(X_1,\dots, X_n)-f_k(\Sigma_{\theta_{\omega}}))^2 
\leq 4{\mathbb E}_{\Sigma_{\theta_{\omega}}} (T_k(X_1,\dots, X_n)-f_k(\Sigma_{\theta_{\omega}}))^2.
\end{align}
Indeed, recall that $f_k(\Sigma_{\theta_{\omega}})= \varphi(0) (\gamma_1-\gamma_2)^s \eps^s \omega_k$ and, without loss of generality, assume that $\omega_k=+1.$
Then $\tilde T_k(X_1,\dots, X_n)-f_k(\Sigma_{\theta_{\omega}})=0$ when $T_k(X_1,\dots, X_n)\geq 0$ and otherwise 
\begin{align*}
|\tilde T_k(X_1,\dots, X_n)-f_k(\Sigma_{\theta_{\omega}})|\leq 2|T_k(X_1,\dots, X_n)-f_k(\Sigma_{\theta_{\omega}})|,
\end{align*}
which implies the claim.

Define also $\hat \omega := {\rm argmin}_{\omega\in B}h(\tilde \omega, \omega)$ and $\hat \theta= \theta_{\hat \omega}.$
By Lemma \ref{tau_hamm}, for all $\omega\in B,$
\begin{align*}
\|\hat \theta-\theta_{\omega}\| &\leq \frac{8}{\varphi(0)(\gamma_1-\gamma_2)^s \eps^{s-1}} \tau(\hat \omega, \omega)
= \frac{8}{\varphi(0) (\gamma_1-\gamma_2)^s \eps^{s-1}} \frac{\varphi(0)(\gamma_1-\gamma_2)^s \eps^s}{\sqrt{d}}\sqrt{h(\hat \omega, \omega)}
\\
&
\leq  \frac{8}{\varphi(0) (\gamma_1-\gamma_2)^s \eps^{s-1}} \frac{\varphi(0) (\gamma_1-\gamma_2)^s\eps^s}{\sqrt{d}}\sqrt{h(\tilde \omega, \omega)+h(\tilde \omega, \hat \omega)}
\\
&
\leq  \frac{8}{\varphi(0) (\gamma_1-\gamma_2)^s\eps^{s-1}} \frac{\varphi(0) (\gamma_1-\gamma_2)^s \eps^s}{\sqrt{d}}\sqrt{2h(\tilde \omega, \omega)}
\\
&
=\frac{8\sqrt{2}}{\varphi(0) (\gamma_1-\gamma_2)^s\eps^{s-1}} \tau(\tilde \omega, \omega).
\end{align*}
Note also that 
\begin{align*}
\tau^2(\tilde \omega, \omega) = \frac{1}{4d} \sum_{k=1}^d (\tilde T_k(X_1,\dots, X_n)- f_k(\Sigma_{\theta_\omega}))^2.
\end{align*}
Therefore,
\begin{align*}
{\mathbb E}_{\theta_{\omega}}\|\hat \theta-\theta_{\omega}\|^2 &\leq \frac{2\times 8^2}{\varphi^2(0) (\gamma_1-\gamma_2)^{2s}\eps^{2(s-1)}}
{\mathbb E}_{\theta_{\omega}}\tau^2(\tilde \omega, \omega)
\\
&
=\frac{8^2/2}{\varphi^2(0) (\gamma_1-\gamma_2)^{2s}\eps^{2(s-1)}}\frac{1}{d}\sum_{k=1}^d {\mathbb E}_{\Sigma_{\theta_{\omega}}} \Bigl(\tilde T_k(X_1,\dots, X_n)-f_k(\Sigma_{\theta_{\omega}})\Bigr)^2
\\
&
\leq 
\frac{2\times 8^2}{\varphi^2(0) (\gamma_1-\gamma_2)^{2s}\eps^{2(s-1)}}\frac{1}{d}\sum_{k=1}^d {\mathbb E}_{\Sigma_{\theta_{\omega}}} \Bigl(T_k(X_1,\dots, X_n)-f_k(\Sigma_{\theta_{\omega}})\Bigr)^2,
\end{align*}
where we also used bound \eqref{T_k_tilde_T_k}. This implies that 
\begin{align*}
\max_{\omega\in B}{\mathbb E}_{\theta_{\omega}}\|\hat \theta-\theta_{\omega}\|^2
\leq 
\frac{2\times 8^2}{\varphi^2(0) (\gamma_1-\gamma_2)^{2s}\eps^{2(s-1)}}\frac{1}{d} \sum_{k=1}^d \max_{\omega\in B}{\mathbb E}_{\Sigma_{\theta_{\omega}}} \Bigl(T_k(X_1,\dots, X_n)-f_k(\Sigma_{\theta_{\omega}})\Bigr)^2.
\end{align*}
Assuming that 
\begin{align*}
\max_{1\leq k\leq d} \inf_{T} \max_{\omega\in B} {\mathbb E}_{\Sigma_{\theta_{\omega}}} \Bigl(T(X_1,\dots, X_n)-f_k(\Sigma_{\theta_{\omega}})\Bigr)^2 <\tau^2,
\end{align*}
one can find, for each $k=1,\dots, d,$ an estimator $T_k(X_1,\dots, X_n)$ such that 
\begin{align*}
\max_{\omega\in B} {\mathbb E}_{\Sigma_{\theta_{\omega}}} \Bigl(T_k(X_1,\dots, X_n)-f_k(\Sigma_{\theta_{\omega}})\Bigr)^2 <\tau^2,
\end{align*} 
which implies the existence of $\hat \theta$ such that 
\begin{align*}
\max_{\omega\in B}{\mathbb E}_{\theta_{\omega}}\|\hat \theta-\theta_{\omega}\|^2
\leq 
\frac{(2\times 8^2) \tau^2}{\varphi^2(0) (\gamma_1-\gamma_2)^{2s}\eps^{2(s-1)}}.
\end{align*}
Note that, for $\lambda=\gamma_1 a, \mu = \gamma_2 a,$ $(\lambda-\mu)(\mu^{-1}-\lambda^{-1})=(\gamma_1-\gamma_2)(\gamma_2^{-1}-\gamma_1^{-1}),$ so condition \eqref{assume_eps_lambda_mu} of Lemma \ref{min_max_theta} becomes 
\begin{align}
\label{eps_gamma1_gamma2}
\eps \leq \frac{c}{\sqrt{(\gamma_1-\gamma_2)(\gamma_2^{-1}-\gamma_1^{-1})}}\sqrt{\frac{d}{n}}:=\eps(\gamma_1,\gamma_2)
\end{align}
with a sufficiently small constant $c>0.$ Recall also that, by \eqref{UP_r}, we have 
$d=[r] \leq \frac{(\gamma_1-\gamma_2)^2}{\gamma_1 \gamma_2} n.$
In view of Lemma \ref{min_max_theta}, it follows that 
\begin{align*}
\frac{\tau^2}{\varphi^2(0) (\gamma_1-\gamma_2)^{2s}\eps^{2(s-1)}}\gtrsim \eps^2
\end{align*}
for an arbitrary $\eps$ satisfying the condition \eqref{eps_gamma1_gamma2}.
We set $\eps:= \eps(\gamma_1,\gamma_2)$ and conclude that 
\begin{align*}
\tau^2  \gtrsim (\gamma_1-\gamma_2)^{2s} \eps^{2s} \gtrsim \gamma_1^s \gamma_2^s \Bigl(\frac{d}{n}\Bigr)^s.
\end{align*}
In other words,
\begin{align*}
\max_{1\leq k\leq d} \inf_{T} \max_{\omega\in B} {\mathbb E}_{\Sigma_{\theta_{\omega}}} \Bigl(T(X_1,\dots, X_n)-f_k(\Sigma_{\theta_{\omega}})\Bigr)^2 \gtrsim \gamma_1^s \gamma_2^s  \Bigl(\frac{d}{n}\Bigr)^s.
\end{align*} 
Finally, by Lemma \ref{SIgma_theta_omega_u} for all $\omega\in B,$ 
$
\|\Sigma_{\theta_{\omega}}-\Sigma_0\|\leq 5\sqrt{2}(\gamma_1-\gamma_2) a \eps.
$
For $\eps=\eps(\gamma_1,\gamma_2),$ we get $\|\Sigma_{\theta_{\omega}}-\Sigma_0\|<\delta$
provided that 
$\delta> 5\sqrt{2} c \sqrt{\gamma_1 \gamma_2} a\sqrt{\frac{d}{n}}.$
Taking also into account that $\|f_k\|_{C^{s,a}(U)}\leq 1, k=1,\dots, d,$ we can conclude that 
\begin{align*}
\sup_{\|f\|_{C^{s,a}(U)}\leq 1} \inf_{T}\sup_{{\rm Im}(\Sigma)=L,\|\Sigma-\Sigma_0\|<\delta}{\mathbb E}_{\Sigma} \Bigl(T(X_1,\dots, X_n)-f(\Sigma)\Bigr)^2 \gtrsim \gamma_1^s \gamma_2^s \Bigl(\frac{d}{n}\Bigr)^s.
\end{align*} 
It remains to combine the last bound with \eqref{bd_init_min_max} and to recall that $d=[r],$ to get 
\begin{align*}
\sup_{\|f\|_{C^{s,a}(U)}\leq 1} \inf_{T}\sup_{\Sigma\in {\mathcal S}(a,r), \|\Sigma-\Sigma_0\|<\delta} \Bigl\|T(X_1,\dots, X_n)-f(\Sigma)\Bigr\|_{L_2({\mathbb P}_{\Sigma})} \gtrsim 
\gamma_1^{s/2} \gamma_2^{s/2} \Bigl(\sqrt{\frac{r}{n}}\Bigr)^s.
\end{align*} 
\end{proof}

The claim of Theorem \ref{loc_min_max} immediately follows by combining propositions \ref{prop_min_max_simple} and \ref{prop_min_max_hard}.

\begin{acks}[Acknowledgments]
The author is very thankful to the anonymous referees for their helpful comments and suggestions. 
\end{acks}

\begin{funding}
Supported in part by NSF Grant DMS-2113121.
\end{funding}

\end{document}